%% file: nvbgrading.tex
\documentclass{amsart}
\pdfoutput=1

\newcounter{dummy}
\makeatletter
\newcommand\myitem[1][]{\item[#1]\refstepcounter{dummy}\def\@currentlabel{#1}}
\makeatother

\usepackage[all,hyperref,numberbysection]{bi-discrete}

\usepackage[backend=biber,maxbibnames=5,maxalphanames=5,style=alphabetic,bibencoding=utf8,giveninits,url=false,isbn=false,doi=true]{biblatex}
\AtBeginBibliography{\small}

\usepackage{placeins}

\usepackage{tikz}
\usetikzlibrary{positioning,calc,arrows,shapes.geometric,patterns,decorations.pathreplacing}
\tikzstyle{arrow}=[draw]
\tikzset{
  cross/.pic = {
    \draw[rotate = 45] (-#1,0) -- (#1,0);
    \draw[rotate = 45] (0,-#1) -- (0, #1);
  }
}

\usepackage{mathtools}
\usepackage[ruled,vlined,norelsize,noend]{algorithm2e}

\providecommand{\Rd}{\setR^d}
\providecommand{\front}{\texttt{front}}
\providecommand{\back}{\texttt{back}}

\providecommand{\gen}{{\tt gen}}
\providecommand{\type}{{\tt type}}
\providecommand{\level}{{\tt lvl}}
\providecommand{\gensharp}{{\gen^\sharp}}
\providecommand{\levelsharp}{{\level^\sharp}}
\providecommand{\typesharp}{{\type^\sharp}}

\providecommand{\simplex}[1]{\llbracket\,{#1}\,\rrbracket}
\providecommand{\leveljump}{\mid}

\providecommand{\head}{{\tt head}}
\providecommand{\tail}{{\tt tail}}
\providecommand{\tria}{\mathcal{T}}
\providecommand{\forest}{\mathcal{F}}
\providecommand{\Bisec}{{\tt Bisec}}
\providecommand{\BisecT}{\Bisec(\tria_0)}
\providecommand{\mtree}{\mathbb{T}}
\providecommand{\vertices}{\mathcal{V}}
\providecommand{\edges}{\mathcal{E}}
\providecommand{\children}{{\tt children}}
\providecommand{\parent}{{\tt parent}}
\providecommand{\hprod}{\odot}
\providecommand{\bse}{\texttt{bse}}
\providecommand{\bsv}{\textup{mid}}
\providecommand{\genconst}{\Gamma}
\providecommand{\patchbd}{\partial \omega_m(v)}
\providecommand{\doublebar}[1]{\Bar{\Bar{#1}}}
\renewcommand{\jump}{\mathcal{J}}

\usepackage{tikz-3dplot}

\makeatletter
\@namedef{subjclassname@2020}{%
	\textup{2020} Mathematics Subject Classification}
\makeatother

\begin{document}

\author[L.\ Diening]{Lars Diening}
\author[J.\ Storn]{Johannes Storn}
\author[T.\ Tscherpel]{Tabea Tscherpel}

\address[L.\ Diening, J.\ Storn]{Department of Mathematics, Bielefeld University, Postfach 10 01 31, 33501 Bielefeld, Germany}
\email{lars.diening@uni-bielefeld.de}
\email{jstorn@math.uni-bielefeld.de}
\address[T.\ Tscherpel]{Department of Mathematics, Technische Universit\"at Darmstadt, Dolivostraße 15, 64293 Darmstadt, Germany}
\email{tscherpel@mathematik.tu-darmstadt.de}
\thanks{The work of the authors was supported by the Deutsche Forschungsgemeinschaft (DFG, German
Research Foundation) – SFB 1283/2 2021 – 317210226.}


\title{Grading of Triangulations Generated by Bisection}

\begin{abstract}  
   For triangulations generated by the adaptive bisection algorithm by Maubach and Traxler we prove existence of a regularized mesh function with grading two. 
   This sharpens previous results in two dimensions for the newest vertex bisection and generalizes them to arbitrary dimensions. 
   In combination with \cite{DieningStornTscherpel21} this yields $H^1$-stability of the $L^2$-projection onto Lagrange finite element spaces for all polynomial degrees and dimensions smaller than seven. 
\end{abstract}
\maketitle
%
\section{Introduction}
Adaptive mesh refinement is of uttermost importance for efficient approximation of partial differential equations that exhibit singular solutions or to resolve geometric features of solutions. 
For elliptic problems such schemes are well understood and even allow for optimal rates of convergence~\cite{CarstensenFeischlPagePraetorius14,DieningKreuzerStevenson16}.
Much less is known for parabolic problems. 
A major difficulty is the need for Sobolev stability of the $L^2$-projection onto finite element spaces. 
Such stability is crucial to avoid an unnatural coupling of spacial mesh size and time step size for time-stepping schemes, see for example \cite{BreitDieningStornWichmann21}. 
Indeed, for time-discrete schemes quasi-optimality is equivalent to Sobolev stability of the~$L^2$-projection, cf.~\cite{TV.2016}. 
This observation extends to simultaneous space-time methods as in~\cite{StevensonWesterdiep20}. 

The need for stability of the $L^2$-projection has inspired strong scientific interest starting with contributions by Crouzeix, Thom\'ee~\cite{CrouzeixThomee87} and by Bramble, Pasciak, Steinbach~\cite{BramblePasciakSteinbach02}. 
A key observation is the relation between stability and mesh grading. 
The latter quantifies the local change of the mesh sizes in the underlying triangulation. 
In~\cite{Carstensen02} Carstensen generalizes the concept of mesh grading to prove an upper bound on the grading of meshes generated by adaptive refinement schemes applied to a fixed initial triangulation~$\tria_0$. 
Refinement strategies with known upper bounds on the grading include the red-green-blue refinement~\cite{C.2004}, the red-green refinement, and the newest vertex bisection~\cite{GHS.2016}, but all results to date are restricted to triangulations of two dimensional domains. 

This work presents the first result for higher dimensions~$d\geq 2$. 
More precisely, we investigate the grading of meshes generated by the (equivalent) algorithms due to Maubach~\cite{Maubach95} and Traxler~\cite{Traxler97} (see Algorithms~\ref{algo:Maubach} and~\ref{algo:Traxler}), which we refer to as~\Bisec{}. 
Under suitable assumptions on the initial triangulation these schemes generalize the newest vertex bisection by Mitchell~\cite{Mitchell91} for two dimensional domains and the refinement strategy by Kossaczk\'y~\cite{Kossaczky94} for three dimensional domains to any dimension.  
Due to its beneficial properties such as the preservation of shape regularity and the closure estimate~\cite[Eqn.~2.9]{CarstensenFeischlPagePraetorius14} proved in~\cite{BinevDahmenDeVore04, Stevenson08} 
this refinement strategy is of particular interest. 

Let $\BisecT$ denote the set of all (conforming) triangulations generated by successive application of the~$\Bisec{}$ routine to an initial triangulation~$\tria_0$ (see Section~\ref{sec:bisection-algorithm-1} below for a more precise definition).
The resulting triangulations are shape-regular, that is, the~$d$-dimensional measure~$\abs{T}$ of any simplex~$T\in \tria\in \BisecT$ is equivalent to its diameter~$\diameter(T)$ to the power $d$ in the sense that~$\abs{T} \eqsim \textup{diam}(T)^d$ with equivalence constants depending only on the shape regularity of~$\tria_0$ and on~$d$. 
\begin{definition}
	\label{def:grading}
For a triangulation~$\tria$ we consider a function~$h_{\tria}\colon \tria \to (0,\infty)$. 
	\begin{enumerate}
		\item \label{itm:meshfct}
		We call~$h_{\tria}$ a \emph{mesh size function} with equivalence constants~$c_1, c_2>0$ if 
		\begin{align*} 
			c_1 \diameter(T) &\leq h_\tria(T) \leq c_2 \diameter(T) \qquad\text{for all }T\in \tria. 
		\end{align*}
		\item \label{itm:grading}
		We say that~$h_{\tria}$ has \emph{grading}~$\gamma\geq 1$, if 
		\begin{align*}
			h_\tria(T) \leq \gamma \, h_\tria(T') \qquad \text{ for all } T,T' \in \tria \text{ with } T \cap T' \neq \emptyset. 
		\end{align*}
	\end{enumerate} 
	A family of triangulations has grading~$\gamma\geq 1$ with equivalence constants~$c_1, c_2 >0$ if each~$\tria$ in the family has a mesh size function~$h_\tria$ with grading~$\gamma$ and equivalence constants~$c_1$ and~$c_2$ independently of $\tria$. 
\end{definition}
Note that by definition quasi-uniform triangulations have grading~$\gamma = 1$. 
The main result in this work is that~$\gamma \leq 2$ for any mesh resulting from the adaptive refinement of a fixed initial triangulation $\tria_0$ for arbitrary dimension. This bound is optimal.  
See Section~\ref{sec:summary} for the summary of all arguments that constitute the proof.  
\begin{theorem}[Main result]
	\label{thm:main-grading2}
	Assume that~$\tria_0$ is an initial triangulation which is colored (see Assumption~\ref{ass:initial-triangulation}) or that satisfies the matching neighbor condition by Stevenson (see Section~\ref{sec:stevenson}). 
	All triangulations $\tria\in \BisecT$ have grading~$\gamma \leq 2$ with uniformly bounded equivalence constants~$c_1$ and~$c_2$ depending solely on~$d$ and on the maximal number of simplices in initial vertex patches in $\tria_0$. 
\end{theorem}
Notice that the upper bound~$\gamma \leq 2$ is sharp, cf.~Remark~\ref{rem:sharpGradingEst}. 
In combination with the results in~\cite{DieningStornTscherpel21} the grading estimate in Theorem~\ref{thm:main-grading2} ensures that
for any dimension~$2 \leq d\leq 6$ and any~$\tria$ generated by $\Bisec{}$ the $L^2$-projection mapping to Lagrange finite element spaces~$\mathcal{L}^1_k(\tria)$ of polynomial degree $k\in \mathbb{N}$ is stable in~$H^1(\Omega)$. 
Previously, this was available only for~$d = 2$ and~$k \leq 12$ with stability constant depending additionally on the number~$\# \vertices(\tria_0)$ of initial vertices, cf.~\cite{BanYse14, GHS.2016}.  
More precisely, we have the following result.
\begin{theorem}[Stability of $L^2$-projection]
	\label{thm:L2-stab}
	Let~$\Omega \subset \Rd$ be a bounded, polyhedral domain with~$d \geq 2$.
	Suppose that $\tria_0$ is an initial triangulation of $\Omega$ satisfying the assumptions of Theorem~\ref{thm:main-grading2} and let~$\BisecT$ be the family of triangulations generated by~\Bisec{}.  
	For any triangulation~$\tria \in \BisecT$ and any polynomial degree~$k \in \mathbb{N}$ let~$\Pi_2\colon L^1(\Omega) \to \mathcal{L}^1_k(\tria)$ be the $L^2(\Omega)$-orthogonal projection. 
Let~$p\in [1,\infty]$ and let~$s \in \lbrace 0,1\rbrace$ be such that
\begin{align*}
2^{s+\left|\frac{1}{2}-\frac{1}{p} \right|} \leq \frac{\sqrt{2k+d} + \sqrt{k}}{\sqrt{2k+d} - \sqrt{k}}. 
\end{align*}
Then there exists a constant~$C < \infty$ depending only on~$d$, the shape regularity of~$\tria_0$, and the constants~$c_1,c_2$ in Theorem~\ref{thm:main-grading2} such that
\begin{align*}
\lVert \nabla^s \Pi_2 u \rVert_{L^p(\Omega)} \leq C \lVert \nabla^s u \rVert_{L^p(\Omega)}\qquad\text{for all }u\in W^{s,p}(\Omega).
\end{align*}
This result extends to zero boundary traces,  see~\cite{DieningStornTscherpel21} for details.  
\end{theorem}
In the special case~$d = 2$ Theorem~\ref{thm:main-grading2} is due to Gaspoz, Heine, and Siebert in~\cite{GHS.2016}.  
Their proof is based on a case by case investigation of constellations resulting from local adaptive mesh refinement and does not treat initial vertices with fine-tuned arguments. 
For this reason in their result the constant in Theorem~\ref{thm:main-grading2} depends on the number of vertices in the initial triangulation, which is a massive overestimation. 
Moreover, their arguments cannot be generalized to higher dimensions because the number of vertices on initial hyperfaces with high valence can be unbounded in dimensions~$d\geq 3$.   
This makes the situation in higher dimensions much more intricate and requires new ideas. 
To face these challenges we employ the following tools:
\begin{itemize}
	\item 
	We reformulate the bisection algorithm by Maubach and Traxler using a \emph{generation structure} of vertices. 
	This gives rise to a \emph{bisection routine for sub\-simplices} (Section~\ref{sec:reform-using-gener}) and allows to conclude fine properties of the algorithm (Section~\ref{subsec:fineProperties}). 	
	\item A novel notion of \emph{edge generation} allows us to bound changes of the size of neighboring simplices (Section~\ref{sec:locGenDiff}).
	This circumvents the case by case study in \cite{GHS.2016}. 
	\item To resolve the difficulties caused by high-valence vertices we employ \emph{neighborhood arguments}. 
	This is based on the observation that every time a chain of simplices leaves a certain neighborhood, the size of the simplices cannot have decreased too fast. 
	Also the worse bounds due to high-valence vertices cannot enter too often. 
	This results from the fact that for two high-valence vertices to be connected by a chain of simplices, the chain has to leave certain neighborhoods (Section~\ref{sec:leavingTheNgh}). 
	\item To prove the final grading estimates we follow the approach by Carstensen in~\cite{Carstensen02,C.2004} and consider chains of intersecting simplices. 
	Employing estimates on the size of the simplices in such a chain (Section~\ref{sec:GenDiffEdgeChain}) yields estimates on the mesh grading (Section~\ref{sec:bisec-grading}). 
 \end{itemize}
The grading result in Theorem~\ref{thm:main-grading2} for arbitrary dimensions improves the ones in \cite{GHS.2016} for dimension two in the sense that the constants entering only depend on the maximal valence rather than the number of initial vertices. 
Furthermore, it is the first result on any kind of grading estimate in general dimensions for genuinely adaptive refinement schemes. 
Previous results on Sobolev stability of the $L^2(\Omega)$-orthogonal projection rely on the assumption that a certain grading is available in 3D, see \cite{BanYse14,DieningStornTscherpel21}. 
Therefore, Theorem~\ref{thm:main-grading2} is the first result on Sobolev stability for general dimensions~$d \geq 3$ for an adaptive refinement scheme. 

\section{Bisection and Grading}
\label{sec:bisec-grading}
In Section~\ref{sec:bisection-algorithm-1} we introduce the bisection routine~$\Bisec{}$ by Maubach \cite{Maubach95} and Traxler \cite{Traxler97} as well as the notion of generations and levels of~$d$-simplices. 
With this at hand, Section~\ref{subsec:meshGrading} portrays the relation of the mesh grading and levels of simplices that are central in all previous grading results, see, e.g.,~\cite{Carstensen02}. 

\subsection{Bisection algorithm}
\label{sec:bisection-algorithm-1} 
The definition of the bisection routine uses the notion of simplices.
A~$k$-simplex with~$0\leq k \leq d$ is the convex hull of~$k+1$ affinely independent points~$v_0,\dots, v_k\in \setR^d$ and is denoted by~$[v_0,\dots,v_k]\subset \mathbb{R}^d$. 
In particular, a~$0$-simplex is a point, a~$1$-simplex is an edge, and a~$(d-1)$-simplex is a face.
We call~$v_0,\dots, v_k$ vertices of the~$k$-simplex~$S = [v_0,\dots, v_k]$ and we denote the set of all vertices of $S$ by~$\vertices(S) = \lbrace v_0,\dots,v_k\rbrace$.  
An~$m$-simplex spanned by vertices~$\set{w_0,\dots, w_m} \subset \vertices(S)$ with~$0 \leq m \leq k$ is called an~$m$-subsimplex of~$S$.
\begin{definition}[Triangulation]
\label{def:triangulation}
  Let~$\Omega \subset \Rd$ be a bounded, polyhedral domain with~$d \geq 2$.   
  Let~$\tria$ be a partition of~$\overline{\Omega}$ into  closed~$d$-simplices~$T\subset \Rd$ with pairwise disjoint interior. 
  Such a partition~$\tria$ is called \emph{(conforming) triangulation} if the intersection of any two~$d$-simplices~$T_1,T_2 \in \tria$ is either empty or an~$m$-subsimplex of both~$T_1$ and~$T_2$ with~$m \in \set{0,\dots, d}$.
\end{definition}
Let~$\tria_0$ denote a triangulation.
To each~$d$-simplex~$T \in \tria_0$ we assign the \emph{generation}~$\gen(T) \coloneqq 0$. 
The bisection of a~$d$-simplex~$T$ generates two $d$-simplices~$T_1$ and~$T_2$ by bisection of the \emph{bisection edge}~$\bse(T)$ of $T$.  
This generates the mid point~$b = \textrm{mid}(\bse(T))$ of~$\bse(T)$ as new vertex, the so-called \emph{bisection vertex} of $T$. 
We call~$T$ \emph{parent} of~$T_1$ and~$T_2$ and the simplices~$T_1$ and~$T_2$ \emph{children} of~$T$ and write 
\begin{align*}
T = \parent(T_1) = \parent(T_2)\qquad\text{and}\qquad \lbrace T_1,T_2\rbrace = \children(T).
\end{align*}
Intuitively, we call~$T_1, T_2$ \emph{siblings} and define their generation as the one of their parent~$T$ increased by one, i.e.,
\begin{align}
  \label{eq:gen-T}
  \gen(T_1) \coloneqq \gen(T_2) \coloneqq \gen(T)+1.
\end{align}
Furthermore, we introduce the \emph{level} of a simplex~$T$ as
\begin{align} \label{eq:lvl-T}
 \level(T) \coloneqq 
  \left\lceil \frac{\gen(T)}{d} 
  \right\rceil.
\end{align}
Increasing the generation by one corresponds to halving the volume, and increasing the level by one corresponds roughly to halving the diameter of a~$d$-simplex. 
Indeed, with equivalence constants depending only on the shape regularity of~$\tria_0$ and on~$d$, any simplex~$T$ resulting from successive bisection of a simplex~$T_0\in \tria_0$ satisfies
\begin{align}
	\label{eq:diam-level}
  \abs{T} =  2^{-\gen(T)}\,\abs{T_0}\qquad\text{and}\qquad	\diameter(T) \eqsim 2^{-\level(T)}\, \diameter(T_0).
\end{align}

The equivalent bisection rules by Maubach and Traxler are displayed in Algorithms~\ref{algo:Maubach} and~\ref{algo:Traxler}, respectively. 

\begin{figure}[h]
\begin{algorithm}[H]
  \caption{Bisection of $d$-simplex (version by Maubach~\cite{Maubach95})}
  \label{algo:Maubach}
  \SetAlgoLined%
  \SetKwFunction{FuncBisection}{Bisection}
  \FuncBisection{$T$}{%

    \KwData{$d$-simplex~$T=[v_0, \dots, v_d]$ of generation~$\gen(T)$ }
    \KwResult{two children~$\lbrace T_1, T_2 \rbrace = \children(T)$ with generation~$\gen(T)+1$}
    
    $k \coloneqq d - (\gen(T) \bmod d) \in \set{1, \dots, d}$\;
    
    $\bse(T) \coloneqq [v_0,v_k]$\tcp*{bisection edge}
    
    $v' \coloneqq  \bsv(\bse(T)) = (v_0+v_k)/2$\tcp*{bisection vertex}
    
    $T_1 \coloneqq [v_0, v_1, \dots, v_{k-1}, v', v_{k+1}, \dots,
    v_d]
    $\;%
    
    $T_2 \coloneqq [v_1,v_2,\ldots, v_k,v', v_{k+1}, \dots,
    v_d]
    $\;%
  }
\end{algorithm}
\end{figure}
\begin{figure}[h]
\begin{algorithm}[H]
  \caption{Bisection of~$d$-simplex (version by Traxler~\cite{Traxler97,Stevenson08})}
  \label{algo:Traxler}
  \SetAlgoLined%
  \SetKwFunction{FuncBisection}{Bisection}
  \FuncBisection{$T$}{%
    
    \KwData{$d$-simplex~$T=[v_0, \dots, v_d]$ of generation~$\gen(T)$ }
    \KwResult{two children~$\lbrace T_1, T_2 \rbrace = \children(T)$ with generation~$\gen(T)+1$}
    
    $\gamma \coloneqq \gen(T) \bmod d \in \set{0, \dots, d-1}$\tcp*{tag}
    
    $\bse(T) \coloneqq [v_0,v_d]$\tcp*{bisection edge}
    
    $v' \coloneqq  \bsv(\bse(T)) =  (v_0+v_d)/2$\tcp*{bisection vertex}
    
    $T_1 \coloneqq [v_0, v', v_1, \dots, v_\gamma, v_{\gamma+1}, \dots,
    v_{d-1}]
    $\;%
    
    $T_2 \coloneqq [v_d, v', v_1, \dots, v_\gamma, v_{d-1}, \dots,
    v_{\gamma+1}]
    $\;
  }
\end{algorithm}
\end{figure}
If we bisect a single~$d$-simplex in a triangulation, the resulting partition is in general no longer conforming. 
To restore conformity a recursive conforming closure routine is employed, cf.~Algorithm~\ref{algo:closure}.
\begin{figure}[h]
\begin{algorithm}[H]
  \caption{Bisection with conforming closure (\Bisec{})}\label{algo:closure}
  \SetAlgoLined%
  \SetKwFunction{FuncBisect}{BisectionClosure}
  \FuncBisect{$\tria, T$}{
    
    \KwData{triangulation~$\tria$ and a~$d$-simplex~$T\in \tria$}
    \KwResult{triangulation~$\tria'$, where~$T$ is bisected}
   
    $\tria' \coloneqq (\tria \setminus \set{T}) \cup \texttt{Bisection}(T)$\;
    \While{%
      there are~$T',T'' \in \tria'$ with~$S \coloneqq T' \cap T'' \neq \emptyset$ not a subsimplex of~$T'$
    }{%
      $\tria' \coloneqq  (\tria' \setminus \set{T'}) \cup \texttt{Bisection}(T')$\;
    }%
  }
\end{algorithm}
\end{figure}
To ensure that Algorithm~\ref{algo:closure} terminates after finitely many iterations, additional properties on the initial triangulation~$\tria_0$ have to be assumed. 
A sufficient condition is the so-called \emph{matching neighbor condition} by Stevenson~\cite[Sec.~4]{Stevenson08}. 
It is rather general in the sense that it is equivalent to the fact that all uniform refinements of~$\tria_0$ are conforming. 
Let us assume that the initial triangulation~$\tria_0$ satisfies this initial condition. 
By~$\BisecT$ we denote the set of triangulations obtained by iterative application of Algorithm~\ref{algo:closure} (\Bisec) starting from~$\tria_0$. 
The set of all simplices in~$\BisecT$ is denoted by
\begin{align}
\label{def:mtree}
\mtree \coloneqq \lbrace T \in \tria \colon \tria \in \BisecT\rbrace.
\end{align}
For~$T\in\mtree$ we denote the set of its vertices ($0$-simplices) and the set of its edges ($1$-simplices) by~$\vertices(T)$ and $\edges(T)$, respectively. 
This definition extends to collections of simplices~$\mathcal{U} \subset \mtree$ by
\begin{align*}
\vertices(\mathcal{U}) = \bigcup_{T\in \mathcal{U}} \vertices(T)\qquad \text{and}\qquad \edges(\mathcal{U}) = \bigcup_{T\in \mathcal{U}} \edges(T).
\end{align*}
In particular,~$\vertices(\mtree)$ and~$\edges(\mtree)$ denote the sets of all vertices and edges that can be created by~$\Bisec{}$. 
If~$T,T'\in \mtree$ are simplices with~$T'\subsetneq T$, we call~$T'$ a \textit{descendant} of~$T$ and we call~$T$ an \emph{ancestor} of~$T'$. 
\begin{remark}[Kuhn cube and simplices]
\label{rem:algo-kuhn}
  The seemingly complicated bisection algorithm by Maubach and Traxler can be derived in an intuitive way. 
  We consider the Kuhn cube~$[0,1]^d$ and split it into~$d!$ equivalent Kuhn simplices. 
  With permutations~$\pi \colon \set{1,\dots,d} \to \set{1,\dots,d}$ and unit vectors~$\fre_1,\dots, \fre_d \in \mathbb{R}^d$ those Kuhn simplices are given by
  \begin{align}
    \label{eq:kuhn-MT}
    T_\pi = [0, \fre_{\pi(1)}, \fre_{\pi(1)}+\fre_{\pi(2)}, \dots, \fre_{\pi(1)}+\dots+\fre_{\pi(d)}],
  \end{align}
see also~\cite[Sec.~4.1]{B.2000}. 
  The simplices~$T_\pi$ are assigned with generation zero. 
  The bisection algorithm by Maubach and Traxler is constructed in such a way that an application of~$d$ uniform refinements leads to a Tucker--Whitney triangulation~\cite[Sec.~2.3]{WeissDeFloriani11}. 
  The latter consists of~$2^d$ reflected Kuhn cubes of half the diameter, see Figure~\ref{fig:TuckerWhitner}. 
  By this procedure the bisection algorithm is uniquely determined. 
  The ordering of the vertices in~\eqref{eq:kuhn-MT} is chosen such that it is compatible with both Algorithms~\ref{algo:Maubach} and~\ref{algo:Traxler}.
  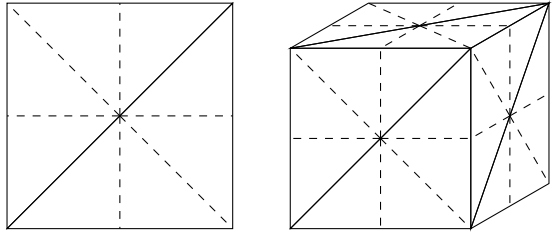
\begin{figure}[ht!]
    \begin{tikzpicture}[scale=3]
      \draw (0,0) -- (1,0) -- (1,1) -- cycle; \draw (0,0) -- (0,1) -- (1,1) -- cycle; \draw[dashed] (.5,.5) -- (0,1); \draw[dashed] (.5,.5) -- (0,.5); \draw[dashed] (.5,.5) -- (1,0); \draw[dashed] (.5,.5) -- (.5,0); \draw[dashed] (.5,.5) -- (.5,1); \draw[dashed] (.5,.5) -- (1,.5);
    \end{tikzpicture}
    \hspace*{.5cm}
    \begin{tikzpicture}[%
      x={(0:1cm)}, y={(30:0.5cm)}, z={(90:1.0cm)}, scale=12/5]
      \draw (0,0,0) -- (1,0,0) -- (1,0,1) -- cycle; \draw (0,0,0) -- (0,0,1) -- (1,0,1) -- cycle; \draw (1,0,0) -- (1,1,0) -- (1,1,1) -- cycle; \draw (1,0,0) -- (1,0,1) -- (1,1,1) -- cycle; \draw (0,0,1) -- (1,0,1) -- (1,1,1) -- cycle; \draw (0,0,1) -- (0,1,1) -- (1,1,1) -- cycle; \draw[dashed] (1/2,0,1/2) -- (0,0,1); \draw[dashed] (1/2,0,1/2) -- (1,0,0); \draw[dashed] (1,1/2,1/2) -- (1,0,1); \draw[dashed] (1,1/2,1/2) -- (1,1,0); \draw[dashed] (1/2,1/2,1) -- (1,0,1); \draw[dashed] (1/2,1/2,1) -- (0,1,1); \draw[dashed] (1/2,0,1/2) -- (0,0,1/2); \draw[dashed] (1/2,0,1/2) -- (1,0,1/2); \draw[dashed] (1/2,0,1/2) -- (1/2,0,0); \draw[dashed] (1/2,0,1/2) -- (1/2,0,1); \draw[dashed] (1,1/2,1/2) -- (1,0,1/2); \draw[dashed] (1,1/2,1/2) -- (1,1,1/2); \draw[dashed] (1,1/2,1/2) -- (1,1/2,0); \draw[dashed] (1,1/2,1/2) -- (1,1/2,1); \draw[dashed] (1/2,1/2,1) -- (0,1/2,1); \draw[dashed] (1/2,1/2,1) -- (1,1/2,1); \draw[dashed] (1/2,1/2,1) -- (1/2,0,1); \draw[dashed] (1/2,1/2,1) -- (1/2,1,1);
    \end{tikzpicture}
    \caption{Partition of Kuhn cube~$[0,1]^d$ for~$d=2,3$ (solid lines) and after~$d$ uniform refinements (dashed lines)}\label{fig:TuckerWhitner}
  \end{figure}
\end{remark}

\begin{remark}[Longest edge]
  \label{rem:longest-edge}
  The bisection rule~$\Bisec$ by Maubach and Traxler can also be interpreted as a {longest edge bisection}. 
  Indeed, let~$\tria_0$ denote the Kuhn triangulation of the Kuhn cube~$[0,1]^d$ into~$d!$ Kuhn simplices. 
  Then the bisection rule is equivalent to the longest edge bisection using the Euclidean distance for~$d\leq 3$. 
  For~$d \geq 4$ this is not true.  
  However,~$\Bisec$ is the longest edge bisection for the Kuhn triangulation for all~$d \geq 2$, if we replace the Euclidean distance by the norm
  \begin{align*}
    \normm{x} \coloneqq \norm{x}_\infty + \frac{1}{d} \norm{x}_1 = \max_{i=1,\dots, d} \abs{x_i} + \frac 1d \sum_{i=1}^d\abs{x_i} \quad \text{ for } x \in \setR^d. 
  \end{align*}
\end{remark}

\subsection{Mesh grading}
\label{subsec:meshGrading}
Theorem~\ref{thm:main-grading2} states that there is a mesh size function with grading~$\gamma = 2$.
Its construction is based on the following distance function. 
\begin{definition}[Simplex distance in $\tria$]
	\label{def:distances}
    For~$\tria \in \BisecT$ we define the distance function~$\delta_{\tria}\colon \tria \times \tria \to \mathbb{N}_0$ as the length of the shortest chain of intersecting simplices minus one, that is, the minimal number~$N\in  \mathbb{N}_0$ of simplices~$T_0,\dots,T_N\in \tria$ with~$T_0 = T$,~$T_N = T'$, and
    \begin{align*}
    T_j \cap T_{j+1}\neq \emptyset \qquad\text{for all }j=0,\dots,N-1.
    \end{align*}
\end{definition}
To verify the grading result stated in Theorem~\ref{thm:main-grading2} we use a mesh size function based on a design by Carstensen in~\cite{Carstensen02}.
\begin{definition}[Regularized mesh size function] \label{def:regMeshSize-T}
	For an initial triangulation $\tria_0$, let $h_{0} \coloneqq (\# \tria_0)^{-1} \sum_{T_0 \in \tria_0} \diameter(T_0)$ denote its averaged diameter. 
	For any trian\-gulation~$\tria \in \BisecT$ the \emph{regularized mesh size function}~$h_\tria\colon \tria \to \mathbb{R}$ reads
	\begin{align}\label{eq:defhTria-v}
		h_\tria (T) \coloneqq h_0\, \min_{T'\in \tria} 2^{-\level(T')+\delta_{\tria}(T,T')}\qquad \text{ for any } T \in \tria. 
	\end{align}
\end{definition}
By definition the function $h_{\tria}$ satisfies
\begin{align}
	\label{est:h-grading2}
	\frac{h_{\tria}(T)}{h_{\tria}(T')} \leq 2\qquad\text{for all }T, T' \in \tria\text{ with }\delta_\tria(T,T') = 1.
\end{align}
According to Definition~\ref{def:grading}~\ref{itm:grading} this means that~$h_{\tria}$ has grading $\gamma = 2$.
The main challenge consists in verifying that~$h_{\tria}$ is indeed a mesh size function in the sense of Definition~\ref{def:grading}~\ref{itm:meshfct}, i.e., to show the equivalence of~$h_{\tria}(T)$ and~$\textup{diam}(T)$.

\begin{lemma}[Level estimate implies grading]
	\label{lem:lvlDiffToMain}
	Suppose there is a constant~$c_0 < \infty$ depending only on the initial triangulation~$\tria_0$ with
	\begin{align}
		\label{eq:mainAbstract-lvl}
		\level(T) - \level(T')\leq  \delta_{\tria}(T,T') + c_0  \qquad\text{for all }T,T'\in \tria\in \BisecT.
	\end{align}
	Then all triangulations $\tria\in \BisecT$ have grading~$\gamma \leq 2$ with constants~$c_1 \eqsim 2^{-c_0}$ and~$c_2\eqsim 1$, where the hidden equivalence constants depend solely on the quasi-uniformity, the shape regularity of~$\tria_0$, and on~$d$.
	In particular, Theorem~\ref{thm:main-grading2} holds. 
\end{lemma}
\begin{proof}
Let~$\tria \in \BisecT$ be arbitrary and let~$h_{\tria}$ be as in Definition~\ref{def:regMeshSize-T}.  
Recall that by \eqref{est:h-grading2} the function~$h_{\tria}$ satisfies Definition~\ref{def:grading}~\ref{itm:grading} with $\gamma = 2$, i.e., it has grading $\gamma = 2$. 
It remains to prove the equivalence~$h_{\tria}(T) \eqsim \diameter(T)$ for all $T\in \tria$. 

By~\eqref{eq:diam-level} the increase of the level halves the diameter of a simplex. Hence, the quasi-uniformity of $\tria_0$ yields with the averaged diameter $h_0$ in $\tria_0$ that
\begin{align}
	\label{eq:diam-level-h0}
\textup{diam}(T) \eqsim h_0 2^{-\level(T)} \qquad\text{for all }T\in \tria.
\end{align}
It follows that for any simplex~$T\in \tria$ one has that
\begin{align*}
		h_{\tria}(T)
		&		= h_0 \min_{T'\in \tria} 2^{- \level(T')+\delta_{\tria}(T,T')}  \leq h_0 2^{- \level(T)} \lesssim \diameter(T). 
\end{align*}
	Moreover, applying the level estimate~\eqref{eq:mainAbstract-lvl} and the property in~\eqref{eq:diam-level-h0} we find that
	\begin{align*}
		h_{\tria}(T)
		&
		= h_0 \min_{T'\in \tria} 2^{ - \level(T')+\delta_{\tria}(T,T')}
		\geq h_0 2^{- \level(T) - c_0} 
		\gtrsim \, 2^{-c_0}\diameter(T). 
	\end{align*}
	Combining both estimates shows the equivalence of the regularized mesh size function~$h_{\tria}$ with grading~$\gamma = 2$ to the local mesh size in the sense that
	\begin{align*}
		 2^{-c_0} \diameter(T)  \lesssim  h_{\tria}(T) \lesssim 	 \diameter(T)\qquad\text{for all } T\in \tria.
	\end{align*}
The equivalence constants depend on $d$ as well as on the shape regularity and quasi-uniformity of $\tria_0$.
\end{proof}

The remainder of this work is dedicated to prove the estimate in~\eqref{eq:mainAbstract-lvl} with an explicit bound for~$c_0$, see Propositions~\ref{pro:genest} and~\ref{pro:genest-stevenson} below.  
Notice that our analysis gives rise to a constant~$c_0$ that depends only on the shape regularity of~$\tria_0$ (more precisely, the maximal length of edge paths on surfaces of vertex patches, cf.~Section~\ref{subsec:gen-diff-vtx} below). 

\section{Properties of the bisection routine}
\label{sec:bisection-algorithm}

In this section we collect properties of triangulations generated by \Bisec{}. 
We start by presenting basic properties in Section~\ref{sec:lattice-structure}. 
In Section~\ref{ssec:gener-level-struct} we introduce a vertex based generation structure which allows us to formulate an equivalent generation based version of~\texttt{Bisec} in Section~\ref{sec:reform-using-gener}. 
This formulation gives us access to additional fine properties presented in Section~\ref{subsec:fineProperties}. 
Finally, in Section~\ref{subsec:genSharp} we investigate a novel alternative notion of generation with structural advantages. 

\subsection{Notation and lattice structure}
\label{sec:lattice-structure}
The refinement by~$\Bisec{}$ defines a natural partial order on~$\BisecT$. 
More precisely, let~$\tria_1,\tria_2 \in \BisecT$ be two triangulations. 
Then we say that~$\tria_1$ is a refinement of~$\tria_2$ (and~$\tria_2$ is a coarsening of~$\tria_1$), in short $\tria_2 \leq \tria_1$, if~$\tria_1$ can be obtained from~$\tria_2$ by repeated~\Bisec{} refinements. 
This makes~$(\BisecT,\leq)$ a partially ordered set.

To any triangulation~$\tria\in\BisecT$ we can assign the \emph{forest}~$\forest(\tria)$ with root elements in~$\tria_0$ and children structure from the bisection algorithm. 
In fact,~$\forest(\tria)$ is a rooted forest of full binary trees each of them rooted in some~$T \in \tria_0$. 
Note that~$\tria$ is just the set of leaves of~$\forest(\tria)$ and that~$\tria_2 \leq \tria_1$ is equivalent to~$\forest(\tria_2) \subset \forest(\tria_1)$.
For each~$\tria_1,\tria_2 \in \BisecT$ there exists a coarsest common refinement~$\tria_1 \vee \tria_2$ and a finest common coarsening~$\tria_1 \wedge \tria_2$. 
The triangulations~$\tria_1 \vee \tria_2$ and $\tria_1 \wedge \tria_2$ are determined by 
\begin{align}
  \label{eq:lattice}
  \begin{aligned}
    \forest(\tria_1 \vee \tria_2) = \forest(\tria_1) \cup \forest(\tria_2)
    \quad \text{and}\quad
    \forest(\tria_1 \wedge \tria_2) = \forest(\tria_1) \cap \forest(\tria_2).
  \end{aligned}
\end{align}
With this structure~$(\BisecT,\vee,\wedge)$ is an order-theoretic lattice with bottom element~$\tria_0$, see~\cite{DieningKreuzerStevenson16}. 
Moreover,~$(\BisecT,\vee,\wedge)$ is distributive, i.e.,
\begin{align}
  \label{eq:distributive}
  \begin{aligned}
    \tria_1 \vee (\tria_2 \wedge \tria_3) &= (\tria_1 \vee \tria_2) \wedge (\tria_1 \vee \tria_2), \\
    \tria_1 \wedge (\tria_2 \vee \tria_3) &= (\tria_1 \wedge \tria_2) \vee (\tria_1 \wedge \tria_2).
  \end{aligned}
\end{align} 
Let~$(\overline{\BisecT},  \vee, \wedge)$ be the Dedekind-MacNeille completion of~$(\BisecT, \vee, \wedge)$. We denote by~$\tria_\infty$ its top element, meaning that~$\tria \leq \tria_\infty$ for all~$\tria \in \overline{\BisecT}$.
We refer to any such~$\tria$ as triangulation, even if it is not finite. 
The forest of~$\tria_{\infty}$ is denoted by~$\forest(\tria_\infty) \coloneqq \bigcup_{\tria \in \BisecT} \forest(\tria)$. 
The collection of nodes of the forest~$\forest(\tria_\infty)$ coincides with the set of simplices~$\mtree$ defined in~\eqref{def:mtree}. 
Both~$\forest(\tria_\infty)$ and~$\mtree$ are referred to as \emph{master tree} in the literature. 
We define also the forests
\begin{align*}
\forest(\tria) \coloneqq \bigcup_{\tria' \in \BisecT \colon \tria' \leq \tria} \forest(\tria') \qquad \text{for any } \tria \in \overline{\BisecT}.
\end{align*}
This allows us to compare two possibly infinite triangulations~$\tria_1, \tria_2 \in \overline{\BisecT}$ in the sense that~$\tria_1 \leq \tria_2$ if~$\forest(\tria_1)\subset \forest(\tria_2)$, which is consistent with the above. 

\subsection{Generation and level structure}
\label{ssec:gener-level-struct}
By Section~\ref{sec:bisection-algorithm-1} each~$d$-simplex~$T \in \mtree$ is assigned with a generation~$\gen(T)$. 
To obtain finer properties of~\Bisec{} we associate to each vertex~$v \in \vertices(\mtree)$ a generation in a way that ensures compatibility with the generation of $d$-simplices. 
More specifically, we define the generation of any vertex~$v\in \vertices(\mtree)$ that is generated by the bisection of a simplex~$T\in \mtree$ as
\begin{align}
  \label{eq:gen-inductive}
  \gen(v) &\coloneqq \gen(T) +1.
\end{align}
In order to assign a generation also to each initial vertex~$v \in \vertices(\tria_0)$ in a consistent manner, we use the following assumption on the initial triangulation which is equivalent to the reflected domain partition condition in~\cite[Sec.~6]{Traxler97}. 

\begin{assumption}[Coloring condition]
  \label{ass:initial-triangulation}
  We call an initial triangulation~$\tria_0$ \emph{colored}, if it is~$(d+1)$-chromatic in the sense that there exists a proper vertex coloring~$\frc\colon \vertices(\tria_0) \to \set{0,\dots, d}$ with
  \begin{align*}
    \set{ \frc(v)\colon v \in \vertices(T)} = \set{0,\dots, d} \qquad \text{for each } T \in \tria_0.
  \end{align*}
\end{assumption}

Suppose that~$\tria_0$ is a colored triangulation. 
We define
\begin{align}
 \label{eq:gen-initial}
  \gen(v) &\coloneqq -\frc(v) \qquad \text{for any } v \in \vertices(\tria_0).
\end{align}

\begin{remark}[Comparison]
\label{rmk:initial-tria}
With a suitable sorting of the initial vertices in initial $d$-simplices as presented below in~\eqref{eq:coloringToMatching}
the coloring condition leads to a partition that satisfies the matching neighbor condition in~\cite[Sec.~4]{Stevenson08}, see Section~\ref{sec:stevenson} below. 
Hence, Assumption~\ref{ass:initial-triangulation} is a stronger assumption. 
Notice however that any triangulation with matching neighbor condition is colored after one full refinement, see Lemma~\ref{lem:tria0plus_colored}.
Furthermore, for an arbitrary triangulation a specific initial refinement as described in \cite[Appendix A]{Stevenson08} leads  to a colored triangulation $\tria_0$.
\end{remark}
The bisection vertex of~$T$ is the youngest vertex of its children~$T_1$ and~$T_2$. 
In other words, the bisection vertex has the highest generation in~$T_1$ and~$T_2$. 
Hence, the construction ensures that for each simplex~$T=[v_0,\dots, v_d]\in \mtree$ we have
\begin{align}
\label{eq:gen-simplices-vertices}
  \gen(T) = \gen([v_0,\dots,v_d]) = \max_{j=0,\dots,d} \gen(v_j).
\end{align}
This also holds for initial simplices~$T \in \tria_0$ since~$\gen(T) = \max \set{-d,-d+1,\dots, 0} = 0$. 
Consistently, for every~$m$-simplex~$S=[w_0,\dots, w_m]$ we set its generation to
\begin{align}
\label{eq:gen-simplices-sub}
  \gen(S) = \gen([w_0,\dots, w_m]) \coloneqq \max_{j=0,\dots,m} \gen(w_j).
\end{align}
In particular, the definition of the generation~$\gen$ for vertices extends to the complex of~$\tria$ and this extension is compatible with the generation structure for $d$-simplices by Maubach and Traxler. 
In the following we shall associate generation with age and refer $m$-simplices as \emph{older} or \emph{younger} than others depending on their generation.
We define for every~$m$-simplex~$S$ its level and type by
\begin{align}
\label{eq:def-type-level}
  \begin{aligned}
    \level(S) &\coloneqq \left\lceil \frac{\gen(S)}{d} \right\rceil,
    \\
    \type(S) &\coloneqq \gen(S) - d\, \big(\level(S) -1\big) \in \set{1,\dots,d}.
  \end{aligned}
\end{align} 
Note that for the quantities $k$ and $\gamma$ in the Algorithms~\ref{algo:Maubach} and~\ref{algo:Traxler} we have that~$k=d - (\type(T) \bmod d)$ and~$\gamma=\type(T) \bmod d$. 
The definition of~$\type$ ensures that
\begin{align}
  \label{eq:generation-level-type}
  \gen(\bigcdot) &= d\, \big(\level(\bigcdot) -1\big) + \type(\bigcdot).
\end{align}
In particular, we have~$\gen(\bigcdot) \bmod d = \type (\bigcdot) \bmod d$. 
We exemplify the corresponding values for dimension~$d=3$ in Table~\ref{tab:types}.
\begin{table}[h!]
  \centering
  \begin{tabular}{c|cc|c|c|l}
    $\gen$ & $\level$ & $\type$ & $k$ Maubach & $\gamma$ Traxler  & description
    \\
    \hline
    0  & 0 & 3 & 3 & 0 & initial triangulation
    \\ \hline
    1  & 1 & 1 & 2 & 1 & first refinement
    \\
    2  & 1 & 2 & 1 & 2 & 
    \\  
    3  & 1 & 3 & 3 & 0 &full uniform refinement
    \\
    \hline
    4  & 2 & 1 & 2 & 1&
  \end{tabular}
  \vspace*{1mm}
  \caption{Generations, levels, and types of~$d$-simplices for~$d=3$.}
  \label{tab:types}
\end{table}%
\subsection{Reformulation of \Bisec{} using generations}
\label{sec:reform-using-gener}
In this section we present an equivalent formulation of the~\Bisec{} algorithm based on the generation structure introduced before. 
To verify the equivalence we rewrite Algorithm~\ref{algo:Maubach} in a more illustrative manner. 
For a simplex~$T = [v_0,\dots,v_d]$ with~$k = d - (\gen(T) \bmod d) \in \{1, \ldots, d\}$ we denote its~$\front$ and its possibly empty~$\back$ by 
\begin{align*}
\front(T) \coloneqq [v_0,\dots,v_k]\qquad\text{and}\qquad\back(T) \coloneqq [v_{k+1},\dots,v_d].
\end{align*}
The division into~$\front$ and~$\back$ is indicated by a vertical line, that is,
\begin{align*}
T = [v_0,\dots,v_k\mid v_{k+1},\dots,v_d].
\end{align*}
The routine by Maubach in Algorithm~\ref{algo:Maubach} determines the bisection edge as the edge between the first and the last vertex of~$\front(T)$. 
The bisection vertex~$v' = \bsv([v_0,v_k])$ is moved to the first position of~$\back(T)$ in the children of $T$, i.e., the children are
\begin{equation}
\label{eq:RefNew1}
\begin{alignedat}{2}
T_1 &= [v_0,\dots,v_{k-1}&\mid v', v_{k+1},\dots,v_d],\\
T_2 &= [v_1,\dots,v_k&\mid v', v_{k+1},\dots,v_d].
\end{alignedat}
\end{equation}
The corresponding~$k$ has reduced by one, which is represented by the position of the vertical line having moved one to the left. 
Notice that for~$k=1$ we have to redefine the children of $T = [v_0, v_1 \mid v_2, \ldots, v_d]$ as
\begin{align*}
\begin{aligned}
T_1 &= [v_0 \mid v' , v_2 ,\dots,v_d] \eqqcolon [v_0 , v' , v_2 ,\dots,v_d \mid\, ] ,\\
T_2 &= [v_1 \mid v' , v_2,\dots,v_d] \eqqcolon [v_1 , v' , v_2 ,\dots,v_d \mid\, ].
\end{aligned}
\end{align*}
This corresponds to the fact that $k = d$ for $T_1$ and $T_2$.  
Suppose that~$\tria_0$ is a colored triangulation and that the vertices in initial simplices~$T = [v_0,\dots,v_d] \in \tria_0$ are sorted such that their coloring reads
\begin{align}
\label{eq:coloringToMatching}
\frc(T) = (\frc(v_0),\dots,\frc(v_d)) = (d,0,1,\dots,d-1).
\end{align}
A simple induction and the properties illustrated in~\eqref{eq:RefNew1} reveal that simplices in~$\back(T)$ have generations of consecutive order, they share the same level, and $\type(v_j) = d-j+1$ for~$j=k+1,\dots,d$. 
Moreover, inductively the following lemma can be proved.

\begin{lemma}[Structure of Maubach sorting]
  \label{lem:maubach}
  Let~$T = [v_0,\dots,v_d]$ with~$k = d- (\gen(T) \bmod d)$ result from bisections according to  Algorithm~\ref{algo:Maubach} of an initial simplex satisfying \eqref{eq:coloringToMatching}.
  \begin{enumerate}
  \item 
  \label{itm:maubach1} If~$k=d$, then
    \begin{alignat*}{2}
      \level(v_0) &< \level(v_1) = \dots = \level(v_d),\;\;\;\; &&
      \\
      \gen(v_0) &< \gen(v_1),
      \\
      \gen(v_j) &= \level(T)d + 1 - j&\;\;\qquad&\text{for $j=1,\dots, d$}.
    \end{alignat*}
  \item 
  \label{itm:maubach2} If~$1 \leq k < d$ and~$\level(v_0) < \level(v_1)$, then
    \begin{alignat*}{2}
      \level(v_1) &= \dots = \level(v_k),
      \\
      \level(v_{k+1}) &= \dots = \level(v_d) =\level(v_k)+1, \,&&
      \\
      \gen(v_0) &< \gen(v_1),
      \\
      \gen(v_j) &= \level(T)d + 1 - d - j &\quad&\text{for $j=1,\dots, k$},\\
      \gen(v_j) &= \level(T)d + 1 - j &\quad&\text{for $j=k+1,\dots, d$}. 
    \end{alignat*}
  \item 
  \label{itm:maubach3} If~$1 \leq k < d$ and~$\level(v_0) = \level(v_1)$, then
    \begin{alignat*}{2}
      \level(v_0) &= \dots = \level(v_k),
      \\
      \level(v_{k+1}) &= \dots = \level(v_d) =\level(v_k)+1, &&
      \\
	\gen(v_{k+1}) &= \gen(v_k) + 2d - \type(v_0),
	\\      
      \gen(v_j) &= \gen(v_{j+1}) +1 &\quad&\text{for $j=0,\dots, k$},
      \\
      \gen(v_j) &= \level(T)d + 1 - j  &&\text{for $j=k+1,\dots, d$}.
    \end{alignat*}
  \end{enumerate}
  \end{lemma}
  \begin{proof}
  	Most of the statements follow by a straight-forward induction on the refinement routine as illustrated in~\eqref{eq:RefNew1}. 
  	Thus, the only claim we want to show explicitly is the fact that~$\gen(v_{k+1}) = \gen(v_k) + 2d - \type(v_0)$ in~\ref{itm:maubach3}. 
  	Noting that~$\type(v_{k+1}) = d-k$ we find that 
  	\begin{align}\label{eq:Proffsdas1}
  	\begin{aligned}
  		\gen(v_{k+1}) &= d (\level(v_{k+1})- 1) + \type(v_{k+1}) \\
  		&= d (\level(v_{k+1})- 1) + d - k 
  		= d \, \level(v_{k+1}) - k. 
  	\end{aligned}
  	\end{align}
  	On the other hand, by the identities on the generations and the levels we have
  	\begin{align*}
	\gen(v_0) = \gen(v_k) + k \quad \text{ and } \quad \level(v_0) = \level(v_k) = \level(v_{k+1})-1. 
  	\end{align*}
  	Hence, it follows that
  	\begin{align}\label{eq:Proffsdas2}
  		\type(v_0) = \gen(v_0) - d (\level(v_0) - 1) = \gen(v_k) + k - d (\level(v_{k+1})-2). 
  	\end{align}
  	Combining the identities in \eqref{eq:Proffsdas1} and \eqref{eq:Proffsdas2} yields the claim. 
\end{proof}
 \noindent Thanks to the previous lemma the vertices of~$T = [v_0,\dots,v_d]$ can be divided into 
  \begin{itemize}
  	\item two blocks of the same level in cases \ref{itm:maubach1} and \ref{itm:maubach3}, and into 
  	\item three blocks of the same level in case \ref{itm:maubach2}. 
  \end{itemize} 
Within each of those blocks the vertices are of consecutive decreasing generation. 
The goal is to reformulate the routine such that the vertices in each~$d$-simplex are sorted by decreasing generation. 
This can be achieved by rearranging the blocks. 

Before doing so, let us introduce some notation.
For an~$m$-simplex~$S$ with vertices~$v_0, \ldots, v_m \in \vertices(S)$ we use the notation 
\begin{align}\label{eq:sorted}
	\simplex{v_0,\dots, v_m} \quad \text{ if } \quad  \gen(v_0) > \dots > \gen(v_m)
\end{align}
to indicate that the vertices are sorted by decreasing generation. 
That is, for any~$m$-simplex denoted with brackets~$\simplex{\;}$ rather than~$[\;]$ the vertices are sorted by decreasing generation. 
Since no two vertices have the same generation in one simplex, this representation is unique and we write~$S = \simplex{v_0,\dots, v_m}$. 
Moreover, it is useful to indicate the last level jump within the simplex by a vertical line. More precisely,~$T=\simplex{v_0,\dots, v_{\ell-1} \leveljump v_\ell,\dots,v_d}$ means that
\begin{align*}
	\level(v_{\ell-1}) > \level(v_\ell) = \dots = \level(v_d)
\end{align*}
and we refer to
\begin{align*}
\head(T) \coloneqq \simplex{v_0, \dots, v_{\ell-1}} \quad \text{ and } \quad \tail(T) \coloneqq \simplex{v_\ell,\dots v_d}
\end{align*}
as the \emph{head} and the \emph{tail} of $T$. 
Note that within~$\head(T)$ there might be another level jump but this is not very important in the following. 
We use the same notation for~$m$-simplices~$S = \simplex{v_0,\dots, v_m}$, i.e., the tail contains all vertices which are of the same level as~$v_m$ and the head contains all vertices with higher level than~$v_m$. 
In contrast to~$d$-simplices the head of an~$m$-simplex with~$m <d$ might be empty.

Let~$T=[y_0,\dots, y_d]\in \mtree$ be a simplex resulting from the bisection routine by Maubach in Algorithm~\ref{algo:Maubach}. 
By Lemma~\ref{lem:maubach} each block of vertices of the same level is  sorted by decreasing generation. Thus, to sort the vertices by decreasing generation we only have to rearrange the blocks by decreasing level. 
Let us present this rearrangement in each of the three cases~\ref{itm:maubach1}--\ref{itm:maubach3} in Lemma~\ref{lem:maubach}: 
\begin{enumerate}
\myitem[(R-a)]
\label{itm:maubach1b} 
If~$k=d$, then~$\bse(T) = [y_0,y_d]$ and 
  \begin{align*}
    [y_0,\dots,y_d \mid \, ]=
    \simplex{y_1, \ldots, y_d \leveljump y_0}.
  \end{align*}
\myitem[(R-b)]
\label{itm:maubach2b} 
If~$1 \leq k < d$ and~$\level(y_0) < \level(y_1)$, then~$\bse(T) = [y_0,y_k]$ and 
  \begin{align*}
    [y_0,\dots, y_k \mid y_{k+1} ,\dots,y_d]=
    \simplex{y_{k+1}, \dots, y_d,  y_1, \dots, y_k \leveljump y_0}.
  \end{align*}
\myitem[(R-c)]
\label{itm:maubach3b} 
If~$1 \leq k < d$ and~$\level(y_0) = \level(y_1)$, then~$\bse(T) = [y_0,y_k]$ and 
\begin{align*}
    [y_0, \dots, y_k \mid y_{k+1} ,\dots,y_d] =\simplex{y_{k+1}, \dots, y_d \leveljump  y_0, \dots, y_k}.
\end{align*}
\end{enumerate}
The following lemma characterizes head, tail, and bisection edge of a $d$-simplex. 

\begin{lemma}[Bisection edge in generation sorting]
 \label{lem:genbased}
  For a colored initial triangulation~$\tria_0$ let~$T = \simplex{v_0,\dots, v_d} \in \mtree$ be a simplex with bisection edge~$\bse(T)\in\edges(T)$.  
  \begin{enumerate}
  \item 
  \label{itm:genbaseda} 
  If $\level(v_{d-1})\neq \level(v_d)$, then $	\head(T) = \simplex{v_0,\dots,v_{d-1}}$,   
  $\tail(T) = \simplex{v_d}$, 
  \begin{align*}
	 	\bse(T) = \simplex{v_{d-1} \mid v_d}.  
  \end{align*}
  \item 
  \label{itm:genbasedb} 
  If~$\level(v_{d-1}) = \level(v_d)$, then we have that $\type(v_0)<d$, as well as $\head(T) = \simplex{v_0,\dots,v_{ \type(v_0)-1}}$,~$\tail(T) = \simplex{v_{\type(v_0)},\dots,v_d}$ and 
  \begin{align*}
  	\bse(T) = \simplex{v_{\type(v_0)},v_d}.
  \end{align*}
  \end{enumerate}
\end{lemma}
\begin{proof}
Let the simplex~$T = \simplex{v_0, \ldots, v_d}$ correspond to~$[y_0,\dots, y_d]$ in the sorting by Maubach. 
With~$k =d - (\type(T) \bmod d)$ the bisection edge is~$\bse(T) = [y_0,y_k]$. 

If~$\level(v_{d-1})\neq \level(v_d)$, then the rearrangements in~\ref{itm:maubach1b} and~\ref{itm:maubach2b} show that
\begin{align*}
\head(T) = \simplex{v_0,\dots, v_{d-1}}\qquad\text{and}\qquad \tail(T) = \simplex{v_d}.
\end{align*}
Moreover, the bisection edge~$[y_0,y_k]=\simplex{v_{d-1}\mid v_d}$ consists of the single vertex~$v_d$ in the tail and the oldest vertex~$v_{d-1}$ of the head. 
This proves~\ref{itm:genbaseda}. 

If~$\level(v_{d-1}) = \level(v_d)$, then the rearrangement in~\ref{itm:maubach3b} shows that
\begin{align*}
\head(T) = \simplex{v_0,\dots, v_{d-k-1}}\quad\text{and}\quad \tail(T) = \simplex{v_{d-k},\dots, v_d} = [y_0,\dots, y_k].
\end{align*}
Since in this case we have~$1\leq k < d$, it follows that~$k = d-\type(T)=d-\type(v_0)$. 
The bisection edge is~$[y_0,y_k] = \simplex{v_{d-k},v_d}=\simplex{v_{\type(y_0)},v_d}$. 
This yields~\ref{itm:genbasedb}. 
\end{proof}

Lemma~\ref{lem:genbased} allows us to equivalently reformulate the refinement routine by Maubach in Algorithm~\ref{algo:Maubach} as generation based routine in~Algorithm~\ref{algo:genbased}.
The children of~$T$ are obtained by removing one of the vertices of the bisection edge and appending the bisection vertex, for which we use the following notation. 
For vertices~$z_0, \ldots, z_{m+1}\in \mathbb{R}^d$ we define
\begin{align*}  
  \simplex{z_0, \ldots, z_{m+1}} \setminus [z_i] \coloneqq \simplex{z_0, \ldots, z_{i-1}, z_{i+1}, \ldots, z_{m+1}}.
\end{align*}
By~\eqref{eq:gen-initial} we have that~$\gen(v_j) = -j$ for initial simplices~$T = \simplex{v_0,\dots, v_d} \in \tria_0$ and~$j=0,\dots, d$. 
Thus, the levels are~$\level(v_j)=0$ for~$j=0,\dots, d-1$ and~$\level(v_d)=-1$. 
In particular, we have~$T = \simplex{v_0,\dots, v_{d-1} \leveljump v_d}$ and~$\type(T)=d$. 
This is consistent with the resorting of initial simplices in~\ref{itm:maubach1b} for the algorithm by Maubach.

\begin{figure}[h]
\begin{algorithm}[H]
  \label{algo:genbased}
  \caption{Generation based bisection of~$d$-simplex}
  \SetAlgoLined%
  \SetKwFunction{FuncBisection}{Bisection}
  \FuncBisection{$T$}{%

    \KwData{$d$-simplex~$T=\simplex{v_0, \dots, v_d}$ sorted by decreasing vertex generation}

    \KwResult{two children~$\lbrace T_1,T_2\rbrace = \children(T)$, bisection edge $\bse(T)=\simplex{b_0,b_1}$ (all sorted by decreasing vertex generation), and bisection vertex~$b$}
    
    \If{$\level(v_d) \neq \level(v_{d-1})$}{$\bse(T) \coloneqq \simplex{v_{d-1}\mid v_d}$\tcp*{two oldest vertices}}
    \If{$\level(v_d) = \level(v_{d-1})$}{$\bse(T)\coloneqq \simplex{v_{\type(v_0)},v_d}$\tcp*{youngest and oldest vertex of tail}}
    
    Bisection vertex~$b \coloneqq (b_0 + b_1)/2$ is midpoint of~$\bse(T)=\simplex{b_0,b_1}$\;
    
    $T_1 \coloneqq \simplex{b, v_0, \dots, v_d} \setminus [b_0]$\tcp*{remove~$b_0$ and add~$b$}
    
    $T_2 \coloneqq \simplex{b, v_0, \dots, v_d} \setminus [b_1]$\tcp*{remove~$b_1$ and add~$b$}
  }
\end{algorithm}
\end{figure}

Note that there is a one-to-one correspondence between the simplices in Maubach sorting and in the version sorted by generation, and the bisection edges uniquely determines the bisection routine. 
Consequently for equivalence of the bisection routines in Algorithms~\ref{algo:Maubach} and \ref{algo:genbased} there is nothing more to show. 

Using the structure of the sorting in the routine by Maubach, the Lemma~\ref{lem:maubach}, and~\ref{itm:maubach1b}--\ref{itm:maubach3b} we obtain additional structure of the generation based sorting. 

\begin{lemma}[Structure of generation based sorting]
  \label{lem:genbased2}
  Let~$\tria_0$ be colored and let $T = \simplex{v_0,\dots, v_{\ell-1} \leveljump v_\ell,\dots, v_d} \in \tria \in \BisecT$ with~$\ell \in \lbrace 1, \dots, d\rbrace$ be a simplex with bisection edge~$\bse(T)$ and bisection vertex~$b = \bsv(\bse(T))$. 
  \begin{enumerate}
  \item 
  \label{itm:genbased2a} 
  Both $\head(T)$ and $\tail(T)$ consist of vertices with consecutive generations.
  \item 
  \label{itm:genbased2b} 
  If~$\level(v_d) = \level(v_{d-1})$, then~$\level(v_{\ell-1}) = \level(v_\ell)+1$ and
    \begin{align*}
     \qquad \gen(v_0) &= \gen(v_d) + 2d - \type(v_\ell)
                  = \gen(v_d) + 2d - \type(\tail(T)).
     \end{align*}
Furthermore, there is no level jump within $\head(T)$.
\item 
\label{itm:genbased2d} 
 If~$\level(v_d) \neq \level(v_{d-1})$, then we have that
\begin{align*}
\level(v_0) = \level(v_{d-1}) \;\; \text{ if and only if }\;\; \type(T)=d.
\end{align*}
I.e., there is no level jump within~$\head(T)$ if and only if~$\type(T) = d$. 
  \item 
  \label{itm:genbased2c} 
  The level of the bisection vertex is~$\level(b) = \level(\bse(T))+1$. 
  \end{enumerate}
\end{lemma}
\begin{proof}
Let~$T = \simplex{v_0,\dots,v_{\ell-1} \leveljump v_\ell,\dots, v_d} \in \tria \in \BisecT$ be a simplex with corresponding Maubach sorting~$[y_0,\dots, y_d]$ with~$k = d - (\type(T) \bmod d)$.
The property in~\ref{itm:genbased2a} follows from Lemma~\ref{lem:maubach} and the rearrangements in \ref{itm:maubach1b}--\ref{itm:maubach3b}.

If~$\level(v_d) = \level(v_{d-1})$, then due to Lemma~\ref{lem:genbased}~\ref{itm:genbasedb} we have that $k = d - \type(T) < d$. 
Hence, the rearrangement in~\ref{itm:maubach3b} applies and yields 
\begin{align*}
v_0 = y_{k+1},\qquad v_{\ell-1} = y_d,\qquad v_\ell = y_0,\qquad v_d =y_k. 
\end{align*}
Therefore, Lemma~\ref{lem:maubach}~\ref{itm:maubach3} implies~\ref{itm:genbased2b}.  

To prove~\ref{itm:genbased2d} note that the head has consecutive decreasing generations. 
The only case in which no level jump appears within the head is if~$\type(v_0) = d$.  

It remains to show~\ref{itm:genbased2c}. 
If~$\level(v_d) \neq \level(v_{d-1})$, then we have that~$\head(T) = \simplex{v_0,\dots, v_{d-1}}$ which  by~\ref{itm:genbased2a} has consecutive generation. 
Noting also that $\bse(T) = [v_{d-1},v_d]$, we find that
  \begin{align*}
    \gen(b) = \gen(v_0)+1 = \gen(v_{d-1})+d = \gen(\bse(T))+d.
  \end{align*}
  This proves~$\level(b) = \level(\bse(T))+1$.
  Suppose now that~$\level(v_d)=\level(v_{d-1})$. 
  This arises in case~\ref{itm:maubach3b} only, and thus Lemma~\ref{lem:maubach}~\ref{itm:maubach3} shows that~$\level(v_0) = \level(v_{\ell-1})$. 
  Applying~\ref{itm:genbased2b} shows that
  \begin{align*}
	\level(v_{\ell-1}) = \level(v_\ell) +1 = \level(\bse(T)) +1. 
  \end{align*}
  Combining both observations concludes the proof of~\ref{itm:genbased2c}.
\end{proof}

\noindent A benefit of our reformulation of~$\Bisec{}$ in Algorithm~\ref{algo:genbased} is that we can deduce the following bisection rule for subsimplices using only the structure of the subsimplex. 
The edge of the~$m$-subsimplex~$S=\simplex{v_0,\dots, v_m}$ that is first bisected during~$\Bisec{}$ is called the bisection edge of~$S$ and denoted by $\bse(S)$. 
Recall that~$\tail(S)$ denotes the vertices of~$S$ with level~$\level(v_m)$.

\begin{figure}[h]
\begin{algorithm}[H]
  \label{algo:subsimplex}
  \caption{Bisection of~$m$-simplex in dimension~$d$}
  \SetAlgoLined%
  \SetKwFunction{FuncBisection}{Bisection}
  \FuncBisection{$S$}{%
    
    \KwData{an~$m$-simplex~$S=\simplex{v_0, \dots, v_m}$, for~$m \in \set{1, \ldots, d} \leq d$ (sorted by decreasing vertex generation) }
    
    \KwResult{bisection edge $\bse(T)=\simplex{b_0,b_1}$, generation of the bisection vertex~$b$}
    
    \eIf{$\level(v_m) \neq \level(v_{m-1})$}{
      $\bse(S) =\simplex{v_{m-1} \leveljump v_m}$\tcp*{two oldest vertices}

      $\gen(b) = \gen(v_{m-1}) + d$\;
    }{
      $v_\ell \coloneqq$ youngest vertex of~$\tail(S)$\;
      
      $\bse(S)=\simplex{v_\ell,v_m}$\tcp*{youngest and oldest vertex of tail}

      $\gen(b) = \gen(v_m) + 2d+1 - \type(v_\ell)$\;
    }
  }
\end{algorithm}
\end{figure}

The following lemma shows that Algorithm~\ref{algo:subsimplex} works correctly.  

\begin{lemma}[Bisection of subsimplices]
\label{lem:subsimplex}
  Algorithm~\ref{algo:subsimplex} correctly determines  the bisection edge and bisection vertex of each~$m$-subsimplex~$S$ arising from~$\BisecT$ where~$\tria_0$ is a colored initial triangulation.
\end{lemma}
\begin{proof}
Let~$T =\simplex{v_0,\dots, v_d} \in \tria \in \BisecT$ be a $d$-simplex with bisection edge~$e_T \coloneqq \bse(T)$ and bisection vertex~$b \coloneqq \bsv(e_T)$.  
Suppose~$S = \simplex{\bar{v}_0,\dots, \bar{v}_m}$ is an $m$-subsimplex of~$T$ with~$e_T \subset S$ and with bisection edge~$e_S$ chosen by Algorithm~\ref{algo:subsimplex}. 
We have to show that~$e_T = e_S$ and that Algorithm~\ref{algo:subsimplex} computes the generation of~$b$ correctly.  

  \textit{Case 1.} 
  If~$\level(v_{d-1}) \neq \level(v_d)$, then by Lemma~\ref{lem:genbased}~\ref{itm:genbaseda} the bisection edge of~$T$ is~$e_T = \simplex{v_{d-1} \mid v_d}$. 
  Since by assumption~$e_T \subset S$, we have that $\{v_{d-1},v_{d}\} \subset \{\bar{v}_{1}, \ldots \bar{v}_{m}\}$. 
  As~$S$ is a subsimplex of~$T$ there are no older vertices and hence~$\bar{v}_{m} = v_d$ and~$\bar{v}_{m-1} = v_{d-1}$.  
  Thus,~$\level(\bar{v}_{m-1}) \neq \level (\bar{v}_m)$ and Algorithm~\ref{algo:subsimplex} chooses the same bisection edge~$ \simplex{\bar{v}_{m-1}\mid \bar{v}_m} = \simplex{v_{d-1} \mid v_d}$ for bisecting~$S$.  
  Since~$\level(v_{d-1})\neq \level(v_d)$, we have that~$\head(T) = \simplex{v_0,\dots, v_{d-1}}$. 
  By Lemma~\ref{lem:genbased2}~\ref{itm:genbased2a}~$\head(T)$ has consecutive generations, and hence it follows that
  \begin{align*}
    \gen(b) = \gen(v_0)+1 = \gen(v_{d-1})+d = \gen(\bar{v}_{m-1})+d.
  \end{align*}
  
  \textit{Case 2.} 
  If~$\level(v_{d-1}) = \level(v_d)$, then~$e_T = \simplex{v_{\type(T)},v_d}$ is the bisection edge of~$T$, see Lemma~\ref{lem:genbased}~\ref{itm:genbasedb}. 
  As by assumption~$e_T \subset S$, we have~$\{v_{\type(T)},v_{d}\} \subset \{\bar{v}_{1}, \ldots  \bar{v}_{m}\}$ and $v_d = \bar{v}_{m}$.
  By construction of~$e_T$, the vertices~$v_{\type(T)}$ and~$v_d$ are the youngest and the oldest vertex of~$T$ with level~$\level(v_d)$.   
  Since~$S \subset T$, they are also the youngest and the oldest vertex of~$S$ with level~$\level(\bar{v}_m)$ and
  \begin{align*}
   \type(\tail(T))=\type(v_{\type(T)})=\type(\tail(S)). 
  \end{align*}
  Thus, Algorithm~\ref{algo:subsimplex} chooses the same bisection edge~$e_S=e_T$. Lemma~\ref{lem:genbased2}~\ref{itm:genbased2b} yields
  \begin{align*}
    \gen(v_0) = \gen(v_d) + 2d - \type(\tail(T)).
  \end{align*}
  Hence, with~$\gen(b) = \gen(v_0)+1$ we obtain that
  \begin{align*}
    \gen(b) &= \gen(v_d) + 2d +1 - \type(\tail(T))\\
   & =\gen(v_d) + 2d +1 - \type(\tail(S)).
  \end{align*}
  This completes the proof. 
\end{proof}

\subsection{Fine properties}
\label{subsec:fineProperties}
Algorithm~\ref{algo:genbased} and the routine for subsimplices in Algorithm~\ref{algo:subsimplex} are mere reformulations of the bisection routine by Maubach and Traxler. 
Still, the notion of vertex generations provides structure and a `language' that allows us to show useful properties of the bisection routine, provided the initial triangulation~$\tria_0$ is colored, cf.~Assumption~\ref{ass:initial-triangulation}. 
Here we shall present some of the immediate consequences under this assumption. 

\begin{lemma}
\label{lem:propertiesGenTypeLvl}
Let $T\in \mtree$ be an arbitrary simplex.
\begin{enumerate}
\item 
\label{itm:vertices-gen}
The vertices of the simplex~$T$ have pairwise different generations. 
\item 
\label{itm:oldestonbse}
The oldest vertex~$\widetilde{v}\in \vertices(T)$ of~$T$ is on the bisection edge of~$T$. 
\item 
\label{itm:oldest-edge}
The simplex~$T$ has a unique oldest edge~$\widetilde{e}\in\edges(T)$, and it satisfies  
\begin{align*} 
	  \gen(\widetilde{e})  \leq  \gen(T)-d+1.
\end{align*}
\item 
\label{itm:simplex-edge} 
The level of each edge~$e \in \edges(T)$ satisfies 
\begin{align*}
\level(e) \leq \level(T) \leq \level(e )+1.
\end{align*}
\item  \label{itm:simplex-refEdge}
The level of the bisection edge~$\bse(T) \in \edges(T)$ of~$T$ is
\begin{align*}
\level(\bse(T)) = \begin{cases}
\level(T) - 1 & \text{if }\type(T) < d,\\
\level(T) & \text{if }\type(T) = d.
\end{cases}
\end{align*}
\item 
\label{itm:refe-typeOneVertex}
The bisection edge of $T$ has~$\type(\bse(T)) = 1$ if and only if~$\type(T) = d$. 
\item 
\label{itm:simplex-face}
Let~$F$ be a face of~$T$. 
Then we have that
\begin{align*}
\gen(F) \leq \gen(T) \leq \gen(F)+1.
\end{align*}
\end{enumerate}
\end{lemma}
\begin{proof}
The claims in~\ref{itm:vertices-gen}--\ref{itm:oldestonbse} follow directly by Algorithm~\ref{algo:genbased} and Lemma~\ref{lem:genbased}. 

\textit{Proof of~\ref{itm:oldest-edge}}. 
Since the vertices in~$T= \simplex{v_0, \ldots, v_d} \in \mtree$ are sorted by decreasing generation, the unique oldest edge is~$\widetilde{e} = \simplex{v_{d-1},v_d}$. 
Since the generations of the vertices are pairwise distinct by~\ref{itm:vertices-gen}, it follows that
 \begin{align*}
 &\gen(\widetilde{e}) + (d-1) =  \gen(v_{d-1}) + (d-1) \leq \gen(v_0) = \gen(T).
  \end{align*}
  
\textit{Proof of~\ref{itm:simplex-edge}}. 
Let~$T = \simplex{v_0,\dots,v_d}\in \mtree$.
Any edge~$e \in \edges(T)$ contains at least one vertex that is younger than~$v_{d-1}$ or has the same generation. 
Since we have that~$\level(v_0) \leq \level(v_{d-1})+1$, by Lemma~\ref{lem:genbased2}~\ref{itm:genbased2b} and~\ref{itm:genbased2d} we obtain
\begin{align*}
&\level(e) \leq \level(T) = \level(v_0) \leq \level(v_{d-1})+1 \leq \level(e)+1.
\end{align*}

\textit{Proof of~\ref{itm:simplex-refEdge}}. 
Let~$T = \simplex{v_0, \ldots, v_d}$ be the simplex under consideration. 
Distinguishing the cases~$\level(v_{d-1}) = \level(v_d)$ and~$\level(v_{d-1}) \neq \level(v_d)$   with Lemma~\ref{lem:genbased}~\ref{itm:genbaseda} and~\ref{itm:genbasedb} in combination with statements on the levels in~$\head(T)$ by Lemma~\ref{lem:genbased2}~\ref{itm:genbased2b} and~\ref{itm:genbased2d} yield the claim. 
  
\textit{Proof of~\ref{itm:refe-typeOneVertex}}. 
By~\ref{itm:simplex-refEdge} we have~$\type(T)= d$ if and only if one has $\level(\bse(T)) = \level(T)$. 
This can occur only if~$T$ is of the form~$T = \simplex{v_0, \ldots, v_{d-1}\leveljump v_d}$ with~$\bse(T) = \simplex{v_{d-1} \leveljump v_d}$, and if there is no level jump in~$\head(T)$. 
This is equivalent to~$\type(v_0) = d$, because the vertices in~$\head(T)$ are of consecutive generation by Lemma~\ref{lem:genbased2}~\ref{itm:genbased2a}. 

\textit{Proof of~\ref{itm:simplex-face}}. 
Let~$T = \simplex{v_0, \ldots, v_d}$. 
If $\head(T)$ consists of at least two vertices, then Lemma~\ref{lem:genbased} and the consecutive generation numbers in~$\head(T)$ yield that~$\gen(v_{0}) = \gen(v_1)+1$. 
Otherwise, if~$\head(T)$ consists only of~$v_0$, then the vertices~$v_1, \ldots, v_d$ are of the same level and have the types~$d, d-1, \ldots, 1$. 
By Lemma~\ref{lem:genbased2}~\ref{itm:genbased2b} applied for~$\ell = 1$ it also follows that~$\gen(v_0) = \gen(v_1)+1$. 
Since~$F$ contains at least one of the vertices~$v_0$ and~$v_1$, in both cases it follows that 
  \begin{align*}  
   &\gen(T)-1 = \gen(v_0)-1 = \gen(v_{1}) \leq \gen(F).
  \end{align*}
Since~$\gen(F) \leq \gen(T)$ holds by definition, this finishes the proof.
\end{proof}
  
\begin{lemma}[Vertices on bisection edges]
\label{lem:bisec-edge-1}
For a simplex~$T \in \mtree$ with child~$T_1$ let~$v\in \vertices(T) \cap \vertices(T_1)$ be a vertex on the bisection edge of~$T$. 
Then~$v$ is on the bisection edge of~$T_1$.
\end{lemma}
\begin{proof}
	Let~$T = \simplex{v_0, \ldots, v_d}$ and~$v \in \bse(T)$. 
	Let us first note that the oldest vertex of a simplex is always contained in the bisection edge, cf.~Lemma~\ref{lem:genbased}. 
	
	\textit{Case 1}. 
	Let~$v$ be the oldest vertex of~$T$. 
	For the child~$T_1$ of~$T$ that contains the vertex~$v$, it is also the oldest vertex and thus it is on the bisection edge of~$T_1$.   
	
	\textit{Case 2}. 
	If~$v$ is not the oldest vertex, let us distinguish the following cases: 
	
	If~$T = \simplex{v_0, \ldots, v_{d-1} \leveljump v_d}$, then by Lemma~\ref{lem:genbased}~\ref{itm:genbaseda} the bisection edge is~$\bse(T) = \simplex{v_{d-1}\leveljump v_d}$, i.e.,~$v  = v_{d-1}$. 
	The child~$T_1$ that contains~$v = v_{d-1}$ does not contain~$v_d$. 
	Hence~$v = v_{d-1}$ is the oldest vertex of~$T_1$ and consequently on its bisection edge. 

	If~$T  = \simplex{v_0, \ldots, v_{\ell-1} \leveljump v_\ell, \ldots,  v_d}$ with~$\ell < d$, then by Lemma~\ref{lem:genbased}~\ref{itm:genbasedb} the bisection edge is~$\bse(T) = \simplex{v_{\ell} , v_d}$, i.e.,~$v  = v_{\ell}$. 
	Then the child, that contains~$v = v_\ell$ is of the form~$T_1 = \simplex{b, v_0, \ldots, v_{\ell-1} \leveljump v_{\ell}, \ldots, v_{d-1}}$. 
	If~$\ell = d-1$, then~$v_{\ell}$ is the oldest vertex, and thus on the bisection edge of~$T_1$. 
	Otherwise,~$\tail(T_1)$ contains at least two vertices, and~$v = v_\ell$ is the youngest vertex of it. 
	Thus, by Lemma~\ref{lem:genbased}~\ref{itm:genbaseda} the bisection edge~$\bse(T_1) = \simplex{v_{\ell},v_{d-1}}$ contains~$v = v_\ell$. 
\end{proof}
  
\begin{lemma}[Type one edges]
\label{lem:typeOneEdges}
  Any simplex~$T\in \mtree$ has exactly~$d-\type(T)+1$ edges of type one. 
  Their levels coincide with the level of~$T$. 
\end{lemma}
\begin{proof}
  Let~$T = \simplex{v_0,\dots,v_d} \in \mtree$ be a $d$-simplex.  Lemma~\ref{lem:genbased} and the consecutively decreasing generations of vertices in~$\head(T)$ show that there is exactly one vertex~$v\in \vertices(\head(T))$ with $\type(v) = 1$. 
  It satisfies~$v = v_{\type(T) - 1}$ and so~$\level(v) = \level(T)$. 
  Since all vertices in the tail are of the same level, a type one vertex~$v \in \vertices(\tail(T))$ must satisfy~$v = v_d$. 
  Therefore, all type one edges in~$T$ are of the form $\simplex{v_{\type(T) - 1},v_\ell}$ with~$\ell = \type(T),\dots,d$ and their level is~$\level(T)$.
\end{proof}
  
\begin{lemma}[Type $d$ simplices]
\label{lem:typeDsimplices}
Each edge~$e\in \edges(\mtree)$ belongs to a type $d$ simplex~$T \in \mtree$ in the sense that
\begin{align}
\label{eq:TypeDsim}
e\in \edges(T)\qquad\text{and}\qquad \gen(T) = \level(e)d.
\end{align}
\end{lemma}
\begin{proof}
Let~$e\in \edges(\mtree)$ and let~$T \in \mtree$ denote the simplex that is created at the same time as~$e$ in the sense that~$e\in \edges(T)$ and~$\gen(T) = \gen(e)$. 
If~$\type(T) = d$, then $T$ satisfies \eqref{eq:TypeDsim}. 
If~$\type(T)<d$, then by Lemma~\ref{lem:propertiesGenTypeLvl}~\ref{itm:simplex-refEdge} we have that~$\level(\bse(T)) = \level(T)-1$, and consequently~$e \neq \bse(T)$. 
Thus one of the children~$T'\in \children(T)$ contains~$e \in \edges(T')$. 
This allows us to proceed inductively until we reach a simplex of type~$d$. 
\end{proof}

\subsection{Alternative notion of generation}
\label{subsec:genSharp}
The previous subsection shows that the notions of generation, level, and type entail a lot of structure. 
However, there are some features that cause difficulties in our analysis. 
For example, when bisecting an edge for the first time, the type of its children may differ from the type of the edge. 
We remedy these drawbacks by introducing an alternative notion of generation which is of particular use  in the following Section~\ref{sec:locGenDiff}. 
Throughout this subsection we assume that~$\tria_0$ is colored according to Assumption~\ref{ass:initial-triangulation}.

\begin{definition}[$\sharp$-Generation]
\label{def:gsharp}
 For any edge~$e\in \edges(\mtree)$ with bisection vertex~$b$ we set~$\gensharp(e) \coloneqq \gen(b)$. 
  More generally, for an~$m$-simplex~$S$ with~$m\geq 1$ and bisection edge~$e$ we define~$\gensharp(S) \coloneqq \gensharp(e)$.
\end{definition}
If an~$m$-subsimplex~$S$ of~$T \in \tria$ is bisected into two $m$-simplices~$S_1$ and~$S_2$ by Algorithm~\ref{algo:genbased}, we call them children of~$S$.  
They satisfy
\begin{align}
 \label{eq:gensharp-children}
  \gen(S_1) = \gen(S_2) = \gensharp(S).
\end{align}
For a~$d$-simplex~$T \in \tria$ with children~$T_1$ and $T_2$ we obtain even
\begin{align}
  \label{eq:gensharp-childrenT}
  \gen(T_1) = \gen(T_2) = \gensharp(T) = \gen(T)+1.
\end{align}
In analogy to~\eqref{eq:def-type-level}, for an~$m$-simplex~$S$ with~$m \geq 1$ we define
\begin{align}
  \label{eq:def-type-level-sharp}
  \begin{aligned}
    \levelsharp(S) &\coloneqq \left\lceil \frac{\gensharp(S)}{d} \right\rceil,
    \\
    \typesharp(S) &\coloneqq \gensharp(S) - d\, \big(\levelsharp(S) -1\big) \in \set{1,\dots,d}.
  \end{aligned}
\end{align}
As before, this leads to the identity 
\begin{align*}
\gensharp(\bigcdot) = d\, \big(\levelsharp(\bigcdot) -1\big) + \typesharp(\bigcdot).
\end{align*}

\begin{lemma}
	\label{lem:PropSharp}
  Let~$T\in \mtree$ be arbitrary.
  \begin{enumerate}
  \item 
  \label{itm:lvlSharpEdge}
  For every edge~$e \in \edges(T)$ we have that~$ \levelsharp(e) = \level(e)+1$. 
  \item 
  \label{itm:gsharp-oldest} 
  For an $m$-subsimplex~$S$ of~$T$ with~$m \geq 2$ the bisection edge~$\bse(S)$ is the $\sharp$-oldest edge of $S$, i.e.,
  \begin{align*}
\gensharp(\bse(S)) < \gensharp(e')\quad \text{ for all }e' \in \edges(S) \setminus \set{\bse(S)}. 
  \end{align*}
  \item 
  \label{itm:gsharp-oldest2}
  We have that~$\gen(T) + 1 = \gensharp(\bse(T))$.
  \item 
  \label{itm:gsharp-edge-child} 
  If an edge~$e \in \edges(T)$ is bisected into edges~$e_1$ and $e_2$, then one has that
    \begin{alignat*}{2}
      \gensharp(e_1) &=
      \gensharp(e_2) &&= 
      \gensharp(e) +d,
      \\
      \levelsharp(e_1) &=
      \levelsharp(e_2) &&= 
      \levelsharp(e) +1,\\
      \typesharp(e_1) &=
      \typesharp(e_2) &&= 
      \typesharp(e).
    \end{alignat*}
  \end{enumerate}
\end{lemma}
\begin{proof}
  The claim in~\ref{itm:lvlSharpEdge} is equivalent to the statement in Lemma~\ref{lem:genbased2}~\ref{itm:genbased2c}.
  The claim in~\ref{itm:gsharp-oldest} follows from the fact that all other edges are not split in the bisection of~$S$ but later, during the bisection of the descendants of~$S$. 
  The property in \ref{itm:gsharp-oldest2} follows by definition.
  Let us finally prove \ref{itm:gsharp-edge-child}.  
  Let~$e = [v_0,v_1] \in \edges(T)$ have the bisection vertex~$b = \bsv(e)$. 
  Then its children read~$e_1=\simplex{b \leveljump v_0}$ and $e_2=\simplex{b \leveljump v_1}$. 
  Thus, the first case in Algorithm~\ref{algo:subsimplex} applies and proves the claim.
\end{proof}

\begin{remark}[$\gensharp$ and edge length]
  \label{rem:longest-edge-sharp}
  There is a relation between $\gensharp$ and the interpretation of~$\Bisec{}$ as longest edge refinement as in Remark~\ref{rem:longest-edge}. 
  For this purpose let~$T_0$ denote the standard Kuhn simplex and~$\tria_0 \coloneqq \set{T_0}$ or~$\tria_0$ can be chosen as the Kuhn triangulation of a cube, cf.~Remark~\ref{rem:algo-kuhn}.  
  For any edge~$e \in \edges(\mtree)$ we may compute $\norm{e}_\infty = 2^{1-\levelsharp(e)}$ and $\norm{e}_1 = 2^{1-\levelsharp(e)} (d+1-\typesharp(e))$. 
  With the norm~$\normm{\cdot}$ as defined in Remark~\ref{rem:longest-edge} we obtain
  \begin{align*}
    \normm{e} &= \norm{e}_\infty + \frac 1d \norm{e}_1 = 2^{1-\levelsharp(e)} \Big( 1 + \frac{d+1-\typesharp(e)}{d}\Big).
  \end{align*}
  The~$\sharp$-oldest edge corresponds to the~$\normm{\cdot}$-longest edge and is the bisection edge.
\end{remark}
  
  \begin{lemma}[$\gensharp$ in a simplex]
  	\label{lem:gensharp-patch}
  	For any~$T \in \mtree$, for any of its vertices~$v \in \vertices(T)$, and for all edges~$e, e' \in \edges(T)$ with~$v \in e \cap e'$ the following estimate holds
  	\begin{align*}
  		\abs{\gensharp(e) - \gensharp(e')} \leq d-1. 
  	\end{align*}
  \end{lemma}
  \begin{proof}
  	Let~$T\in \mtree$ and~$v \in \vertices(T)$ be an arbitrary but fixed vertex of~$T$. 
  	There are~$d$ edges in~$\edges(T)$ containing $v$. 
  	They can be sorted by strictly increasing~$\gensharp$ since the edges and their children containing~$v$ have non-empty intersection with the children of~$T$ containing~$v$, and since in each bisection only one edge can be bisected. 
  	Hence, let~$e_1, \ldots, e_d\in \edges(T)$ be the edges of~$T$ containing~$v$ labeled such that 
  	\begin{align*}
  		\gensharp(e_1)< \ldots < \gensharp(e_d).
  	\end{align*}
  	Let~$T_1$ denote the descendant of~$T$ that contains~$v$ and that has~$\gensharp(T) = \gensharp(e_1)$, i.e.,~$\bse(T_1) = e_1$. 
  	By Lemma~\ref{lem:bisec-edge-1} the vertex~$v$ is on the bisection edge of any further descendant of~$T_1$. 
  	More specifically, for the sequence of descendants~$T_1,T_2, \ldots$ of~$T$ containing~$v$ with~$T_{k+1} \in \children(T_{k})$ for~$k  \in \mathbb{N}$, we have that~$v \in \bse(T_k)$, for any~$k \in \mathbb{N}$. Consequently, we have that~$\gensharp(\bse(T_{k+1})) =\gensharp(\bse(T_k))+1$, for any~$k \in \mathbb{N}$. 
  	Note that all of these bisection edges are descendants of~$e_1, \ldots, e_d$.
  	
  	If there would be a~$j \in \{1, \ldots, d-1\}$ such that~$\gensharp(e_j)+1<\gensharp(e_{j+1})$, then we could choose the smallest such, i.e.,~$\gensharp(e_k) = \gensharp(e_1) + (k-1)$ for any~$k = 1, \ldots, j$. 
  	Then, there would be a~$k \in \{1, \ldots, j\}$ and a child~$e_k'$of $e_k$ with 
  	\begin{align*}
  		\gensharp(e_j)+1  = \gensharp(e_k') < \gensharp(e_{j+1}),
  	\end{align*}
  	which means that~$\gensharp(e_k') - \gensharp(e_k) \leq j < d-1$.  This, however, contradicts  Lemma~\ref{lem:PropSharp}~\ref{itm:gsharp-edge-child}, which states that 
  	\begin{align*}
  		\gensharp(e_k') = \gensharp(e_k) + d \qquad \text{ for any } k = 1, \ldots, d.  
  	\end{align*}
  	This suffices to conclude that~$\gensharp(e_k) = \gensharp(e_1) + (k-1)$ for any~$k \in \{1, \ldots, d\}$. 
  	In turn, this implies that we have for any~$e,e' \in \edges(T)$ with~$v \in e \cap e'$ 
  	\begin{align*}
&  		\abs{\gensharp(e) - \gensharp(e')} \leq d-1. \qedhere
  	\end{align*}
  \end{proof}

\begin{lemma}[Bisection edge]
\label{lem:bisec-edge-2}
Let~$T_1, T_2 \in \mtree$ be simplices such that~$T_1 \cap T_2$ is a non-empty subsimplex of~$T_1$ and~$T_2$.
  \begin{enumerate}
  \item 
  \label{itm:bisec-edge-a}
    If~$\bse(T_1) \subset T_2$ and~$\bse(T_2) \cap T_1 \neq \emptyset$, then we have~$\bse(T_1) \cap \bse(T_2) \neq \emptyset$. 
  \item 
  \label{itm:bisec-edge-b}
   If additionally~$\gen(T_1) = \gen(T_2)$, then we have~$\bse(T_1) = \bse(T_2)$. 
  \end{enumerate} 
\end{lemma}
\begin{proof}
  To prove~\ref{itm:bisec-edge-a} we proceed by contradiction. 
  For this purpose we assume that~$\bse(T_1) \cap \bse(T_2) =\emptyset$.  
  By the assumption~$\bse(T_2) \cap T_1 \neq \emptyset$ there exists a vertex~$v_0 \in \bse(T_2) \cap T_1$.  
  The property~$\bse(T_1) \cap \bse(T_2) = \emptyset$ ensures that~$v_0 \notin \bse(T_1) = [v_1,v_2] \subset T_1 \cap T_2$. 
  Thus, the three vertices~$v_0, v_1, v_2$ are contained in~$T_1 \cap T_2$. 
Then, Lemma~\ref{lem:PropSharp}~\ref{itm:gsharp-oldest} yields that
  \begin{align}
  	\label{est:bisection-edge}
    \gensharp([v_1,v_2]) < \min_{i = 1,2}\gensharp([v_0, v_i])). 
  \end{align}
  Let~$T_2'$ be the youngest descendant of~$T_2$ that contains~$[v_0, v_1, v_2]$, i.e., in particular we have that~$\bse(T_2') \subset \edges([v_0, v_1, v_2])$.  
  Lemma~\ref{lem:bisec-edge-1} and the fact that~$v_0\in \bse(T_2)$ yield~$v_0 \in \bse(T_2')$. 
  Thus, we have
  \begin{align*}
    \min_{i = 1,2}(\gensharp([v_0, v_i])) < \gensharp([v_1, v_2]). 
  \end{align*}
  This contradicts~\eqref{est:bisection-edge} and hence proves~\ref{itm:bisec-edge-a}. 
  
  To prove~\ref{itm:bisec-edge-b} let~$\gen(T_1)=\gen(T_2)$ and thus by Lemma~\ref{lem:bisec-edge-2}~\ref{itm:gsharp-oldest2} we have that  
  \begin{align*}
\gensharp(\bse(T_1)) = \gen(T_1)+1 =  \gen(T_2)+1 = \gensharp(\bse(T_2)).   	
 \end{align*}
  Since~$\bse(T_1)$ is contained in~$T_2$ and~$\gensharp(\bse(T_2)) = \gen(T_2)+1$, it is also the unique bisection edge of~$T_2$, see Lemma~\ref{lem:PropSharp}~\ref{itm:gsharp-oldest}.
\end{proof}

\section{Local generation and level estimates}
\label{sec:locGenDiff}
In this section we use the properties derived in Section~\ref{sec:bisection-algorithm} to estimate the level differences of edges~$e,e'\in\edges(\tria)$ with~$\tria\in \BisecT$ connected by a vertex~$v = e\cap e' \in \vertices(\tria)$. 
Our analysis distinguishes the following types of vertices depending on the dimension of a corresponding initial subsimplex that contains it.
Recall that (sub)simplices are closed. 
\begin{definition}[Macro vertices]
	\label{def:MacroVertex}
	Let~$v\in \vertices(\mtree)$ be a vertex.
	\begin{enumerate}
		\item 
		If~$v\in \vertices(\tria_0)$, then we call the vertex a \emph{$0$-macro vertex}.
		\item 
		If~$v$ is in an~$n$-subsimplex but not in an $(n\!-\!1)$-subsimplex of~$\tria_0$ with~$n \in \lbrace 1 ,\dots, d\rbrace$, then we call the vertex an \emph{$n$-macro vertex}. 
		\item 
		If~$v$ is an~$n$-macro vertex with~$n\leq d-2$, we call~$v$ \emph{critical}. 
		Otherwise, i.e., if~$v$ is an~$n$-macro vertex with~$n \geq d-1$ we refer to~$v$ as \emph{non-critical}.  
	\end{enumerate}
	For an~$n$-macro vertex we refer to~$n$ as the dimension of~$v$. 
\end{definition}
By this definition any vertex $v\in \vertices(\mtree)$ has a unique dimension~$n \in \{0, \ldots, d\}$, depending on the dimension of a corresponding initial subsimplex that contains it.
In particular, non-critical vertices lie
\begin{enumerate}[label = (\roman*)]
	\item in the interior of an initial~$d$-simplex $T_0 \in \tria_0$ or 
	\item in the relative interior of a~$(d-1)$-subsimplex of an initial~$d$-simplex~$T_0 \in \tria_0$. 
\end{enumerate}
Any non-critical macro vertex is contained in at most two initial~$d$-simplices, while for example in dimension~$d = 3$ vertices on initial edges may have many more adjacent initial simplices. 

The objective of this section is to prove for any edge~$e,e'\in\edges(\tria)$ with~$n$-macro vertex~$v = e\cap e'$ and triangulation~$\tria\in \BisecT$ a level estimate of the form 
\begin{align*}
	\level(e) - \level(e') \leq c_n.
\end{align*}
Here~$c_n\in \mathbb{N}$ is a constant depending on~$n\in \{0, \ldots, d\}$. 
We shall see that for non-critical vertex dimension the constant~$c_n$ is smaller or equal to two. 
To achieve this, we first consider vertices in the Kuhn simplex in Section~\ref{subsec:gener-estim-kuhn} and estimate the generation differences of edges adjacent to it. 
Then, in Section~\ref{subsec:gen-diff-vtx} we derive upper bounds for differences with respect to $\gensharp$ for edges intersecting in general (macro) vertices. 
Finally, in Section~\ref{subsec:mainResultLocal-level} we transfer this to level estimates.

\subsection{Generation estimates in Kuhn simplices}
\label{subsec:gener-estim-kuhn}
In this subsection we consider the special case of Kuhn triangulations. 
This will be the basic building block for the generation estimates of edges intersecting in non-critical vertices. 

Let~$\tria_0$ denote the Kuhn partition of the cube~$[0,1]^d$ into~$d!$~Kuhn simplices~$T_\pi$ as introduced in Remark~\ref{rem:algo-kuhn}. 
For initial vertices~$v \in \set{0,1}^d$ we choose the coloring~$\frc(0) \coloneqq d$ and~$\frc(v) \coloneqq \norm{v}_1-1$ for $v \neq 0$. 
Let~$\fre_1,\dots, \fre_d\in \mathbb{R}^d$ denote the canonical unit vectors. 
With the convention~$\gen(v) = - \frc(v)$ in \eqref{eq:gen-initial}, for any permutation~$\pi\colon \lbrace 1,\dots,d\rbrace \to \lbrace 1,\dots,d\rbrace$ we obtain the sorted simplex
\begin{align}
\label{eq:kuhn-simplex-sorted}
	\begin{aligned}
		T_\pi &= [0, \fre_{\pi(1)}, \fre_{\pi(1)}+\fre_{\pi(2)}, \dots, \fre_{\pi(1)}+\dots+\fre_{\pi(d)}]
		\\
		& = \simplex{\fre_{\pi(1)}, \fre_{\pi(1)}+\fre_{\pi(2)}, \dots, \fre_{\pi(1)}+\dots+\fre_{\pi(d)}\leveljump 0}.
	\end{aligned}
\end{align}
\begin{figure}[h]
	\begin{tikzpicture}[%
		x={(0:1cm)}, y={(30:0.5cm)}, z={(90:1.0cm)}, scale=12/5]
		\draw[] (0,0,0) -- (1,0,0) node[midway,below] {\color{gray}0(3)};    
		\draw[] (1,0,0) -- (1,1,0) node[midway,below right] {\color{gray}0(3)};
		\draw[] (1,1,0) -- (1,1,1) node[midway,right] {\color{gray}-1(2)};
		\draw[] (0,0,0) -- (1,1,1) node[midway,above left] {\color{gray}-2(1)};		
		\draw[] (1,1,1) -- (1,0,0) node[midway,left] {{\color{gray}0(3)}};
		\draw[dashed] (0,0,0) -- (1,1,0) node[midway,above] {\color{gray}-1(2)};
		\draw (0,0,0) node[anchor=east] {-3};
		\draw (1,1,1) node[anchor=west] {-2};
		\draw (1,0,0) node[anchor=north west] {0};
		\draw (1,1,0) node[anchor=west] {-1}; 
		
		\begin{scope}[xshift = 2.5cm]
			x={(0:1cm)}, y={(30:0.5cm)}, z={(90:1.0cm)}, scale=12/5]
			\draw[] (0,0,0) -- (1,0,0) node[midway,below] {\color{gray}3};    
			\draw[] (1,0,0) -- (1,1,0) node[midway,below right] {\color{gray}3};
			\draw[] (1,1,0) -- (1,1,1) node[midway,right] {\color{gray}3};
			\draw[] (0,0,0) -- (1,1,1) node[midway,above left] {\color{gray}1};
			\draw[] (1,1,1) -- (1,0,0) node[midway,left] {{\color{gray}2}};
			\draw[dashed] (0,0,0) -- (1,1,0) node[midway,above] {\color{gray}2};
			\draw (0,0,0) node[anchor=east] {-3};
			\draw (1,1,1) node[anchor=west] {-2};
			\draw (1,0,0) node[anchor=north west] {0};
			\draw (1,1,0) node[anchor=west] {-1};		
		\end{scope}
	\end{tikzpicture}
	\caption{Kuhn simplex with vertex generation (black), edge generations and types (gray, left) and edge $\sharp$-generation (gray, right).}
\end{figure}
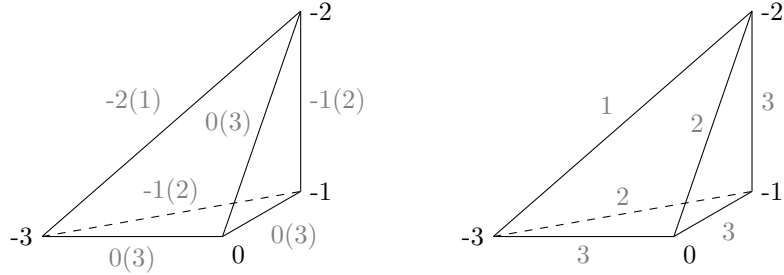%
With this generation structure one \emph{full refinement} ($d$~successive uniform refinements) leads to the Tucker--Whitney triangulation that consists of reflected Kuhn cubes of half the size, see Remark~\ref{rem:algo-kuhn}. 
This process can be repeated. 
Any vertex~$v\in\vertices(\mtree)$ generated in this process attains its maximal \emph{edge valence} (the number of edges containing~$v$) after one further full refinement. 
Moreover, after such a full refinement every interior vertex~$v$ is surrounded by~$2^d$ smaller Kuhn cubes as depicted for the vertex~$(\frac 12, \frac 12, \frac 12)$ in Figure~\ref{fig:kuhnedges3D}. 
Additional refinements neither increase the edge valence of~$v$ nor do they change the direction of the edges containing~$v$. 
This is because only edges emanating from~$v$ are bisected, cf.~Lemma~\ref{lem:bisec-edge-1}. 
Hence, the maximal edge valence of every vertex~$v$ is at most~$3^d-1$ and any~$v$ is contained in at most~$d!\, 2^d$ many $d$-simplices. 
For interior vertices these values are attained. 

\begin{remark}[ndof]
Similar arguments are used in~\cite[Sec.~3.1]{DieningStornTscherpel22} and in~\cite{ST.2023.A} to compare the number of degrees of freedom of mixed finite element space pairs.
\end{remark}

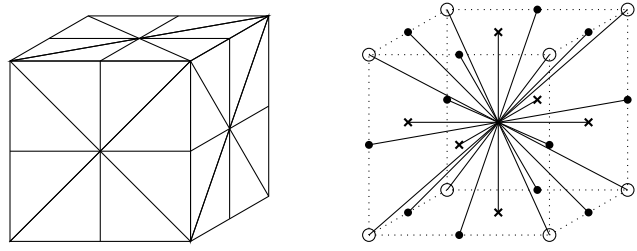
\begin{figure}[ht!]
	\begin{tikzpicture}[%
		x={(0:1cm)}, y={(30:0.5cm)}, z={(90:1.0cm)}, scale=12/5]
		\draw[] (0,0,0) -- (1,0,0) -- (1,0,1) -- cycle;
		\draw[] (0,0,0) -- (0,0,1) -- (1,0,1) -- cycle;
		\draw[] (1,0,0) -- (1,1,0) -- (1,1,1) -- cycle;
		\draw[] (1,0,0) -- (1,0,1) -- (1,1,1) -- cycle;
		\draw[] (0,0,1) -- (1,0,1) -- (1,1,1) -- cycle;
		\draw[] (0,0,1) -- (0,1,1) -- (1,1,1) -- cycle;
		\draw[] (1/2,0,1/2) -- (0,0,1);
		\draw[] (1/2,0,1/2) -- (1,0,0);
		\draw[] (1,1/2,1/2) -- (1,0,1);
		\draw[] (1,1/2,1/2) -- (1,1,0);
		\draw[] (1/2,1/2,1) -- (1,0,1);
		\draw[] (1/2,1/2,1) -- (0,1,1);
		\draw[] (1/2,0,1/2) -- (0,0,1/2);
		\draw[] (1/2,0,1/2) -- (1,0,1/2);
		\draw[] (1/2,0,1/2) -- (1/2,0,0);
		\draw[] (1/2,0,1/2) -- (1/2,0,1);
		\draw[] (1,1/2,1/2) -- (1,0,1/2);
		\draw[] (1,1/2,1/2) -- (1,1,1/2);
		\draw[] (1,1/2,1/2) -- (1,1/2,0);
		\draw[] (1,1/2,1/2) -- (1,1/2,1);
		\draw[] (1/2,1/2,1) -- (0,1/2,1);
		\draw[] (1/2,1/2,1) -- (1,1/2,1);
		\draw[] (1/2,1/2,1) -- (1/2,0,1);
		\draw[] (1/2,1/2,1) -- (1/2,1,1);
	\end{tikzpicture}
	\hspace{10mm}
	\begin{tikzpicture}[%
		x={(0:1cm)}, y={(30:0.5cm)}, z={(90:1.0cm)}, scale=12/5]
		\draw[dotted] (0,0,0) -- (1,0,0) -- (1,0,1);
		\draw[dotted] (0,0,0) -- (0,0,1) -- (1,0,1);
		\draw[dotted] (1,0,0) -- (1,1,0) -- (1,1,1);
		\draw[dotted] (1,0,1) -- (1,1,1);
		\draw[dotted] (0,0,1) -- (0,1,1) -- (1,1,1);
		\draw[dotted] (0,0,0) -- (0,1,0) -- (1,1,0);
		\draw[dotted] (0,1,0) -- (0,1,1);    
		
		\draw (  0,  0,  0) -- (1/2,1/2,1/2);
		\draw (  0,  0,1/2) -- (1/2,1/2,1/2);
		\draw (  0,  0,  1) -- (1/2,1/2,1/2);
		\draw (  0,1/2,  0) -- (1/2,1/2,1/2);
		\draw (  0,1/2,1/2) -- (1/2,1/2,1/2);
		\draw (  0,1/2,  1) -- (1/2,1/2,1/2);
		\draw (  0,  1,  0) -- (1/2,1/2,1/2);
		\draw (  0,  1,1/2) -- (1/2,1/2,1/2);
		\draw (  0,  1,  1) -- (1/2,1/2,1/2);
		\draw (1/2,  0,  0) -- (1/2,1/2,1/2);
		\draw (1/2,  0,1/2) -- (1/2,1/2,1/2);
		\draw (1/2,  0,  1) -- (1/2,1/2,1/2);
		\draw (1/2,1/2,  0) -- (1/2,1/2,1/2);
		\draw (1/2,1/2,  1) -- (1/2,1/2,1/2);
		\draw (1/2,  1,  0) -- (1/2,1/2,1/2);
		\draw (1/2,  1,1/2) -- (1/2,1/2,1/2);
		\draw (1/2,  1,  1) -- (1/2,1/2,1/2);
		\draw (  1,  0,  0) -- (1/2,1/2,1/2);
		\draw (  1,  0,1/2) -- (1/2,1/2,1/2);
		\draw (  1,  0,  1) -- (1/2,1/2,1/2);
		\draw (  1,1/2,  0) -- (1/2,1/2,1/2);
		\draw (  1,1/2,1/2) -- (1/2,1/2,1/2);
		\draw (  1,1/2,  1) -- (1/2,1/2,1/2);
		\draw (  1,  1,  0) -- (1/2,1/2,1/2);
		\draw (  1,  1,1/2) -- (1/2,1/2,1/2);
		\draw (  1,  1,  1) -- (1/2,1/2,1/2);
		
		\draw (  0,  0,  0)circle (1pt);
		\draw (  0,  0,  1) circle (1pt);
		\draw (  0,  1,  0) circle (1pt); 
		\draw (  0,  1,  1) circle (1pt);
		\draw (  1,  0,  0)  circle (1pt);
		\draw (  1,  0,  1) circle (1pt);
		\draw (  1,  1,  0)  circle (1pt);
		\draw (  1,  1,  1)  circle (1pt);
		
		\draw (  0,1/2,1/2)  pic[thick] {cross=2pt};
		\draw (1/2,  0,1/2) pic[thick] {cross=2pt};		
		\draw (1/2,1/2,  0) pic[thick] {cross=2pt};
		\draw (1/2,1/2,  1) pic[thick] {cross=2pt};		
		\draw (1/2,  1,1/2) pic[thick] {cross=2pt};
		\draw (  1,1/2,1/2) pic[thick] {cross=2pt};
		
		\node[circle = 0.5pt,fill, inner sep=0pt,minimum size=3pt] at (  0,  0,1/2) {};
	\node[circle = 0.5pt,fill, inner sep=0pt,minimum size=3pt] at (  0,1/2,  0) {};
	\node[circle = 0.5pt,fill, inner sep=0pt,minimum size=3pt] at (  0,1/2,  1) {};
	\node[circle = 0.5pt,fill, inner sep=0pt,minimum size=3pt] at (  0,  1,1/2) {};
	
	\node[circle = 0.5pt,fill, inner sep=0pt,minimum size=3pt] at (1/2,  0,  0) {};
	\node[circle = 0.5pt,fill, inner sep=0pt,minimum size=3pt] at (1/2,  0,  1) {};
	\node[circle = 0.5pt,fill, inner sep=0pt,minimum size=3pt] at (1/2,  1,  0) {};
	\node[circle = 0.5pt,fill, inner sep=0pt,minimum size=3pt] at  (1/2,  1,  1) {};
	
	\node[circle = 0.5pt,fill, inner sep=0pt,minimum size=3pt] at (  1,  0,1/2) {};
	\node[circle = 0.5pt,fill, inner sep=0pt,minimum size=3pt] at (  1,1/2,  0) {};
	\node[circle = 0.5pt,fill, inner sep=0pt,minimum size=3pt] at (  1,1/2,  1) {};
	\node[circle = 0.5pt,fill, inner sep=0pt,minimum size=3pt] at  (  1,  1,1/2) {};	
\end{tikzpicture}
	\caption{Partition of Kuhn cube~$[0,1]^3$ after~$d=3$~uniform refinements with exterior edges (left) and interior edges (right). 
		End points of edges of~$\typesharp$ equal~$1$ are marked with a circle, of~$\typesharp$ equal~$2$ with a bullet and of~$\typesharp$ equal~$3$ with a cross.}
	\label{fig:kuhnedges3D}
\end{figure}
We say that a vertex~$v\in\vertices(\tria)$ with~$\tria\in \BisecT $ has \emph{full valence} if further refinements do not increase its edge valence. 
Thanks to Lemma~\ref{lem:bisec-edge-1} this is the case if and only if~$v$ is contained in the bisection edge of all $d$-simplices containing~$v$. 
\begin{proposition}[$\sharp$-generation in Kuhn simplex]
	\label{pro:gsharp-2d-Kuhn}
	For some Kuhn simplex~$T_0 \subset \mathbb{R}^d$ let~$\tria \in \Bisec(\lbrace T_0 \rbrace)$ and let~$v \in \vertices(\tria)$. 
	Then we have the estimate
	\begin{align*}
		\gensharp(e) - \gensharp(e') \leq 2d\qquad\text{for all }e,e'\in \edges(\tria) \,\text{ with }\, v \in e \cap e'.
	\end{align*}
	In case of equality the types~$\typesharp(e) = \typesharp(e')$ coincide. If we have equality and $\typesharp(e) \in \lbrace 1,d\rbrace$, then the edge~$e'$ is uniquely determined by~$e$ and vice versa. 
\end{proposition}
\begin{proof}
Let~$\tria\in\Bisec(\lbrace T_0\rbrace)$ with edges~$e,e'\in \edges(\tria)$, vertex~$e\cap e' = v \in \vertices(\tria)$, and
\begin{align*}
\gensharp(e') < \gensharp(e).
\end{align*} 
It suffices to consider interior vertices~$v \in \vertices(\tria)$, since the constellation at the boundary is the same as the one of interior vertices for a subset of edges. 

\textit{Step 1 (Reduction to full valence).} 
Our proof starts by constructing a conforming vertex patch~$\omega(v) \subset \mtree$ that contains the edges~$e$ and~$e'$ and has full valence. 
Notice that this patch does not necessarily belong to a triangulation in~$\Bisec(\lbrace T_0\rbrace)$. 
However, the left-hand side~$\gensharp(e)-\gensharp(e')$ of the estimate as stated in the lemma does not change. 
This allows us to restrict the proof of the estimate to vertices $v$ of full valence in Step~2 below.
	
If~$\omega(v) \coloneqq \lbrace T \in \tria\colon v\in \vertices(T)\rbrace$ already has full valence, then we proceed with Step~2.  
Otherwise, there exists at least one simplex~$\widetilde{T}$ with~$v \notin \bse(\widetilde{T})$.  
We choose an edge~$\widetilde{e} \in \edges(\omega(v))$ with minimal $\sharp$-generation among all bisection edges not containing~$v$. 
We claim that~$\widetilde{e}$ can be bisected without conforming closure within~$\omega(v)$. 
By this we mean that when restricting the bisection routine to the triangulation~$\omega(v)$, then no conformal closure is needed. 
	For this we have to show that for any~$T \in \omega(v)$ which as~$\widetilde{e}$ as edge, $\widetilde{e}$ is its bisection edge. 
	We proceed by contradiction and assume that~$T \in \omega(v)$ with~$\widetilde{e} \in \edges(T) $ is such that~$\bse(T) \neq \widetilde{e}$. 
	By Lemma~\ref{lem:PropSharp}~\ref{itm:gsharp-oldest} we have that~$\gensharp(\bse(T)) < \gensharp(\widetilde{e})$.  
	We have chosen~$\widetilde{e}$ as bisection edge with minimal $\sharp$-generation among all the bisection edges that do not contain~$v$. 
	Consequently, we have that~$v \in \bse(T)$. 
	Since~$v\in \widetilde{T}\cap T$, by Lemma~\ref{lem:bisec-edge-2}~\ref{itm:bisec-edge-a} it follows that~$\bse(T) \cap \widetilde{e} \neq \emptyset$. 
	In combination with the fact that~$v \in \bse(T)$ this shows that both bisection edges~$\widetilde{e}$ and~$\bse(T)$ are contained in~$\widetilde{T} \cap T$. 
	Since~$\widetilde{e}$ is the bisection edge of~$\widetilde{T}$ and~$\bse(T)\subset \widetilde{T}$, by Lemma~\ref{lem:PropSharp}~\ref{itm:gsharp-oldest} we have that~$\gensharp(\widetilde{e})<\gensharp(\bse(T))$, which is a contradiction to the above. 
	Hence, any simplex~$T \in \omega(v)$ has~$\widetilde{e}$ as bisection edge. 
	Thus, when~$\widetilde{e}$ is bisected no interior edges are bisected and in particular the edges~$e, e'$ are not bisected. 
	We update~$\omega(v)$ as the patch that contains all resulting simplices with vertex~$v$. 
	This does not alter the domain covered by the patch. 
	We repeat this procedure until~$v$ has full valence.
	
\textit{Step 2 ($\sharp$-generation).}
	The vertex~$v$ is contained in the uniform refinement~$\tria_k \coloneqq \lbrace T\in \mtree\colon \gen(T) = k\rbrace$, in which all simplices have generation~$k=(\level(v)+2) d \geq \gen(v) + 2d$. 
	By Lemma~\ref{lem:propertiesGenTypeLvl}~\ref{itm:oldestonbse} and~\ref{itm:simplex-edge} the vertex~$v$ has full valence in~$\tria_k$ and hence the number and directions of the edges containing~$v$ in the patch~$\omega(v)$ from Step~1 and~$\tria_k$ agree. 
	Therefore we can identify any edge~$e_1 \in \edges(\tria_k)$ containing~$v$ with an edge~$e_2\in \edges(\omega(v))$ containing~$v$ in the sense that either~$e_1 \subset e_2$ or~$e_2\subset e_1$. 
	
	The patch around~$v$ in~$\tria_k$ consists of~$2^d$ reflected smaller Kuhn cubes as depicted in Figure~\ref{fig:kuhnedges3D}. 
	To identify the edges we transform this patch of~$2^d$ cubes to~$[-1,1]^d$. 
	The vector~$0$ represents the vertex~$v$ and each $\tau \in \{-1,0,1\}^d\setminus \{0\}$ represents an edge~$e_\tau \in \edges(\omega(v))$ with~$v \in e_\tau$. 
	For~$\tau \in \{-1,0,1\}^d\setminus \{0\}$ we define the vector~$\abs{\tau} \coloneqq (\abs{\tau_i})_i\in \{0,1\}^d\setminus \{0\}$.
	The transformed patch~$[-1,1]^d$ consists of~$d!\,2^d$ Kuhn simplices. 
	With permutation~$\pi\colon \set{1,\dots,d}\to \set{1,\dots,d}$ and $\gamma \in \set{\pm 1}^d$ those read
		\begin{align*}
			T_{\pi,\gamma} \coloneqq \simplex{\gamma_1\fre_{\pi(1)}, 
				\gamma_1\fre_{\pi(1)}+\gamma_2 \fre_{\pi(2)}, \ldots,  \gamma_1\fre_{\pi(1)}+\dots+\gamma_d\fre_{\pi(d)} \leveljump 0 }.
		\end{align*}
	Each simplex~$T \in \omega(v)$ corresponds to one of the simplices~$T_{\pi,\gamma}$. 
	This allows us to determine whether two edges~$e_{\tau}$ and~$e_{\tau'}$ with representations~$\tau,\tau' \in \{-1,0,1\}^d\setminus \{0\}$ are contained in one~$d$-simplex~$T \in \omega(v)$. 
	This is the case if and only if we have with Hadamard product~$a \hprod b \coloneqq (a_i b_i)_i$ that
	\begin{align}
		\label{eq:kuhn-change-edge}
		\tau'= \tau \hprod \mu \qquad \text{or} \qquad \tau = \tau' \hprod \mu \qquad \text{ for some } \mu \in \set{0,1}^d \setminus \{0\}.
	\end{align}
	In other words,~$e_\tau$ and~$e_{\tau'}$ are contained in one  simplex if $\tau$ arises from~$\tau'$ by exactly one of the following two operations: 
	\begin{enumerate}[label = (\roman*)]
		\item 
		\label{itm:nonz-z} 
		an arbitrary number (but not all) non-zero entries are set to zero,
		\item 
		\label{itm:z-nonz} 
		an arbitrary number of zero entries are set to non-zero entries.
	\end{enumerate}
	If~$e_\tau$ and~$e_{\tau'}$ share a joint simplex, we write~$\tau \sim \tau'$.
	Lemma~\ref{lem:PropSharp}~\ref{itm:gsharp-edge-child} shows that the~$\sharp$-type is preserved for any edge~$e_\tau$ when bisected. 
	Thus, we can identify the $\sharp$-type with the ones of the corresponding edges in~$\tria_k$ that have the same direction. 
	Considering the types of edges in~\eqref{eq:kuhn-simplex-sorted} one can deduce
	\begin{align}
		\label{eq:kuhn-gensharp}
		\typesharp(e_\tau) = (d+1) - \norm{\tau}_{1}. 
	\end{align}
	
	Let now~$\sigma, \sigma' \in \set{-1,0,1}^d \setminus \set{0} $ be such that~$e = e_\sigma$ and~$e' = e_{\sigma'}$. 
	To prove the statement we aim to find sequences~$\tau_0, \tau_2, \ldots, \tau_N \in \set{-1,0,1}^d \setminus \set{0}$ with~$\tau_i \sim \tau_{i+1}$ for~$i = 0, \ldots, N-1$, and~$\tau_0=\sigma$,~$\tau_N = \sigma'$ with~$N\leq 4$. 
	Each~$\tau_i$ represents an edge containing the vertex~$v$.
	Due to Lemma~\ref{lem:gensharp-patch} we obtain
	\begin{align}
		\label{est:gen-edge-simpl}
		\abs{\gensharp(e_{\tau_i}) - \gensharp(e_{\tau_{i+1})}} \leq d-1 \qquad \text{for $i=0,\dots, N-1$.}
	\end{align}
	This would result in~$\abs{\gensharp(e_{\tau_0}) - \gensharp(e_{\tau_N})} \leq N(d-1)$.  
	We shall construct sequences of recurring $\sharp$-type to improve this estimate.
	For example, the estimate $\abs{\gensharp(e_{\tau_i}) - \gensharp(e_{\tau_{i+2}})} \leq 2(d-1)$ reduces to
	\begin{align}
		\label{eq:reduce-to-d}
		\abs{\gensharp(e_{\tau_i}) - \gensharp(e_{\tau_{i+2}})} \leq d \qquad \text{ if }\typesharp(e_{\tau_i}) = \typesharp(e_{\tau_{i+2}}).
	\end{align}
	
	Obviously, there is nothing to show if~$e = e'$. 
	We start with two cases, where a sequence of length~$N=2$ suffices.
	\begin{enumerate}
		\item (Same face of $[-1,1]^d$) \label{itm:2d-sameface}
		If~$\sigma, \sigma'$ share one non-zero entry, i.e.,~$\sigma_j=\sigma_j'\neq 0$ for some~$j \in \set{1,\dots d}$, then we set~$\tau_0 \coloneqq \sigma$, $\tau_2 \coloneqq \sigma'$, and 
		\begin{align*}
			\tau_1 \coloneqq  \sigma \hprod \fre_j = \sigma' \hprod \fre_j \neq 0.
		\end{align*}
		Then we have~$\tau_0 \sim \tau_1 \sim \tau_2$ and thus~$\abs{\gensharp(e) - \gensharp(e')} \leq 2(d-1)$.
		\item (Same sub-cube of $[-1,1]^d$)
		\label{itm:2d-samecube}
		If~$\sigma \hprod \sigma' = 0$, then we  set~$ \tau_0 \coloneqq \sigma$, $\tau_2 \coloneqq \sigma'$, and
		\begin{align*}
			\tau_1 \coloneqq \sigma + \sigma' 
			\in \set{-1,0,1}^d \setminus \set{0}.
		\end{align*}
		Since~$\sigma, \sigma' \neq 0$, we have that~$\abs{\sigma},\abs{\sigma'} \in \{0,1\}^d \setminus\{0\}$ and we obtain
		\begin{alignat*}{4}
			\tau_1 \hprod \abs{\sigma}\, &= \sigma \hprod \abs{\sigma} \,+ \sigma' \hprod \abs{\sigma} &&= \sigma &&= \tau_0,
			\\
			\tau_1 \hprod \abs{\sigma'} &= \sigma \hprod \abs{\sigma'} + \sigma' \hprod \abs{\sigma'} &&= \sigma' &&= \tau_2.
		\end{alignat*}
		Hence, we have~$\tau_0 \sim \tau_1 \sim \tau_2$, and thus~$\abs{\gensharp(e) - \gensharp(e')} \leq 2(d-1)$.
	\end{enumerate}
	It remains to consider the case, in which there exists an index~$j \in \{1, \ldots, d\}$ such that~$\sigma_j = - \sigma'_j \neq 0$.  
	We first investigate the case of a sequence of length~$N=3$, which suffices if the types of the edges do not agree. 
	\begin{enumerate}\setcounter{enumi}{2} 
		\item 
		\label{itm:2d-notsamet} 
		Let~$\typesharp(e) \neq \typesharp(e')$. 
		For this argument it does not matter which edge has higher generation. Thus, we may assume without loss of generality that~$\typesharp(e) > \typesharp(e')$, i.e., due to~\eqref{eq:kuhn-gensharp} we have~$\norm{\sigma}_1 < \norm{\sigma'}_1$.  
		First we construct~$\tau_1$ such that~$\typesharp(e_{\tau_1}) = \typesharp(e')$, i.e.,  $\norm{\tau_1}_1 = \norm{\sigma'}_1$. 
		For this purpose we choose a vector~$\mu \in \set{0,1}^d \setminus \{0\}$ with~$\norm{\mu}_1 = \norm{\sigma'}_1 - \norm{\sigma}_1 \in \{1, \ldots, d-1\}$, $\mu \hprod \sigma = 0$ and $\mu \hprod \abs{\sigma'} = \mu$. 
		We set 
		\begin{align*}
			\tau_0 \coloneqq \sigma,\quad \tau_1 \coloneqq \sigma + \mu \hprod \sigma',\quad
			\tau_2 \coloneqq \mu \hprod \tau_1 \quad \text{ and } \quad  \tau_3 \coloneqq \sigma'. 
		\end{align*}
		By construction of~$\mu$ we have $\norm{\tau_1}_1 = \norm{\sigma}_1 + \norm{\mu}_1 = \norm{\sigma'}_1$ and therefore~$\typesharp(e_{\tau_1}) = \typesharp(e')$. 
		Denoting~$\bfone \coloneqq (1,\dots, 1)$ and using~$\mu \hprod \sigma = 0$ shows that~$\tau_0 = (\bfone - \mu) \hprod \tau_1$ and $\tau_2 = \mu \hprod \sigma' = \mu \hprod \tau_3$. 
		Hence we have that~$\tau_0 \sim \tau_1 \sim \tau_2 \sim \tau_3$ and by~\eqref{est:gen-edge-simpl} and~\eqref{eq:reduce-to-d} we obtain
		\begin{align*}
			\qquad\quad \abs{\gensharp(e) - \gensharp(e')}
			&\leq
			\abs{\gensharp(e_{\tau_0}) - \gensharp(e_{\tau_1})}+
			\abs{\gensharp(e_{\tau_1}) - \gensharp(e_{\tau_3})}
			\\
			&\leq (d-1) + d = 2d-1.
		\end{align*}
	\end{enumerate}
	In case of coinciding types a sequence of length~$N=4$ can be found. 
	\begin{enumerate}\setcounter{enumi}{3} 
		\item 
		\label{itm:2d-samet-d}
		If~$\typesharp(e) = \typesharp(e') =  d$, then we have~$\sigma = \pm \fre_j$ for some~$j \in \set{1,\dots, d}$ and~$\sigma'=-\sigma$. 
		We choose~$k \in \{1, \ldots, d\}\setminus \{j\}$ and define
		\begin{align*}
			\tau_0 \coloneqq \sigma, \quad
			\tau_1 \coloneqq \sigma + \mathfrak{e}_k, \quad
			\tau_2 \coloneqq  \mathfrak{e}_k, \quad
			\tau_3 \coloneqq  \mathfrak{e}_k + \sigma', \quad
			\tau_4 \coloneqq \sigma'. 
		\end{align*}
		Then, we have~$\tau_1 \hprod \fre_j = \tau_0$, $\tau_1 \hprod \fre_k = \tau_2$, $\tau_3 \hprod \fre_k = \tau_2$ and~$\tau_3 \hprod \fre_j = \tau_4$. 
		Since $\typesharp(e_{\tau_0}) = \typesharp(e_{\tau_2}) =  \typesharp(e_{\tau_4}) =d$, by~\eqref{eq:reduce-to-d} we obtain
		\begin{align*}
			\qquad\quad \abs{\gensharp(e) - \gensharp(e')}
			&\leq
			\abs{\gensharp(e_{\tau_0}) - \gensharp(e_{\tau_2})}+
			\abs{\gensharp(e_{\tau_2}) - \gensharp(e_{\tau_4})}\\
			&\leq 2d.
		\end{align*}
		\item 
		\label{itm:2d-samet-notd} 
		If~$\typesharp(e) = \typesharp(e')\neq d$ and~$\sigma_j = -\sigma_j' \neq 0$, we set~$\tau_0 \coloneqq \sigma$, $\tau_4 \coloneqq \sigma'$ and
		\begin{align*}
			\tau_1 &\coloneqq (\bfone - \fre_j) \hprod \sigma,
			\\
			\tau_2 &\coloneqq \tau_1 + \fre_j \hprod \sigma' = \tau_1 - \fre_j \hprod \sigma,
			\\
			\tau_3 &\coloneqq  \tau_2 \hprod \fre_j = \sigma' \hprod \fre_j \neq 0.
		\end{align*}	
		Thus, comparing~$\tau_2$ with~$\tau_0$ the sign of the~$j$-th entry has changed.  
		Note that~$\tau_1 = (\bfone -\fre_j)\hprod \tau_2$, and hence~$\tau_1 \sim \tau_2$. 
		Since~$\tau_0 \sim \tau_1\sim \tau_2\sim \tau_3\sim \tau_4$ and~$\typesharp(\tau_0) = \typesharp(\tau_2) = \typesharp(\tau_4)$, by~\eqref{eq:reduce-to-d} we find that
		\begin{align*}
			\qquad\quad \abs{\gensharp(e) - \gensharp(e')}
			&\leq
			\abs{\gensharp(e_{\tau_0}) - \gensharp(e_{\tau_2})}+
			\abs{\gensharp(e_{\tau_2}) - \gensharp(e_{\tau_4})}
			\\
			&\leq 2d.
		\end{align*}
	\end{enumerate} 
	Note that the cases~\ref{itm:2d-samet-d} and \ref{itm:2d-samet-notd} are the only ones, in which a difference of~$2d$ is possible. 
	In both cases we have~$\typesharp(e) = \typesharp(e')$. 
	Moreover, in~\ref{itm:2d-samet-d} $\sigma = \fre_j$ uniquely determines~$\sigma' = -\fre_j$.  
	If~$\type(e) = \type(e') = 1$, then all entries in~$\sigma$ and~$\sigma'$ are non-zeros and hence we can apply~\ref{itm:2d-sameface} except for the case~$\sigma = -\sigma'$.  
	In the latter case again~$e$ uniquely determines~$e'$ if~$\gensharp(e) - \gensharp(e') = 2d$ and $\typesharp(e) = \typesharp(e') = 1$. 
	This finishes the proof.
\end{proof}

\subsection{Generation estimates in general vertex patches}
\label{subsec:gen-diff-vtx}
Our next goal is to obtain estimates on the~$\gensharp$ difference of edges intersecting in general vertices. 
For non-critical vertices as introduced in Definition~\ref{def:MacroVertex} this is a direct consequence of the estimates in the Kuhn simplex in Proposition~\ref{pro:gsharp-2d-Kuhn} in Section~\ref{subsec:gener-estim-kuhn}. 

\begin{lemma}[Non-critical vertex]
	\label{lem:genSharpNonMacro}
	For~$\tria\in \BisecT$ with colored initial triangulation~$\tria_0$ let~$v\in \vertices(\tria)$ be a  non-critical vertex. 
	Then we have that
	\begin{align*}
		\gensharp(e) - \gensharp(e') \leq 2d\qquad\text{for all }e,e'\in \edges(\tria)\text{ with } v \in e \cap e'.
	\end{align*}  
	If equality holds, the types~$\typesharp(e) = \typesharp(e')$ agree.
	If additionally~$\typesharp(e) = \typesharp(e') \in \lbrace 1,d\rbrace$, then the edge~$e'$ is uniquely determined by~$e$ and vice versa.
\end{lemma}
\begin{proof}
	By Definition~\ref{def:MacroVertex} a non-critical vertex~$v$ is either contained in the interior of an initial~$d$-simplex~$T_0\in \tria_0$ or in the relative interior of a~$(d-1)$-subsimplex of two initial~$d$-simplices~$T_0,T_1\in \tria_0$.
	In both cases we can embed the initial simplex~$T_0$ (and~$T_1$) into a (fully refined) Kuhn simplex. 
	Proposition~\ref{pro:gsharp-2d-Kuhn} then yields the lemma.
\end{proof}
For critical $n$-macro vertices, i.e., with~$n\leq d-2$, the estimate in Lemma~\ref{lem:genSharpNonMacro} does not hold in general. 
This is due to the fact that the valence might be too high to allow for a transformation to a Kuhn simplex, see Figure~\ref{fig:HighValence}.
Still, for~$n \geq 1$ the following lemma shows that the $\gensharp$ difference of edges~$e,e'$ intersecting in an $n$-macro vertex~$v$ can be controlled independently of the valence. 
 
\begin{lemma}[Critical $n$-macro vertices for $n\geq 1$]
	\label{lem:genSharpnMacro}
	For~$\tria\in \BisecT$ with colored initial triangulation~$\tria_0$, 
	let~$v\in \vertices(\tria)$ be an~$n$-macro vertex with~$n \in \lbrace 1,\dots, d-2\rbrace$ and let~$e,e'\in \edges(\tria)$ be edges intersecting in~$v \in e \cap e'$.
	\begin{enumerate}
		\item 
		\label{itm:genSharpnMacro1}
		We have the estimate
		\begin{align*}
			\abs{\gensharp(e) - \gensharp(e')}
			\leq 4d. 
		\end{align*}
		\item 
		\label{itm:genSharpnMacro2}
		If in addition~$e,e'\subset T_0$ for some~$T_0 \in \tria_0$, then we have that
		\begin{align*}
			\abs{\gensharp(e) - \gensharp(e')} \leq 2d.
		\end{align*}
	\end{enumerate}
\end{lemma}
\begin{proof}
	The statement in~\ref{itm:genSharpnMacro2} follows by Lemma~\ref{lem:genSharpNonMacro} and the fact that all edges~$e,e'\in \edges(\tria)$ with~$\tria\in \BisecT$ and~$e,e'\subset T_0\in \tria_0$ are included in a triangulation~$\tria'\in \Bisec(\lbrace T_0\rbrace)$. 
	It remains to prove~\ref{itm:genSharpnMacro1}. 	
	Since~$v$ is an~$n$-macro vertex with $n\in \set{1, \ldots, d-2}$,
	there are initial simplices~$T_0, T_0'\in \tria_0$ sharing an~$n$-subsimplex such that~$v \in T_0 \cap T_0'$, and~$e \subset T_0$ as well as~$e'\subset T_0'$. 
	In particular, there exists an edge~$e_0 \in \edges(\tria)$ with~$v \in e_0 \subset T_0 \cap T_0'$. 
Using~\ref{itm:genSharpnMacro2} we have~$\abs{\gensharp(e)-\gensharp({e}_0)} \leq 2d$ and
~$\abs{\gensharp(e')-\gensharp({e}_0)} \leq 2d$. 
This proves that~$\abs{\gensharp(e)-\gensharp(e')} \leq 4d$.
\end{proof}
\begin{remark}[Improvement for small dimensions]
It is possible to slightly improve Lemma~\ref{lem:genSharpnMacro}~\ref{itm:genSharpnMacro1} due to the fact that there is some freedom in the choice of the edge $e_0$ in the proof of the lemma, but this is not crucial in the following. 
\end{remark}

The previous lemma shows that for critical $n$-macro vertices with~$ n \geq 1$ the difference of the $\sharp$-generation is rather moderate even for edges with high valence. 
In contrast, large differences of the $\sharp$-generation may occur for edges intersecting in a $0$-macro vertex. 
We shall see that its value depends on the `diameter' of the initial vertex patch in terms of simplex chains connected by an edge. 

Let us make this more precise. 
By \emph{$1$-chain} of simplices~$T_0,\dots,T_N$ we refer to a sequence of \emph{$1$-neighbors} in the sense that~$T_j$ and~$T_{j+1}$ share an edge for all~$j=0,\ldots,N-1$. 
Given~$\tria\in \BisecT$, let~$\delta^1_{\tria}(T,T')$ denote the minimal $1$-distance between two simplices~$T,T'\in \tria$, that is, 
\begin{align*}
	\delta^1_{\tria}(T,T') \coloneqq \min \lbrace N\in \mathbb{N}_0\colon& \text{ there exists a $1$-chain of simplices }T_0,\dots,T_N \in \tria\\
	&\text{ with }T_0 = T\text{ and }T_N = T'\rbrace.
\end{align*}
Let~$\tria_0^+ = \lbrace T \in \mtree \colon \gen(T) = d\rbrace$ denote the triangulation resulting from a full uniform refinement ($d$ bisections of each simplex) of~$\tria_0$  and consider the vertex patch~$\omega_{\tria_0^+}(v) \coloneqq \lbrace T \in \tria_0^+\colon v\in \vertices(T)\rbrace$ for~$v\in \vertices(\tria_0)$. 
We define the quantity 
\begin{align}
	\label{def:C0}
	C(\tria_0) \coloneqq \max_{v\in \vertices(\tria_0)}\, \max_{T,T'\in \omega_{\tria_0^+}(v)} \delta^1_{\omega_{\tria_0^+}(v)}(T,T').
\end{align}

\begin{lemma}[Critical $0$-macro vertex]
	\label{lem:genSharp0Macro}
	Let~$\tria\in \BisecT$ be a triangulation with colored~$\tria_0$. 
	Let~$v\in \vertices(\tria)$ be a $0$-macro vertex and let~$e,e'\in \edges(\tria)$ be edges intersecting in~$v \in e \cap e'$. 
	Then we have that
	\begin{align*}
		\gensharp(e) - \gensharp(e') \leq (C(\tria_0)+1) (d-1). 
	\end{align*}
\end{lemma}
\begin{proof}
	Let~$v\in \vertices(\tria)$ be a $0$-macro vertex. 
	The arguments in Step~1 in the proof of Proposition~\ref{pro:gsharp-2d-Kuhn} apply and show that without loss of generality we can assume that $v$ has full valence. 
	Lemma~\ref{lem:gensharp-patch} shows that the maximal $\gensharp$ difference for two edges in the same~$d$-simplex is bounded by~$d-1$. 
	Since all initial simplices in~$\tria_0$ are of type~$d$, by Lemma~\ref{lem:propertiesGenTypeLvl}~\ref{itm:simplex-refEdge} it follows that after~$d$ refinements all initial edges~$e\in \edges(\tria_0)$ have been bisected. 
	Thus, after one full refinement each $0$-macro vertex is fully refined. 
	Hence, for any two given edges containing~$v \in \vertices(\tria_0)$ there exists a sequence of 1-neighboring simplices of length at most $C(\tria_0)$ connecting those edges. 
	Combining these properties yields the statement. 
\end{proof}

\begin{remark}[Alternative bounds on $C(\tria_0)$]
	\label{rmk:bounds-CT0}\ 
	\begin{enumerate}
		\item 
		The constant $C(\tria_0)$ can be bounded by the shape regularity of~$\tria_0$. 
		However, even when the shape regularity degenerates, the constant~$C(\tria_0)$ might not be affected. 
		This can be seen, for example, by scaling a given triangulation in one direction or by considering an initial triangulation consisting of a bisected edge patch with very high simplex valence. 
		\item 
		An alternative quantity, depending more directly on~$\tria_0$ than~$C(\tria_0)$ is
		\begin{align*}
			C'(\tria_0) \coloneqq \max_{v\in \vertices(\tria_0)}\, \max_{T,T'\in \omega_{\tria_0}(v)} \delta^1_{\omega_{\tria_0}(v)}(T,T').
		\end{align*} 
		Since the bisection vertex~$b$ of any initial simplex~$T_0 \in \tria_0$ with vertex~$v\in \vertices(\tria_0)$ satisfies~$\gen(b) = 1$, we have $\gensharp(\simplex{b\leveljump v}) = d+1$. 
		Consequently, all descendants of~$T_0$ resulting from one full refinement and containing~$v$, contain the edge $\simplex{b\leveljump v}$ and are thus 1-neighbors. 
		Hence, the length of any minimal $1$-chain in the vertex patch has increased by at most factor~$2$, i.e., 
		\begin{align*}
			C(\tria_0) \leq 2 C'(\tria_0).
		\end{align*}
	\end{enumerate}
\end{remark}

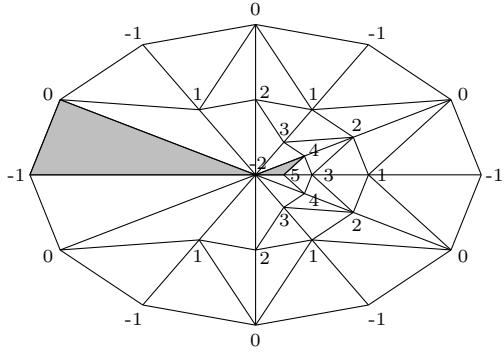
\begin{figure}
	\input{PictureTex/BadMacro2d.tex}
	\caption{Triangulation with vertex generations: 
	There is a generation difference $5$ between two intersecting simplices (gray) sharing the high-valence vertex in the center. 
	}
	\label{fig:HighValence}
\end{figure}

\subsection{From~\texorpdfstring{$\sharp$}{\#}-generation to level estimates}\label{subsec:mainResultLocal-level}
In this subsection we transfer the $\gensharp$ estimates of intersecting edges to level estimates, which leads to the main result of this section, Proposition~\ref{prop:MaxLvlJump}. 
This, in turn, is instrumental in the proof of Proposition~\ref{prop:LeavingIncreasesDim} in Section~\ref{subsec:colored-T0}, the key to our main result. 

\begin{lemma}[Local level estimates]
	\label{lem:genSharpToLvl}
	For a triangulation~$\tria \in \BisecT$ with colored~$\tria_0$
	let~$e,e'\in \edges(\tria)$  be edges such that for some number~$N \in \mathbb{N}$ we have
	\begin{align*}
		\gensharp(e) - \gensharp(e') \leq d N. 
	\end{align*} 
	Then, we have the level estimate
\begin{align*}
\level(e) - \level(e') \leq  N.
\end{align*}	
\end{lemma}
\begin{proof}
By Lemma~\ref{lem:PropSharp}~\ref{itm:lvlSharpEdge} we have 
		\begin{align*}
			\level(e) - \level(e') 
			&= 	\levelsharp(e) - \levelsharp(e')	\\
			&= 	
			\tfrac{1}{d} (\gensharp(e) - \gensharp(e') + \typesharp(e') - \typesharp(e) ) \\
			& \leq  \tfrac{1}{d} ( d N + d-1 ) = N + 1 - \tfrac{1}{d}.
		\end{align*}
Since the level difference is integer, the claim follows. 
\end{proof}

Applying Lemma~\ref{lem:genSharpToLvl}  to the $\sharp$-generation estimates in 
Section~\ref{subsec:gen-diff-vtx} allows us to bound the following maximal edge level jump.
\begin{definition}[Maximal level jump]
	\label{def:maxLvlJump-2}
	Let~$\tria_0$ be colored and~$v \in \vertices(\mtree)$. 
	We define the maximal level jump~$\jump(v)$ at~$v$ as 
	\begin{align*}
		\jump(v) &\coloneqq \max_{\tria \in \BisecT} \max_{e,e' \in \edges(\tria)\colon v \in e \cap e'}
		\bigabs{\level(e) - \level(e')}.
	\end{align*}
\end{definition}
We aim for bounds on the level difference in terms of the constants 
\begin{align}\label{eq:defJn-2}
	J_n \coloneqq \begin{cases}
		0& \text{for }  n \in \{d-1,d\},\\
		2 & \text{for } n \in \{1,\ldots,d-2 \},\\
		\lceil(C(\tria_0) +1)(1-1/d)\rceil - 2 \;\;& \text{for }n = 0.
	\end{cases}
\end{align}
For $n\in \mathbb{N}$ the constant~$J_n$ estimates by how much the maximal level jump of an $n$-macro vertex may exceed two. 
More precisely, we have the following result.
\begin{proposition}[Level jump]
	\label{prop:MaxLvlJump}
	Suppose that the initial triangulation~$\tria_0$ is colored. 
	Let~$v \in \vertices(\mtree)$ be an arbitrary $n$-macro vertex for~$n \in \{0, \ldots, d\}$, and let~$J_n\in \mathbb{N}_0$ be defined as in \eqref{eq:defJn-2}.
Then the maximal edge level jump satisfies~$\jump(v) \leq 2 + J_n$. 
	That is, for all edges~$e,e'\in \tria\in \BisecT$ with~$e\cap e' = v$ we have the estimate
		\begin{align*}
			\level(e) - \level(e') &\leq 2 + J_n.
		\end{align*}
\end{proposition}
\begin{proof}
	Due to Lemma~\ref{lem:genSharpNonMacro}, Lemma~\ref{lem:genSharpnMacro}~\ref{itm:genSharpnMacro1}, and Lemma~\ref{lem:genSharp0Macro} we have for any~$e,e'\in \mathcal{E}(\mtree) $ with~$v\in e \cap e'$ we have that
	\begin{align*}
		\gensharp(e') - \gensharp(e) \leq  
		\begin{cases}
			2d& \text{for } n \in \{d-1,d\},\\
			4d & \text{for } n \in \{1,\ldots,d-2 \},\\
			(C(\tria_0) +1)(d-1) & \text{for }n = 0. 
		\end{cases} 
	\end{align*}
	The terms on the right-hand side are bounded above by~$(2+J_n)\, d$ with~$J_n$ as defined in \eqref{eq:defJn-2}. 
	The claim then follows using Lemma~\ref{lem:genSharpToLvl}. 
\end{proof}

	\section{Neighborhood and auxiliary triangulation}
	\label{sec:leavingTheNgh}
	In this section we investigate properties of triangulations~$\tria\in \BisecT$ within a neighborhood of a vertex~$v\in \vertices(\tria)$. 
	For this purpose we construct a strongly graded auxiliary triangulation in Section~\ref{subsec:AuxTria} with some beneficial layer structure investigated in Section~\ref{subsec:Layers}. 
	In Section~\ref{subsec:FinerTria} we present conditions under which  this auxiliary triangulation is finer than~$\tria$. 
	This allows us to prove generation estimates for edge chains within a neighborhood in Section~\ref{subsec:ChainsLeavingNgh}. 
	
	The neighborhood is defined as follows. 
	For an arbitrary vertex~$v\in \vertices(\mtree)$ and a number~$m\in \mathbb{N}_0$ with~$m \geq \level(v)$ we define the non-empty uniform vertex patch 
	\begin{align}\label{def:defnodalngbh}
		\omega_m(v) \coloneqq \lbrace T\in \mtree \colon \gen(T) = md \,\text{ and }\,v \in \vertices(T)\rbrace.
	\end{align}
	We define the boundary of the patch~$\omega_m(v)$ as
	\begin{align*}
		\patchbd \coloneqq \bigcup \lbrace F \text{ is a $(d-1)$-simplex (face) in }\omega_m(v) \colon v \not\in F\rbrace.
	\end{align*}
	The set covered by~$\omega_m(v)$ is denoted by
	\begin{align}\label{eq:defOmegahat}
		\overline{\Omega}_m(v) &\coloneqq \bigcup  \omega_m(v)\quad\text{and} \quad
		\Omega_m(v) \coloneqq \overline{\Omega}_m(v) \setminus \partial \omega_m(v).
	\end{align}
	If~$v\in \Omega$ is an interior vertex, the neighborhood~$\Omega_m(v)$ is open, $\patchbd$ is the topological boundary~$\partial \Omega_m(v)$, and~$\overline{\Omega}_m(v)$ its topological closure. 
	Otherwise, if~$v \in \partial \Omega$, then~$\patchbd$ is essentially the part of the vertex patch boundary that lies in the interior of the domain~$\Omega$.

	\begin{lemma}[Macro hyperface]
		\label{lem:passingKfaces}
		For a vertex~$v\in \vertices(\mtree)$ let~$m \in \mathbb{N}_0$ be such that~$m \geq \level(v)$. 
		Let~$n \in \set{0, \ldots, d}$ be such that~$v$ is an $n$-macro vertex as in Definition~\ref{def:MacroVertex} with~$F$ the $n$-subsimplex  of a $d$-simplex in~$\tria_0$ such that~$v$ is contained in~$F$. 
		Then~$\Omega_m(v)$ does not contain any further $n$-subsimplices in~$\tria_0$. 
		In other words, any $n$-subsimplex~$F'\neq F$ in $\tria_0$ satisfies~$F' \cap \Omega_m(v) = \emptyset$. 
	\end{lemma}
	\begin{proof}
	The lemma follows directly from the definition in \eqref{eq:defOmegahat}. 
	\end{proof}

	\subsection{Auxiliary triangulation}
	\label{subsec:AuxTria}
	In the following we construct an auxiliary triangulation and investigate its properties. 
	We assume that the initial triangulation~$\tria_0$ is colored. 
	Moreover, we fix a vertex~$v \in \vertices(\mtree)$ and a number~$m\in \mathbb{N}_0$ with
	\begin{align}\label{eq:assump-v-aux}
	 \level(v) < m.
	\end{align} 
	Later we shall see that this condition ensures that~$v$ is the oldest vertex in each $d$-simplex~$T \in \omega_m(v)$ with the vertex patch defined in~\eqref{def:defnodalngbh}. 
	
	Starting from the uniform triangulation of generation~$m d$ we refine infinitely often outside of the uniform vertex patch~$\omega_m(v)$. 
	Later we shall denote the resulting triangulation by~$\tria_m^\infty(v) \in \overline{\BisecT}$ and its restriction to the vertex patch by~$\omega_m^\infty(v)$. 
	Let us make this construction precise:  
	Starting from the uniform triangulation~$\omega_m^0(v) \coloneqq \omega_m(v)$, we successively refine all simplices touching the boundary~$\partial \omega_m(v)$ of the vertex patch. 
	More precisely, we define recursively a sequence of partitions~$(\omega_m^{j}(v))_{j \in \mathbb{N}_0}$. 
	Each partition arises by refining all $d$-simplices in the previous partition that intersect with the boundary~$\patchbd$. In other words, for~$j \in \mathbb{N}$ we define 
	\begin{align}
		\label{eq:refStrategy}
		\begin{aligned}
			\omega_m^{j}(v) 
			\coloneqq \lbrace 
			& T \in \mtree\colon  T \in \omega_m^{j-1}(v) \,\text{ and }\,T \cap \patchbd = \emptyset \rbrace \\
			& \cup \lbrace T \in \mtree\colon  \parent(T) \in \omega_m^{j-1}(v)\,\text{ and }\,\parent(T) \cap \patchbd \neq \emptyset\rbrace.
		\end{aligned}
	\end{align}
	Figure~\ref{fig:Layer} shows the restrictions of such an auxiliary triangulations to a single simplex in~$\omega_m(v)$. 
	By design each~$\omega_m^{j}(v)$ forms a partition of the domain~$\Omega_m(v)$, but it is not obvious that it is conforming. 
	This shall be confirmed in the sequel.  
A partition of~$\Omega_m(v)$ consisting of infinitely many simplices is given by 
	\begin{align*}
		\omega_m^\infty(v) \coloneqq \liminf_{j\to \infty} \omega_m^j(v) = \bigcup_{k =0}^\infty \bigcap_{j = k}^\infty \omega_{m}^j(v). 
	\end{align*}
	Equivalently, the set can be characterized as
	\begin{align}\label{eq:charomegaINfty}
		\omega_m^\infty(v) = \lbrace T \in \omega_m^j(v) \colon j\in \mathbb{N}\,\text{ and }\,T \cap \patchbd = \emptyset\rbrace.
	\end{align}
	By construction the simplices~$T \in \omega_m^j(v)$ that touch the boundary~$\patchbd$, i.e., $T \cap \patchbd \neq \emptyset$, have generation~$\gen(T) = md+j$.  
	Hence, it is natural to extend the partition~$\omega_m^j(v)$ outside of~$\Omega_m(v)$ by a uniform partition of simplices of generation~$md +j$. 
	In this way we obtain for any~$j\in \mathbb{N}_0$ a partition of~$\Omega$ given by
	\begin{align}\label{def:auxTria}
		\tria_m^j(v) \coloneqq \omega_m^{j}(v) \cup \lbrace T \in \mtree\colon \gen(T) = md +j \text{ and } T \cap  \Omega_m(v) = \emptyset\rbrace.
	\end{align}
	This will turn out to be a conforming triangulation in the sense of  Definition~\ref{def:triangulation}, that is $\tria_m^j(v)\in \BisecT$, 
	see Proposition~\ref{prop:regTria} below. 
	Its proof shows that bisecting the simplices at the boundary does not require a closure step and hence conformity is preserved. 
	This fact relies on the following observations. 
	\begin{figure}
		\begin{minipage}{.45\textwidth}
			\input{PictureTex/Layer2d.tex}
		\end{minipage}
		\begin{minipage}{.45\textwidth}
			\input{PictureTex/Layer3D.tex}
		\end{minipage}
		\vspace*{-.5cm}
		\caption{Triangulations~${\omega}_m^{12}(v)$ for~$d=2$ (left) and~$\omega_m^9(v)$ for~$d=3$ (right) restricted to a single simplex in $\omega_m(v)$.}
		\label{fig:Layer}
	\end{figure}
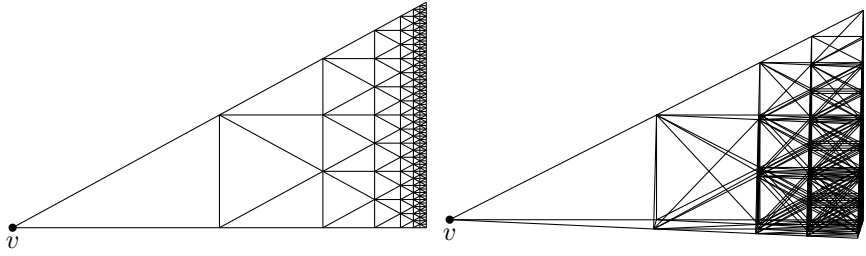
	
	\begin{lemma}[Boundary vertices]
		\label{lem:boundary-vertices}
		Let~$v \in \vertices(\mtree)$ be a vertex and~$m \in \mathbb{N}_0$ be such that~$\level(v)<m$.
		The following is valid for any~$T \in \omega_m^{j}(v)$ and~$j \in \mathbb{N}_0$. 
		\begin{enumerate}[label=(\roman*)]
			\item 
			\label{itm:boundary-vertices-prop} 
			Let~$T$ intersect with the patch boundary $\partial \omega_m(v)$, i.e., there is a vertex $a \in \vertices(T) \cap \patchbd$. 
			Then, for every vertex~$v' \in \vertices(T) \setminus \set{v}$ with 
			\begin{align*}
				\level(v') 
				< \level(a) \;\; \text { or } \;\;  
				\big(\level(v') 
				= \level(a) \; \text{ and  }\;  \gen(a)<\gen(v')\big)
			\end{align*} 
			we have that~$v'\in \patchbd$. In other words, for simplex $T$ of the form 
\begin{align*}	
			T = \simplex{v_0,\dots,v_{\ell-1}, a, v_{\ell + 1},\dots ,v_{\ell+n}\mid v_{\ell+n+1},\dots, v_d},
\end{align*} 
the vertices $v_0,\dots,v_\ell$ and $v_{\ell+n+1},\dots, v_d$ are on the boundary $\patchbd$.
			\item 
			\label{itm:boundary-vertices-type1} 
			If~$v' \in \vertices(T)$ is a vertex with~$m < \level(v')$ and~$\type(v')=1$, then~$v'$ is not on the boundary~$\patchbd$. 
		\end{enumerate}
	\end{lemma}
	\begin{proof}
		Before proving the statement by induction, let us collect the consequences of the assumption~$\level(v)<m$, cf.~\eqref{eq:assump-v-aux}. 
		Recall that the vertex patch~$\omega_m(v)$ consists of all $d$-simplices~$T \in \mtree$ containing~$v$ that have generation~$md$, i.e., they have type~$d$ and level~$m$. 
		Hence, in conjunction with Lemma~\ref{lem:genbased2}~\ref{itm:genbased2d}  this ensures that 
		\begin{align}
		\label{eq:lvl-v}
			\level(v) < m = \level(v')  \quad \text{ for all } v' \in \vertices(T) \setminus \{v\}\text{ and }T\in \omega_m(v).
		\end{align}
		That is,~$v$ is the oldest vertex and
		by Lemma~\ref{lem:propertiesGenTypeLvl}~\ref{itm:oldestonbse} we know that~$v \in \bse(T)$ for any~$T \in \omega_m(v)$. 
		Furthermore, by Lemma~\ref{lem:bisec-edge-1} it follows that 
		\begin{align}\label{eq:bisec-v}
			 v \in \bse(T') \quad \text{ for any } T'\in \forest(\omega_m^\infty(v))\text{ with } v \in \vertices(T').
		\end{align}
		We prove the lemma for all~$T \in \omega_m^j(v)$ by induction over~$j \in \mathbb{N}_0$.  
		
		\textit{Base case $j = 0$.} 
		Let~$T \in \omega_m(v)$ be arbitrary. 
		By the definition of the boundary of the vertex patch, all vertices~$v' \in \vertices(T)\setminus \{v\}$ are contained in $\patchbd$. 
		In combination with~\eqref{eq:lvl-v} this implies~\ref{itm:boundary-vertices-prop}. 
		Also, there is no vertex for which the conditions of~\ref{itm:boundary-vertices-type1} are valid, and thus it holds trivially. 
	
		\textit{Induction step.} 
		Suppose that~$T=\simplex{v_0,\dots,v_d}\in \omega_m^{j-1}(v)$ satisfies the claim.  
		We have to show that for each child~$T' \in \omega_m^{j}(v)$ of $T$ the statement holds as well. 
		Since only the simplices touching the boundary~$\patchbd$ are bisected it suffices to consider simplices~$T$ that intersect the boundary. 
		Thus, we may assume that~$T \cap \patchbd \neq \emptyset$. 
		
		\textit{Case 1 ($\type(T)=d$).} 
		In this case we have~$T = \simplex{v_0, \ldots, v_{d-1} \leveljump v_d}$ and the bisection edge is~$\bse(T) = \simplex{ v_{d-1} \leveljump v_d} $, see~Lemma~\ref{lem:genbased}~\ref{itm:genbaseda}.   
		Since the bisection vertex~$b = \bsv([v_{d-1},v_d]) = \vertices(T')\setminus \vertices(T)$ is the only new vertex in~$T'\in \children(T)$, it suffices to prove that~$b \notin \patchbd$. 
		Indeed, with~$\level(b)>\level(T)$ the assumption in~\ref{itm:boundary-vertices-prop} is not satisfied and hence the statement holds. 
	 	Furthermore, with~$\type(b) = 1$ and~$v \notin \patchbd$ also~\ref{itm:boundary-vertices-type1} follows.  
		If~$v \in \vertices(T)$, we know that~$v \in \bse(T)$ due to \eqref{eq:bisec-v}. 
		With~$v \notin \patchbd$ it follows that~$b \notin \patchbd$. 
		If on the other hand~$v \notin \vertices(T)$, then it follows  that~$v_{d-1} \notin \patchbd$, since otherwise by~\ref{itm:boundary-vertices-prop} all vertices of~$T$ would be contained in~$\patchbd$, which is not possible. 
		Again, from~$\bse(T) = \simplex{v_{d-1}\leveljump v_d}$ it follows that~$b\notin \patchbd$. 
		
		\textit{Case 2 ($\type(T)\neq d$)}. 
		This case is equivalent to~$\type(b) \neq 1$ for the bisection vertex~$b = \textup{mid}(\bse(T))$. 
		Hence~\ref{itm:boundary-vertices-type1} holds, since the assumption is not satisfied. 
		It remains to prove~\ref{itm:boundary-vertices-prop}. 
		If~$v_0 \notin \patchbd$, then the induction hypothesis implies that none of the vertices in~$T$ is on the boundary, which contradicts~$T \cap \patchbd \neq \emptyset$. 
		Hence, we have that the vertex~$v_0 \in \patchbd$ and it remains to show that~$b\in \patchbd$. 
		Since~$\type(T)<d$, by Lemma~\ref{lem:propertiesGenTypeLvl}~\ref{itm:simplex-refEdge} we have that~$\level(\bse(T))< \level(v_0)$, and hence the induction hypothesis for~$T$ shows with~\ref{itm:boundary-vertices-prop} that both vertices of~$\bse(T)$ are contained in $\patchbd$. 
		Consequently, also~$b\in \patchbd$ is on the boundary. 
		This concludes the proof. 
	\end{proof}
	Lemma~\ref{lem:boundary-vertices} allows us to show conformity of the auxiliary triangulations. 
	\begin{proposition}[Conforming triangulation]
		\label{prop:regTria}
		For~$v\in \vertices(\mtree)$ let~\eqref{eq:assump-v-aux} be satisfied, that is $\level(v) < m\in \mathbb{N}_0$. 
		Then for each~$j \in \mathbb{N}_0$ the partition~$\tria_m^j(v)$ defined in~\eqref{def:auxTria} is a conforming triangulations of the domain~$\Omega$ in the sense that
		\begin{align*}
			\tria_m^j(v) \in \BisecT\qquad\text{for any }j\in \mathbb{N}_0.
		\end{align*}
	\end{proposition}
	\begin{proof}
		For  vertex~$v\in \vertices(\mtree)$ and $\level(v) < m\in \mathbb{N}_0$ we know that the vertex patch~$\omega_m(v)$ is non-empty. Let~$j\in \mathbb{N}_0$ be arbitrary. 
		Conformity of~$\tria_m^j(v)$ at the boundary~$\patchbd$ is immediate, because the triangulation $\lbrace T \in \mtree\colon \gen(T) = md + j\rbrace $ is conforming, cf.~\cite[Thm.~4.3]{Stevenson08}. 
		Thus, it remains to show that conformity is preserved when refining~$\omega_m^{j-1}(v)$ to obtain~$\omega_m^j(v)$, which by definition in~\eqref{eq:refStrategy} arises by bisection of the simplices touching the boundary~$\patchbd$. 
		For this it suffices to show that for any simplex~$T \in \omega_m^{j}(v)$ with~$T \cap \patchbd \neq \emptyset$, the bisection edge touches the boundary~$\patchbd$ as well, i.e.,~$\bse(T) \cap \patchbd \neq \emptyset$. 
		Then bisecting all those simplices leads to a conforming triangulation, since no closure step is needed. 	
		
		Let~$T \in \omega_m^{j}(v)$ be a simplex with~$T \cap \patchbd \neq \emptyset$ of the form
		\begin{align*}
			T = \simplex{v_0,\dots,v_{\ell-1}\leveljump v_\ell,\dots,v_d}\qquad\text{for some } \ell \in \{ 1,\dots,d\}.
		\end{align*} 
		At least one vertex~$v_r$ with index~$r\in \lbrace 0 ,\dots,d\rbrace$ is on the boundary $\patchbd$. 
		In both cases~$\ell\leq r$ and~$\ell>r$ Lemma~\ref{lem:boundary-vertices}~\ref{itm:boundary-vertices-prop} shows that~$v_\ell \in \patchbd$. 
		Since the bisection edge~$\bse(T)$ of $T$ contains the vertex~$v_\ell$, this concludes the proof. 
	\end{proof}
In the following we consider the triangulation
	\begin{align}\label{def:tria_inf}
		\tria_m^\infty(v) \coloneqq \liminf_{j \to \infty} \tria_m^j(v) =  \bigcup_{k=0}^\infty \bigcap_{j = k}^\infty \tria_m^j(v) \in \overline{\BisecT}.
	\end{align}
We aim to show that~$\tria \leq \tria_m^\infty(v)$, that is $\forest(\tria) \subset \forest(\tria_m^\infty(v))$. 
				
\subsection{Layers}
\label{subsec:Layers}	
The partition~$\omega_m^j(v)$ defined in~\eqref{eq:refStrategy} exhibits a useful layer structure as visualized in Figures~\ref{fig:Layer} and~\ref{fig:LayersLem}, which we investigate in the following. 
For arbitrary simplices~$T,T'\in \omega_m^\infty(v)$ let~$\delta_{\omega_m^\infty(v)}(T',T)$ denote the simplex distance  in the triangulation~$\omega_m^\infty(v)$ of $\Omega_m(v)$, as defined in Definition~\ref{def:distances}. 
We define the distance between~$v$ and~$T \in \omega_m^\infty(v)$ in~$\omega_{m}^\infty(v)$ by 
\begin{align}
\label{def:distance-vT}
\delta_{\omega_m^\infty(v)}(v,T) \coloneqq \min_{T' \in \omega_m^\infty(v) \colon v \in \vertices(T')} \delta_{\omega_m^\infty(v)}(T',T). 
\end{align}
For any~$\ell\in \mathbb{N}$ we define the $\ell$-th layer triangulation and the layer covered by it by
	\begin{align}
		\label{eq:DefLambda}
		\lambda_m^\ell(v) \coloneqq \lbrace T \in \omega_m^\infty(v) \colon \delta_{\omega_m^\infty(v)}(v,T) + 1 = \ell \rbrace
		\quad \text{and}\quad 
		\Lambda_m^\ell(v) \coloneqq \bigcup  \lambda^\ell_m(v).
	\end{align}
	Moreover, we set~$\Lambda^0_m(v) \coloneqq \lbrace v \rbrace$. 
	By definition these layers satisfy  for all~$\ell \in \mathbb{N}_0$
	\begin{align}\label{eq:PropLayer}
		\Lambda_m^\ell(v) \subset \Omega_m(v)\qquad\text{and}\qquad \Lambda_m^\ell(v) \cap \Lambda_m^{\ell+2}(v) = \emptyset.
	\end{align}
Since no closure is needed thanks to Proposition~\ref{prop:regTria}, the layer structure builds up in the construction of~$\omega_m^\infty(v)$. 
We define for any~$\ell \in \mathbb{N}_0$ the boundary layer as 
	\begin{align}\label{eq:PropLayer2}
		\gamma_m^\ell(v) \coloneqq \lbrace T \in \omega_m^{\ell d}(v) \colon T \cap \partial \omega_m(v) \neq \emptyset\rbrace\qquad\text{and}\qquad \Gamma_m^\ell(v) \coloneqq \bigcup \gamma_m^\ell(v).
	\end{align}
	The following lemma summarizes properties of the layers illustrated in Figure~\ref{fig:LayersLem}. 
	\begin{figure}
		\input{PictureTex/Layer2d_genA.tex}
		\input{PictureTex/Layer2d_genB.tex}
		\vspace*{-.5cm}
		\caption{Parts of partitions~$\omega_m^4(v)$ and~$\omega_m^6(v)$ with layers~$\lambda_m^\ell(v)$  and~$\gamma_m^\ell(v)$ with respect to~$v$ (containing all simplices above the label) and vertex generations. 
	The layers are shaded in varying intensity and the boundary layer is hatched.}
	\label{fig:LayersLem}
	\end{figure}
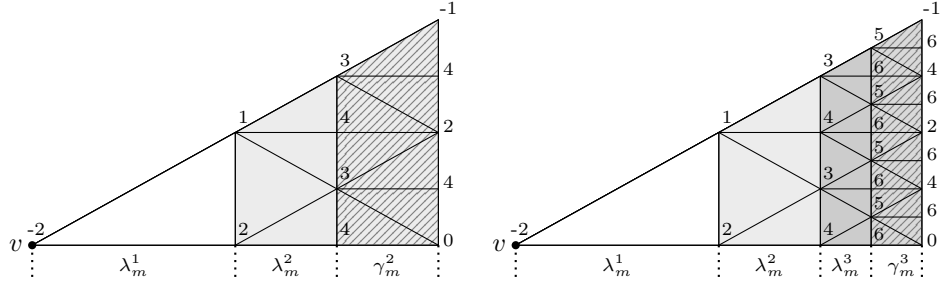

	\begin{lemma}[Properties of layers]\label{lem:Layers}
	For a vertex~$v\in \vertices(\mtree)$ let~$m \in \mathbb{N}_0$ be such that~$\level(v)<m$. 
	Then for all~$\ell\in \mathbb{N}$ the following statements hold:
	\begin{enumerate}
		\item 
		\label{itm:Layer1}
		An alternative characterization of each layer reads
		\begin{align*}
			\lambda^\ell_m(v) &= \omega_m^{\ell d}(v) \setminus \big(\omega_m^{(\ell-1) d}(v) \cup \gamma_m^{\ell}(v)\big)\ \text{ and }\ 
			\Lambda_m^\ell(v) = \overline{\Gamma_m^{\ell-1}(v) \setminus \Gamma_m^{\ell}(v)}.
		\end{align*}
		\item 
		\label{itm:Layer1b}
		Any~$T\in \gamma_m^{\ell}(v)$ has one vertex on the interface~$\Lambda_m^{\ell}(v) \cap \Lambda_m^{\ell+1}(v)$.
		Its other vertices are on the boundary~$\partial \omega_m(v)$ or on the interface~$\Lambda_m^{\ell}(v) \cap \Lambda_m^{\ell+1}(v)$.    
		\item 
		\label{itm:Layer2}
		Any simplex~$T \in \lambda_m^\ell(v)$ satisfies~$\level(T) = m+\ell$.
		\item
		\label{itm:Layer3}  
		Any vertex~$v' \in \vertices(\omega_m^\infty(v)) \setminus ( \patchbd \cup \set{v})$ is contained in some interface~$\Lambda^{\ell}_m(v) \cap  \Lambda^{\ell+1}_m(v)$ with~$\ell \in \mathbb{N}$ and satisfies~$\level(v') = m+\ell$. 
	\end{enumerate}
\end{lemma}
\begin{proof}	
	We prove the statements by induction over~$\ell\in \mathbb{N}$. 
	
	\textit{Base case $\ell = 1$.}
	Let~$T \in \omega_m^0(v)$ be arbitrary but fixed. 
	Since~$\level(v) < m$ and~$\gen(T) = md$, the type~$d$ simplex~$T$ is of the form 
	\begin{align*}
		T = \simplex{v_d,\dots,v_1\leveljump v}.
	\end{align*}
	By definition the vertices~$v_d,\dots,v_1$ are on the vertex patch boundary~$\partial \omega_m(v)$. 
	Let~$T'\in \omega_m^d(v)$ be a descendant of~$T$ containing the vertex~$v \in \vertices(T')$. 
	Algorithm~\ref{algo:genbased} and the definition of~$\omega_m^d(v)$ uniquely determine~$T'$ as
	\begin{align}
		\label{eq:ProofMidT}
		T' = \simplex{ \bsv([v_d,v]),\dots,\bsv([v_1,v])\leveljump v} \in \lambda_m^1(v). 
	\end{align} 
	Therefore, any other descendant~$T'' \in \omega_m^d(v)$ of~$T$  does not contain~$v$ and is thus not in~$\lambda_m^1(v)$.  	
	Furthermore, each bisection removes one vertex and hence each such~$T''$ contains at least one of the vertices~$\{v_1, \ldots, v_d\}$. 
	Thus, it touches the boundary~$\patchbd$ and belongs to~$\gamma_m^1(v)$.
Therefore, the layer~$\lambda_m^1(v)$ satisfies~\ref{itm:Layer1}, that is
	\begin{align*}
		\lambda_m^1(v) = \omega_m^d(v) \setminus \gamma_m^1(v) = \omega_m^d(v) \setminus \big( \omega_m^0(v) \cup \gamma_m^1(v) \big).
	\end{align*}
Due to the definition~$\Gamma_m^0(v) = \overline{\Omega}_m(v) = \bigcup \omega_m^d(v)$ taking the union of those sets proves the second identity in~\ref{itm:Layer1}, that is,
\begin{align*}
	\Lambda^1_m(v)  = \overline{\Gamma_m^0(v) \setminus \Gamma_m^1(v)}.
\end{align*}
Moreover, Algorithm~\ref{algo:genbased} shows that any~$T'' \in \gamma_m^1(v)$ contains the vertex~$\bsv([v_1,v])$. 
According to~\eqref{eq:ProofMidT} this vertex is also contained in~$T'\in \lambda_m^1(v)$, and consequently it is in the interface~$\Lambda_m^1(v)\cap \Lambda_m^2(v)$. 
This yields the first statement in~\ref{itm:Layer1b}. 

	In addition, Algorithm~\ref{algo:genbased} shows that all possible new vertices resulting from~$d$ bisections of~$T$ are of the form~$\bsv([v_j,v])$ or~$\bsv([v_j,v_k])$ with~$1\leq j < k \leq d$. 
	Indeed, vertices that result from the bisection of an edge that contains such a new vertex have higher level than~$m+1$ and thus cannot be created after only~$d$ bisections of~$T$.  
	All vertices~$\bsv([v_j,v])$ with~$j=1,\dots,d$ are on the interface~$\Lambda_m^1(v) \cap \Lambda_m^2(v)$. 
	The vertices~$\bsv([v_j,v_k])$ with~$1\leq j < k \leq d$ are on the boundary~$\partial \omega_m(v)$. 
	This yields the second statement in~\ref{itm:Layer1b}. 
	Any simplex~$T'\in \lambda_m^1(v)$ satisfies~$\gen(T') = (m+1)d$, which shows~\ref{itm:Layer2}.
	Moreover, any new vertex~$v'\in \vertices(\omega_m^d(v)) \setminus \vertices(\omega_m^0(v))$ has level~$\level(v') = m + 1$. 
	In particular all vertices on the interface~$\Lambda_m^1(v) \cap \Lambda_m^2(v)$ are of level~$\level(v') = m + 1$, which yields~\ref{itm:Layer3}. 
	This concludes the proof of the base case.
	
	\textit{Induction step.}
	Let~$\ell > 1$ and let~$T\in \gamma_m^{\ell-1}(v)$.
	The induction hypothesis in~\ref{itm:Layer1b} shows that~$T$ has a vertex~$v'\in \vertices(T)$ with~$v'\in \Lambda^{\ell-1}_m(v) \cap \Lambda^{\ell}_m(v)$. 
	Since by definition~$T \in \gamma_m^{\ell-1}(v)$ touches the boundary~$\patchbd$ and is in~$\omega_m^{(\ell-1)d}(v)$, it is of generation~$(m+\ell-1)d$, and in particular of type~$d$. 
	Algorithm~\ref{algo:genbased} shows that all descendants of $T$ with generation less or equal to~$(m+\ell) d = \gen(T) + d$ share the bisection vertex~$\bsv(\bse(T))$ of $T$. 
	Moreover, one of the descendants~$T' \in \omega_m^{\ell d}(v)$ of~$T$ contains the vertex~$v'$. 
	This ensures that
	\begin{align}
	\label{eq:ProofLayer3}
	\begin{cases}
		\text{each simplex }T'' \in \omega_m^{\ell d}(v) \text{ with ancestor $T \in \gamma_m^{\ell-1}(v)$ }\\
		\text{touches a simplex }T' \in \omega_m^{\ell d}(v) \text{ with } v' \in T' \cap \Lambda^{\ell-1}_m(v).
	\end{cases}
	\end{align}
Any~$T'' \in \omega_m^{\ell d }(v)$ with ancestor~$T\in \gamma_m^{\ell-1}(v)$ satisfies~$\gen(T'') \leq (m + \ell) d$ and according to Algorithm~\ref{algo:genbased} it shares at least one vertex~$v_\textup{old}$ with its ancestor~$T$. 		
Due to the induction hypothesis in~\ref{itm:Layer1b} the vertex~$v_\textup{old}$ of~$T \in \gamma_m^{\ell-1}(v)$ belongs either to the boundary~$\partial \omega_m(v)$ or to the interface~$\Lambda_m^{\ell-1}(v) \cap \Lambda_m^\ell(v)$, that is,
	\begin{align}
		\label{eq:ProofLayer3b}
		T'' \cap \partial \omega_m(v) \neq \emptyset\qquad\text{or}\qquad v_\textup{old} \in \Lambda_m^{\ell-1}(v) \cap \Lambda_m^\ell(v).
	\end{align} 
	We show that both properties in~\eqref{eq:ProofLayer3b} cannot occur for the same simplex~$T'' \in \omega_m^{\ell d}(v)$ by contradiction. 
	For is, we assume that~$v_\textup{old} \in \Lambda_m^{\ell-1}(v) \cap \Lambda_m^\ell(v)$ and there exists a vertex~$v'' \in \vertices(T'')$ with~$v'' \in \partial \omega_m(v)$. 
	Since~$T''\cap \patchbd \neq \emptyset$, the simplex~$T''$ is generated by~$d$ bisections of~$T$, and according to Algorithm~\ref{algo:genbased} the vertex~$v_{\textup{old}}$ is the only vertex shared by~$T$ and~$T''$. 
	Consequently, the vertex~$v'' = \bsv(\simplex{a,b})$ results from the bisection of an edge with vertices~$a,b\in \vertices(T) \cap \partial \omega_m(v)$ on the boundary.
	All vertices except one in the type~$d$ simplex~$T$ are of level~$m+\ell-1$. 
	Moreover, the induction hypothesis in~\ref{itm:Layer3} shows that the vertex~$v_\textup{old}$ on the interface~$\Lambda_m^{\ell-1}(v) \cap \Lambda_m^\ell(v)$ is of level~$m+\ell-1$. 
	Thus, Lemma~\ref{lem:boundary-vertices}~\ref{itm:boundary-vertices-prop} shows that the face consisting of the vertices~$a,b,v_\textup{old}$ reads
	\begin{align*}
		f = \simplex{ a, v_\textup{old} \leveljump b}\qquad\text{or}\qquad f = \simplex{a,b,v_\textup{old}}.
	\end{align*}
	By Algorithm~\ref{algo:subsimplex} in none of the two cases~$[a,b]$ is the bisection edge. 
	Thus, $v''$~cannot be created without removing the vertex~$v_\textup{old}$ from the face (and so the ancestor of $T''$) before. 
	Hence, $T''$~cannot contain~$v_{\textup{old}}$ and $v''$~at the same time, which is a contradiction. 
	In combination with~\eqref{eq:ProofLayer3b} this shows that 
	\begin{align}
		\label{eq:ProofMainHelp}
		\text{either }\quad T'' \cap \partial \omega_m(v) \neq \emptyset\quad\text{or}\quad v_\textup{old} \in \Lambda_m^{\ell-1}(v) \cap \Lambda_m^\ell(v).
	\end{align} 
This proves that~$T''$ belongs either to~$\gamma_m^\ell(v)$ or to $\lambda_m^\ell(v)$, which yields~\ref{itm:Layer1}. 
The property in~\ref{itm:Layer2} is a direct consequence of~\ref{itm:Layer1}. 
Moreover, due to the conformity in Proposition~\ref{prop:regTria} any vertex~$v'\in \vertices(\omega_m^\infty(v))$ with~$v' \in \Lambda_m^\ell(v) \cap \Lambda_m^{\ell+1}(v)$ belongs to a simplex~$T' \in \lambda_m^\ell(v)$. 
Thus by~\ref{itm:Layer2} we find that $\level(v') \leq \level(T') = m+\ell$.  
On the other hand, by~\ref{itm:Layer1} we have that~$v'\in \Lambda_m^\ell(v) \cap \Lambda_m^{\ell+1}(v) \subset \Gamma_m^{\ell-1}(v)$. Then the induction hypothesis~\ref{itm:Layer1b} shows that~$v' \notin \vertices(\gamma_m^{\ell-1}(v))$. 
Therefore~$v'$ results from the bisection of a simplex in~$\gamma_m^{\ell-1}(v)$ or a descendant thereof, which yields that~$m+\ell \leq \level(v')$. 
Combining both inequalities we infer~\ref{itm:Layer3}. 

It remains to show~\ref{itm:Layer1b}.
Let~$T'' \in \gamma_m^\ell(v)$ be a simplex with ancestor~$T\in \gamma_m^{\ell-1}(v)$. 
According to Algorithm~\ref{algo:genbased} the simplex~$T''$ contains the bisection vertex of~$T$. 
Moreover, there exists a descendant~$T' \in \lambda_m^\ell(v)$ of~$T$ that contains this vertex as well.  
Thus, this bisection vertex is on the interface~$  \Lambda_m^{\ell}(v) \cap  \Gamma_m^{\ell}(v) = \Lambda_m^{\ell}(v) \cap \Lambda_m^{\ell+1}(v) $, where the identity is a consequence of~\ref{itm:Layer1}.  
This is the first statement in~\ref{itm:Layer1b}. 

To prove the second statement in~\ref{itm:Layer1b}, we consider~$T'\in \gamma_m^\ell(v)$ with oldest vertex~$v_{\textup{old}}$ and ancestor~$T \in \gamma_m^{\ell-1}(v)$, which are both of type~$d$ and share exactly one vertex, namely~$v_\textup{old}$.  
We use the property that any vertex~$v' \in \vertices(T')\setminus \lbrace v_\textup{old}\rbrace$ results from the bisection of an edge~$[a,b]$ with vertices~$a,b\in \vertices(T)$, cf.~Algorithm~\ref{algo:genbased}.
Due to the induction hypothesis~\ref{itm:Layer1b} these vertices satisfy that
\begin{align*}
a,b \in \partial \omega_m(v) \cap \big(\Lambda_m^{\ell-1}(v) \cap \Lambda_m^{\ell}(v)\big).
\end{align*}
If both vertices~$a,b\in \partial \omega_m(v)$ are on the boundary, then the vertex~$v'$ is on the boundary as well. 
If~$a \in \big(\Lambda_m^{\ell-1}(v) \cap \Lambda_m^{\ell}(v)\big)$ and~$b\in \partial \omega_m(v)$, then there exists a simplex~$T'' \in \lambda_m^\ell(v)$ that contains the edge~$\simplex{v'\leveljump a}$ and in consequence, one has~$v' \in T'' \subset \Lambda_m^{\ell}(v)$ and~$v' \in T' \subset \Gamma_m^\ell(v)$. 
In particular, we have 
\begin{align*}
v' \in \Lambda_m^\ell(v) \cap \Gamma_m^\ell(v) = \Lambda_m^\ell(v) \cap \Lambda^{\ell+1}(v).
\end{align*}
A similar argument applies if we exchange the role of~$a$ and~$b$. 
The case~$a,b \in \Lambda_m^{\ell-1}(v) \cap \Lambda_m^{\ell}(v)$ cannot occur, since the edge~$[a,b]$ belongs to some simplex in~$\lambda_m^{\ell-1}(v)$ and this simplex is not bisected due to the definition of the conforming triangulation~$\omega_m^\infty(v)$, cf.~Proposition~\ref{prop:regTria}. 
Combining these observations yields the second statement in~\ref{itm:Layer1b} for all~$v'\in \vertices(T')\setminus \lbrace v_\textup{old}\rbrace$. If~$v' = v_\textup{old}$, the induction hypothesis yields~$v' \in \partial\omega_m(v)$. 
This concludes the proof. 
\end{proof}
	
\begin{remark}[Sharpness of $\gamma \leq 2$]
	\label{rem:sharpGradingEst}
		The layer structure presented in Lemma~\ref{lem:Layers} provides an example of a chain of intersecting simplices~$T_0,T_1,\dotsc \in \tria_m^\infty(v)$ with
		\begin{align*}
			\level(T_N) - \level(T_0) = N\qquad\text{for all }N\in \mathbb{N}.
		\end{align*}	
		With the considerations in Section~\ref{subsec:meshGrading} this shows that grading~$\gamma=2$ in Theorem~\ref{thm:main-grading2} is attained by the auxiliary triangulations. 
		Hence, the estimate~$\gamma\leq 2$ in Theorem~\ref{thm:main-grading2} is optimal in the sense that the upper bound is sharp.
	\end{remark}
	We conclude this subsection with the following observation.
	\begin{lemma}[Bisection edge]
		\label{lem:RefEdgeOnBdd}
		Let~$v\in \vertices(\mtree)$ with~$\level(v) < m \in \mathbb{N}_0$.
		Let the simplex~$T \in \lambda^\ell_m(v)$ be in the $\ell$-th layer  of~$\omega_k^\infty(v)$, for some $\ell \in \mathbb{N}$, with bisection edge~$\bse(T) \eqqcolon \simplex{a,b} \in \edges(T)$.
		\begin{enumerate}
			\item 
			If~$\type(T) \neq d$, then~$\bse(T)$ is on the interface~$\Lambda^{\ell-1}_m(v)\cap  \Lambda^{\ell}_m(v)$. 
			\item 
			If~$\type(T) = d$, then~$\bse(T)$ crosses the layer, in the sense that
			\begin{align*}
				b \in  \Lambda^{\ell-1}_m(v) \cap  \Lambda^{\ell}_m(v)\quad \text{and}\quad a\in \Lambda^\ell_m(v) \cap  \Lambda^{\ell+1}_m(v).
			\end{align*}
		\end{enumerate}
	\end{lemma}
	\begin{proof}
		Combining  Lemma~\ref{lem:propertiesGenTypeLvl}~\ref{itm:simplex-refEdge} and Lemma~\ref{lem:Layers}~\ref{itm:Layer3} yields the claim.
	\end{proof}

	\subsection{Finer triangulation}
	\label{subsec:FinerTria}
	In this subsection we present conditions under which the auxiliary triangulation defined in the previous subsection is finer than a given triangulation~$\tria\in \BisecT$. 
	The proof uses further properties of the auxiliary triangulation involving the notion of pre-diamonds and edge patches. The latter reads for any edge~$e\in \edges(\tria)$ and triangulation~$\tria\in \BisecT$ 
	\begin{align*}
		\omega_\tria(e) \coloneqq \lbrace T\in \tria \colon e\in \edges(T)\rbrace.
	\end{align*}
	
	\begin{definition}[Pre-diamond]
		\label{def:preDiamond}
		For~$e\in \edges(\tria)$ and~$\tria\in \BisecT$ we call the edge patch~$\omega_\tria(e)$ \emph{pre-diamond}, if~$e$ is the bisection edge of each simplex~$T\in \omega_\tria(e)$.
	\end{definition}
	~
	In the literature the collection of simplices arising from the bisection of a pre-diamond is referred to as diamond, cf.~\cite{WeissDeFloriani11}, which motivates the name pre-diamond.
	Let~$v\in \vertices(\mtree)$ and~$m\in \mathbb{N}_0$ with~$\level(v)< m$. 
	Recall the definition of the auxiliary triangulation $	\tria_m^\infty(v)$, cf.~\eqref{def:tria_inf}, as
	\begin{align*}
		\tria_m^\infty(v) \coloneqq \bigcup_{\ell=0}^\infty \bigcap_{j = \ell}^\infty \tria_m^j(v) \in \overline{\BisecT}.
	\end{align*} 
	We call all edges~$e\in \edges(\omega_m(v))$ of simplices in~$\omega_m(v)$ with~$\type(e) = 1$ and~$v \in e$ \emph{type one diagonals} of~$\omega_m(v)$. 
	Before establishing the main result of this subsection we need the following two results. 
	They characterize edges on type one diagonals as pre-diamonds and investigate chains of such edge patches. 
	
	\begin{lemma}[Pre-diamonds in $\tria_m^\infty(v)$]
		\label{lem:FinerTriaPreDiamonds}
		Let~$v\in \vertices(\mtree)$  be a vertex and let~$m  \in \mathbb{N}_0$ with~$\level(v)<m$. 
		Then, for all edges~$e\in \edges(\omega_m^\infty(v))$ the following equivalence holds: 
		\begin{align*}
			\omega_{\tria_m^\infty(v)}(e) \text{ is a pre-diamond if and only if }e\text{ lies on a type one diagonal of } \omega_m(v).
		\end{align*}
	\end{lemma}
	\begin{proof} 
			Let~$v\in \vertices(\mtree)$  be a vertex and let~$m \in \mathbb{N}_0$ be such that~$\level(v)<m$. 
			Furthermore, let~$e \in \edges(\tria_m^\infty(v))$ be an edge with edge patch~$\omega(e) \coloneqq \omega_{\tria_m^\infty(v)}(e)$. 
		
		\textit{Step 1 (Type one edge).}
		Suppose that~$\type(e) \neq 1$. Lemma~\ref{lem:propertiesGenTypeLvl}~\ref{itm:refe-typeOneVertex} shows that any~$T\in \lambda^\ell_m(v)$ with~$\ell \in \mathbb{N}$ and $e = \bse(T)$ satisfies~$\type(T) \neq d$.
		Lemma~\ref{lem:RefEdgeOnBdd} yields
		\begin{align*}
			e \subset \Lambda^{\ell-1}_m(v) \cap \Lambda_m^\ell(v).
		\end{align*}
		Hence, there exists a simplex~$T' \in {\lambda}^{\ell-1}_m(v)$ with~$e\in \edges(T')$ and by Lemma~\ref{lem:Layers}~\ref{itm:Layer2} we have that~$\level(T') = m + \ell - 1 < m + \ell = \level(T)$. 
		Since by definition simplices in a pre-diamond have the same generation and level, the patch~$\omega(e)$ containing~$T$ and~$T'$ cannot be a pre-diamond. 
		Thus, by contraposition we obtain that
		\begin{align}
			\label{eq:proofStep1}
			\text{if }\omega(e) \text{ is a pre-diamond, then we have $\type(e) = 1$}.
		\end{align}
		
			\textit{Step 2 (Layer).} 	
		If~$\omega(e)$ is a pre-diamond, then by~\eqref{eq:proofStep1} the type of~$e$ is one. 
		By the definition of pre-diamonds and  by Lemma~\ref{lem:propertiesGenTypeLvl}~\ref{itm:refe-typeOneVertex} all~$T \in \omega(e)$ are of type~$d$. 
		They also have the same level, and thus, according to Lemma~\ref{lem:Layers}~\ref{itm:Layer2}, are contained in one layer~$\lambda_m^\ell(v)$ for some~$\ell\in \mathbb{N}$. 
		Indeed, by Lemma~\ref{lem:Layers}~\ref{itm:Layer3} we know that~$\omega(e) \subset \lambda_m^\ell(v)$ for~$\ell \in \mathbb{N}$ such that~$\level(e) = m + \ell$. 
		Collecting these properties, for any~$e \in \edges(\lambda_m^\ell(v))$ with~$\ell \in \mathbb{N}$ we have that
		\begin{align}
			\label{eq:proofStep2}
			\begin{cases}
				\text{if }\omega(e) \text{ is a pre-diamond, then we have $\omega(e) \subset \lambda_m^\ell(v)$,}\\
				\text{$\type(T) = d$ for all $T\in \omega(e)$, and $\type(e) = 1$}.
			\end{cases}
		\end{align}
		
		\textit{Step 3 (Ancestors).}
		Let~$T\in \lambda_m^\ell(v)$ be a simplex with~$\ell \in \mathbb{N}$. 
		Then, by Lemma~\ref{lem:Layers}~\ref{itm:Layer1} it has a type~$d$ ancestor 
		\begin{align*}
			\overline{T} \in \gamma_m^{\ell-1}(v). 
		\end{align*}
		\textit{Case 1.}
		Suppose that the ancestor~$\overline{T}$ has at least two vertices not contained in the boundary~$\patchbd$. 
		Hence, it has an edge~$e'$ on the interface~$\Lambda^{\ell-1}_m(v) \cap \Lambda^{\ell}_m(v)$. 
		The conformity of~$\tria_m^\infty(v)$ in Proposition~\ref{prop:regTria} ensures that~$e'$ is not bisected in the construction of~$\omega_m^\infty(v)$.   
		Algorithm~\ref{algo:genbased} shows that after~$d$ bisections of the type~$d$ simplex~$\overline{T}$ all edges in~$\edges(\overline{T})$ would be bisected, in particular the edge~$e'$. 
		Thus,~$T$ results from strictly less than~$d$ bisections of~$\overline{T}$ and hence~$\type(T) \neq d$.\\
		\textit{Case 2.}
		Suppose now that the ancestor~$\overline{T}$ has~$d$ vertices on the boundary~$\patchbd$. 
		Therefore, any descendant resulting from~$d-1$ bisections of~$\overline{T}$ has at least one vertex on the boundary. 
		Hence, according to the definition of~$\omega^{\ell d}_m(v)$ any descendant of~$\overline{T}$ in~$\lambda_m^\ell(v)$ has a generation larger or equal to~$\gen(\overline{T}) + d = (m+\ell)d$, and this holds in particular for~$T$.  Lemma~\ref{lem:Layers}~\ref{itm:Layer2} implies that~$\level(T) = m+\ell$, and thus~$\gen(T) = (m+\ell)d$. 
		This yields~$\type(T) = d$.
		
		Altogether, for arbitrary~$T \in \lambda_m^\ell(v)$ and its ancestor~$\overline{T} \in \gamma_m^{\ell-1}(v)$ one has that
		\begin{align}
			\label{eq:Step3Ancestors_1}
			\type(T) = d \;\text{ if and only if }
			\; \text{all vertices but one of  $\overline{T}$ are in $\patchbd$.}
		\end{align}
		Let us show for any~$e \in \edges(\lambda_m^\ell(v))$ with~$\ell \in \mathbb{N}$ the following equivalence:
		\begin{align}
			\label{eq:Step3Ancestors}
			\begin{cases}
				\text{$\omega(e) \subset \lambda_m^\ell(v)$ is a pre-diamond if and only if} \\
				\text{$\type(e) = 1$ and the ancestor in $\gamma_m^{\ell-1}(v)$ of each  }
				\\	
			 \text{simplex in $\omega(e)$ 
			 	has all vertices but one on $\partial \omega_m(v)$.}
			\end{cases}
		\end{align}
		Indeed, the implication is an immediate consequence of~\eqref{eq:proofStep2} and~\eqref{eq:Step3Ancestors_1}. 
		For the reverse implication, thanks to Lemma~\ref{lem:Layers}~\ref{itm:Layer3} and with~$\type(e)=1$ one can infer that~$\omega(e) \subset \lambda_m^\ell(v)$. 
		Lemma~\ref{lem:Layers}~\ref{itm:Layer2} shows that~$\level(T) = m+\ell$ for any~$T \in \omega(e)$. 
		Furthermore, since for any~$T \in \omega(e)$ its ancestor~$\overline{T}$ in~$\gamma_m^{\ell-1}(v)$ has all vertices but one on~$\patchbd$, by~\eqref{eq:Step3Ancestors_1} it follows that~$\type(T) = d$ for any~$T \in \omega(e)$. 
		Now, all~$T \in \omega(e)$ have the same type and the same level, and thus the same generation, which implies that~$\bse(T) = e$ for any~$T \in \omega(e)$. 
		This shows that $\omega(e)$ is a pre-diamond. 
	
		\textit{Step 4 (Ancestor patch)}. 
		We want to prove that for any~$e \in \edges(\lambda_m^\ell(v))$ we have
		\begin{align}
			\label{eq:Step4AncestorPatch}
			\begin{cases}
				\text{if $\omega(e)$ is a pre-diamond, then the union of its ancestors in }\gamma_m^{\ell-1}(v) \\
				\text{forms an edge patch $\omega(\overline{e}) \subset  \gamma_m^{\ell-1}(v)$ with type one edge $\overline{e} \supset e$}.
			\end{cases}
		\end{align}	
		Let~$\omega(e)\subset \lambda_m^\ell(v)$ be a pre-diamond. 
		By~\eqref{eq:proofStep1} we have~$\type(e) = 1$.  
		Let~$T \in \omega(e)$ be arbitrary with  ancestor~$\overline{T} \in \gamma_m^{\ell-1}(v)$.  
		The equivalence in~\eqref{eq:Step3Ancestors} yields that~$\overline{T}$ has all vertices but one on the boundary. 
		We refer to the non-boundary vertex of~$\overline{T}$ as~$\overline{v} \notin \patchbd$. 
		By the bisection routine, due to~$\type(\overline{T})=d$ the type one edge $e$ results from the first bisection of~$\overline{T}$, and hence~$\bsv(\bse(\overline{T})) \in e$. 
		Since~$e$ does not touch the boundary~$\patchbd$ and~$\overline{v}$ is the only vertex of~$\overline{T}$ that is not on~$\patchbd$, $e$ is uniquely determined as~$e = \simplex{ \bsv(\bse(\overline{T})) \leveljump \overline{v}}$. 
		Furthermore, we have~$e \subset \overline{e} \coloneqq \bse(\overline{T})$, which is of type one since~$\type(\overline{T}) = d$, cf.~Lemma~\ref{lem:propertiesGenTypeLvl}~\ref{itm:refe-typeOneVertex}. 
		Also, since~$T \in \omega(e)$ is arbitrary this implies that the union of ancestors of~$\omega(e)$ in~$\gamma_m^{\ell-1}(v)$ form the edge patch~$\omega(\overline{e})\coloneqq \omega_{\tria_{m}^{(\ell-1)d}(v)}(\overline{e})$. 
		This proves~\eqref{eq:Step4AncestorPatch}. 
		
		\textit{Step 5 (Boundary patch)}.  
		For edge patches~$\omega(\overline{e}) \subset \gamma_m^{\ell-1}(v)$ with~$\ell \in \mathbb{N}$ and with some type one edge $\overline{e}$ we claim the following property: 
		\begin{align}
			\label{eq:Step5BoundaryPatch}
			\begin{cases}
				\text{The edge patch $\omega(\overline{e})$ has all vertices but one on the boundary $\partial \omega_m(v)$}\\
				\text{if and only if $\overline{e}$ is on a type one diagonal}. 
			\end{cases}
		\end{align}
		We prove this claim by induction. 
		The base case~$\ell = 1$ follows by definition of~$\omega_m(v) = \gamma_m^0(v)$. 
		Hence, we aim to verify the statement for~$\ell>1$. 
		Let~$\omega(\overline{e}) \subset \gamma_m^{\ell-1}$ be the edge patch of~$\overline{e} \eqqcolon \simplex{v_1 \leveljump v_0}$ with~$\type(\overline{e}) = \type(v_1)=1$. 
		
		\textit{Step 5.1 (``$\,\Rightarrow$'').}
		Suppose~$\omega(\overline{e})$ has all vertices but one on the boundary~$\patchbd$. 
	Lemma~\ref{lem:boundary-vertices}~\ref{itm:boundary-vertices-type1} shows that the type one vertex~$v_1 \notin \patchbd$, and hence~$v_0 \in \patchbd$. 
	
The vertex~$v_1$ results from the bisection of a type one edge~$\doublebar{e}$ with~$\level(\doublebar{e}) = \level(\overline{e})-1$,  cf.~Lemma~\ref{lem:propertiesGenTypeLvl}~\ref{itm:refe-typeOneVertex}. 
Thus, the ancestors in~$\gamma_m^{\ell-2}(v)$ of the simplices in~$\omega(\overline{e})$ belong to the edge patch~$\omega(\doublebar{e}) \subset \gamma_m^{\ell-2}(v)$. 

Let us now show that each ancestor~$\doublebar{T}\in \omega(\doublebar{e})$ has all vertices but one on the boundary~$\patchbd$. 
Indeed, if such a~$\doublebar{T}$ has two interior vertices, then by the bisection routine in Algorithm~\ref{algo:genbased} its descendant in~$\omega(\overline{e})$ contains at least two interior vertices as well. 
This contradiction verifies our claim. 

Since~$v_1 \not\in \partial \omega_m(v)$, the edge~$\doublebar{e}$ is not on the  boundary~$\patchbd$, that is, there exists a vertex~$v_{\doublebar{e}}\in \vertices(\doublebar{e})$ with~$v_{\doublebar{e}}\not\in \partial \omega_m(v)$. 
In case~$\overline{e} \not\subset\doublebar{e}$, then the edge~$[v_{\doublebar{e}},v_0] \in \edges(\gamma_m^{\ell-2}(v))$ exists. 
As~$\doublebar{T}$ has all vertices but one on the boundary, its descendants in~$\gamma_m^{\ell-1}(v)$ result from~$d$ bisections of~$\doublebar{T}$. 
Thus, there is an ancestor~$\overline{T} \in\omega(\overline{e})$ with~$\bsv[v_{\doublebar{e}},v_0] \in \vertices(\omega(\overline{e}))$.  
Since~$v_{\doublebar{e}}$ is not on the boundary~$\patchbd$, this vertex is not on the boundary~$\patchbd$ either. 
Furthermore, the simplex~$\overline{T}$ contains the interior vertex~$v_1$. 
Thus, it has two vertices that are not on the boundary~$\patchbd$. 
This  contradicts the assumptions on~$\omega(\overline{e})$. 
Therefore, we have that~$\overline{e} \subset \doublebar{e}$. 
Consequently, by the induction hypothesis~$\overline{e} \supset \doublebar{e}$ lies on a type one diagonal.

\textit{Step 5.2 (``$\,\Leftarrow$'').} 
Let~$\overline{e}$ be on a type one diagonal. 
There exists an edge patch~$\omega(\doublebar{e}) \subset \gamma_m^{\ell - 2}(v)$ with~$\overline{e} \subset \doublebar{e}$ and $\level(\overline{e}) = \level(\doublebar{e})+1$.
Since~$\doublebar{e}$ is on a type one diagonal, by the induction hypothesis all vertices but one of~$\omega(\doublebar{e})$ are on the boundary~$\patchbd$. 
The bisection routine in Algorithm~\ref{algo:genbased} then shows that all vertices in ~$\vertices(\omega(\overline{e}))\setminus \lbrace v_1\rbrace$ must be on the boundary~$\patchbd$.
		
\textit{Step 6 (Conclusion).}
		Combining \eqref{eq:proofStep2} and~\eqref{eq:Step3Ancestors}--\eqref{eq:Step5BoundaryPatch} concludes the proof. 
	\end{proof}

Combining Lemmas~\ref{lem:Layers} and ~\ref{lem:FinerTriaPreDiamonds} leads to the following observation on chains of type one pre-diamonds. 

\begin{corollary}[Chain of pre-diamonds]
	\label{cor:ChainTypeOneDia}
Let~$v\in \vertices(\mtree)$ and let~$m\in \mathbb{N}_0$ be such that~$\level(v) < m$. 
Then each type one diagonal~$\overline{e} \in \edges(\omega_m(v))$ is covered by a chain of type one edges~$e_1,e_2,\dots \in \edges(\tria_m^\infty(v)) \cap \overline{e}$ with~$v \in e_1$ that satisfies for all~$\ell \in \mathbb{N}$
\begin{align}
	\label{eq:PropTypeOneChain}
 e_\ell \cap e_{\ell+1} \neq \emptyset, \quad\level(e_\ell) = m +\ell,\quad \omega_{\tria_m^\infty(v)}(e_\ell)\text{ is a pre-diamond}.
\end{align}
\end{corollary}
	After collecting these auxiliary results we are now in the position to verify the main result in this section. 
	More specifically, we show that under suitable conditions the auxiliary triangulation is finer than a given triangulation. 
	For this purpose, let us recall the definition of the maximal level jump in a vertex, cf.~ Definition~\ref{def:maxLvlJump-2}, 
		\begin{align*}
			\jump(v) &\coloneqq \max_{\tria \in \BisecT} \max_{e,e' \in \edges(\tria)\colon v \in e \cap e'}
			\bigabs{\level(e) - \level(e')} \quad \text{ for } v \in \vertices(\mtree). 
		\end{align*}
	
	\begin{proposition}[Auxiliary triangulation]
		\label{prop:FinerTria}
		Let~$\tria \in \BisecT$ with~$\tria_0$ colored and let~$v \in \vertices(\tria)$ be a vertex. 
		Suppose that~$m \in \setN_0$ satisfies
		\begin{align}
			\label{eq:ass-finer-tria}
			\begin{aligned}
				\level(v) &< m 
				\qquad
				\text{and} \qquad \jump(v) + \min_{e \in \edges(\tria)\colon v \in e} \level(e) \leq m+1. 
			\end{aligned}
		\end{align}
		Then the triangulation~$\tria^\infty_m(v)$ is finer than~$\tria$ in the sense that~$\tria \leq \tria^\infty_m(v)$.
	\end{proposition}
	\begin{proof}
		By assumption we have~$\level(v)<m$ and~\eqref{eq:assump-v-aux} is satisfied. 
		Thus, the triangulation~$\tria^\infty_m(v)$ is well defined and it remains to show that~$\tria \leq \tria^\infty_m(v)$.
		We proceed by contradiction and assume that $\tria \not\leq \tria^\infty_m(v)$. 
		Without loss of generality we may assume that~$\tria$ is minimal in the sense that any proper coarsening~$\tria'$ of~$\tria$, i.e.,~$\tria' \leq \tria$ with $\tria' \neq \tria$, satisfies~$\tria' \leq \tria^\infty_m(v)$. 
		Due to this minimality there exists exactly one edge~$\widetilde{e} \in \tria^\infty_m(v)$ that is bisected in~$\tria$ 
		and the edge patch
\begin{align*}		
		 \omega_{\tria_m^\infty(v)}(\widetilde{e}) \,\text{ is a pre-diamond}.
\end{align*}
		By the properties of~$\tria_m^\infty (v)$ in Lemma~\ref{lem:FinerTriaPreDiamonds} the edge~$\widetilde{e}$ lies on a type one diagonal of~$\omega_m(v)$. 
		Thus, by Corollary~\ref{cor:ChainTypeOneDia} there exists a chain of type one edges~$e_1,\dots, e_N \in \edges(\tria_m^\infty(v))$ with~$N \in \mathbb{N}$, $\widetilde{e} = e_N$, with consecutive levels, and pre-diamond edge patches. 
		By the minimality of~$\tria$ none of the edges $e_1, \ldots, e_{N-1}$ is bisected in~$\tria$. 
		Let us also show that all edges~$e_1, \ldots, e_{N-1}$ are contained in~$\edges(\tria)$. 
		The identity~$\level(e_\ell) = m+\ell$ and the tree structure of bisections of the type one diagonal show for all~$\ell=2,\dots,N$ that each edge~$e_\ell$ is a child of the sibling of~$e_{\ell-1}$. 
		In particular, the edge~$\widetilde{e} = e_N$ and its descendant in~$\tria$ can only exist, if the edges~$e_1, \ldots, e_{N-1}$ exist. 
		Thus, we have~$e_1, \ldots, e_{N-1} \in \edges(\tria)$.
		
		We denote by~$e \in \edges(\tria)$ the edge resulting from the bisection of~$e_N = \widetilde{e}$ which touches~$e_{N-1}$ in the sense that~$e \cap e_{N-1} \eqqcolon v_{N-1}\in \vertices(\tria)$. 
		This edge satisfies 
\begin{align*}
\level(e) = \level(e_N) + 1 = m+N+1.
\end{align*}
		If~$N = 1$, then we have that~$v \in e$. 
		Using the previous identity, the definition of~$\jump(v)$ and the assumption on~$m$ leads to the  contradiction
		\begin{align*}
m+2 = \level(e) \leq \max_{e' \in \edges(\tria)\colon v \in e'} \level(e') \leq \jump(v) + \min_{e' \in \edges(\tria)\,:\, v \in e'} \level(e') = m+1. 
		\end{align*}
		Hence, it remains to consider the case~$N \geq 2$. 
		Recall that~$e_1, \ldots, e_{N-1} \in \edges(\tria)$ and none of those edges is bisected in~$\tria$. 
		In~$\tria_m^\infty(v)$ their edge patches are pre-diamonds, but the patches in~$\tria$ could be larger. 
		In order to show that this is not the case let~$T \in \omega_\tria(e_{N-1})$ be a simplex with a type one edge~$e' \in \edges(T)$. 
	 Lemma~\ref{lem:typeOneEdges} states that $\level(e') = \level(T) = \level(e_{N-1}) = m + N  -1$. 
	 Thus, we have the $\gensharp$-jump
\begin{align*}
	2d = d\, (\level(e) - \level(e')) = d\, (\levelsharp(e) - \levelsharp(e')) = \gensharp(e) - \gensharp(e').
\end{align*}
According to Lemma~\ref{lem:genSharpNonMacro} this shows that any type one edge~$e'\in \edges(T)$ coincides with the type one edge $e_{N-1}$. 
In other words, $T$ contains no other type one edges than~$e_{N-1}$. 
Thus, Lemma~\ref{lem:typeOneEdges} yields that~$\type(T) = d$. 
Combining this with~$\level(T) = m+N-1$ shows that $\gen(T) = d (m+N-1)$ and hence the edge patch $\omega_\tria(e_{N-1})$ is a pre-diamond. 
This argument can be applied inductively: Since~$\omega_\tria(e_{N-1})$ is a pre-diamond, it could be bisected without closure in~$\tria$. 
Then the same arguments as above show that~$\omega_{\tria}(e_{N-2})$ is a pre-diamond and so on.

	Let~$e_0 \in \edges(\tria)$ denote the edge with~$\level(e_0) = \min_{e' \in \edges(\tria)\,:\, v \in e'} \level(e')$. 
	The worst possible level jump at~$v$ is~$\jump(v)$. 
	However, since~$\omega_\tria(e_1)$ is a pre-diamond, the level jump in~$\tria$ at~$v$ is at most~$\jump(v)-1$. 
	Indeed, if the level jump in~$\tria$ would be~$\jump(v)$ a bisection of~$e_1$ would result in a level jump~$\jump(v)+1$ at $v$ which contradicts the definition of~$\jump(v)$. 
This observation leads to
\begin{align*}
	m+1 = \level(e_1) &\leq \level(e_0) + \jump(v)-1
	=  \jump(v) + \min_{e' \in \edges(\tria)\,:\, v \in e'} \level(e') -1 \leq  m.
\end{align*}
This contradiction concludes the proof. 
	\end{proof}
		
	\subsection{Edge chains in neighborhoods}
	\label{subsec:ChainsLeavingNgh}
	In the following we investigate chains of edges within a neighborhood. 
	For~$\tria\in \BisecT$ let a \emph{chain of edges}~$e_1,\dots, e_N \in \edges(\tria)$ with~$N \in \setN$ be connected by vertices~$v_0, \dots, v_N \in \vertices(\tria)$, i.e.,
	\begin{align}
		\label{eq:defRelatedNodalChain}
		e_j = [v_{j-1},v_j]\qquad\text{for all }j=1,\dots,N.
	\end{align} 
	A key property of such chains using the layer structure as introduced in~\eqref{eq:DefLambda} is stated in the following lemma and illustrated in Figure~\ref{fig:LayersLem}:
	There are no simplices in~$\mtree$ that contain both a vertex in~$\Lambda_m^\ell(v)$ and a vertex in~$\Lambda_m^{\ell + 2 + r}(v)$ with~$\ell,r\in \mathbb{N}_0$.  
	 Thus, an edge with one vertex in~$\Lambda_m^\ell(v)$ has either a vertex that is not contained in the neighborhood~$\Omega_m(v)$ or a vertex that is within the strip  of layers~$\bigcup^{\ell+1}_{r=\ell-1} \Lambda_m^r(v)$. 
	 
	\begin{lemma}[Passing through layers]
		\label{lem:leavingLayers}
Let~$\tria\in \BisecT$ be a triangulation with colored $\tria_0$. 
Let~$e_1,\dots,e_N \in \edges(\tria)$ with~$N \in \mathbb{N}$ be a chain of edges with corresponding vertex chain $v_0,\dots,v_N\in \vertices(\tria)$ as in \eqref{eq:defRelatedNodalChain}. 
	If the chain remains in the neighborhood~$\Omega_m(v_0)$ with~$\level(v_0) < m \in \mathbb{N}_0$ in the sense that~$v_0,\dots,v_{N}\in \Omega_m(v_0)$, then one has that
		\begin{align*}
			v_\ell \in  \bigcup_{j=0}^{\ell} {\Lambda}_m^j(v_0) \qquad\text{for all }\ell = 0,\dots,N.
		\end{align*} 
	\end{lemma}
	\begin{proof}
		Let the edge chain~$e_1,\dots,e_N \in \edges(\tria)$ and with vertices~$v_0,\dots,v_N\in \vertices(\tria)$ be as stated. 
		We prove the result for all vertices~$v_\ell$ with $\ell = 0,\dots,N-1$ by induction over $\ell$. 
		The base case~$v_0 \in {\Lambda}^0_m(v_0)$ follows by definition.
		For the induction step we assume that $v_{\ell-1}\in \bigcup_{j=0}^{\ell-1} {\Lambda}^j_m(v_0)$ for some~$\ell < N-1$. 
					Recall the definition of the boundary layer introduced in~\eqref{eq:PropLayer2} as
					\begin{align*}
						\gamma_m^{\ell}(v_0) = \lbrace T \in \omega^{\ell d}_m(v_0) \colon T \cap \partial \omega_m(v) \neq \emptyset\rbrace.
					\end{align*}
		Lemma~\ref{lem:Layers}~\ref{itm:Layer1b} implies that the vertices of all~$T\in \gamma_m^{\ell}(v_0)$ are either on the interface~$\Lambda^{\ell}_{m}(v_0) \cap \Lambda^{\ell+1}_{m}(v_0)$ or on the boundary~$\partial \omega_m(v_0)$.

		Let us define the set 
		\begin{align*}
			\Gamma_{m,\textup{int}}^\ell(v_0) \coloneqq
			\Omega_m(v_0) \setminus \bigcup_{j=0}^{\ell} {\Lambda}^j_m(v_0). 
		\end{align*}		
		For the special case of an interior vertex~$v_0 \notin \partial \Omega$ this set coincides with the interior of~$\Gamma_m^\ell(v_0)$.   	
		Using Lemma~\ref{lem:Layers}~\ref{itm:Layer1} all vertices in~$\tria$ that are in~$\Gamma_{m,\textup{int}}^\ell(v_0)$ result from bisections of simplices in~$\gamma_m^{\ell}(v_0)$. 
	In particular any~$T \in \mtree$ with a vertex in~$\Gamma_{m,\textup{int}}^\ell(v_0)$ is contained in~$\Gamma_{m}^\ell(v_0)$. 
	The latter set has no intersection with $\bigcup_{j = 0}^{\ell-1} \Lambda_m^j(v_0)$.  
	Hence, there is no edge with one vertex in~$\Gamma_{m,\textup{int}}^\ell(v_0) $ and one vertex in~$\bigcup_{j = 0}^{\ell-1}\Lambda_m^{j}(v_0)$. 
Since~$v_\ell$ is connected to~$v_{\ell-1} \in \Lambda_m^{\ell-1}(v_0)$ by an edge in~$\edges(\tria)$, this observation implies that $v_\ell \in \bigcup_{j=0}^{\ell} {\Lambda}^j_m(v_0)$ or that~$v_\ell \not\in \Omega_m(v_0)$ and hence concludes the proof.
	\end{proof}

	The previous lemma shows that a chain of edges has to `jump' over the remaining layers in order to leave a neighborhood. 
	Thus, such an edge cannot be too small, and its level cannot not be too high, which is made precise in the following proposition. 
	
	\begin{proposition}[Leaving a neighborhood]
		\label{prop:leavongOldNode}
		Let~$\tria\in \BisecT$ be a triangulation with colored~$\tria_0$ and chain of edges~$e_1,\dots,e_N\in \edges(\tria)$ with vertices~$v_0,\dots,v_N \in \vertices(\tria)$ as in~\eqref{eq:defRelatedNodalChain}. 
		Suppose the chain leaves a neighborhood~$\Omega_m(v_0)$ with~$\level(v_0) < m \in \mathbb{N}_0$ in the sense that~$v_j \in \Omega_m(v_0)$ for all~$j=0,\dots,N-1$ and that~$v_N \not \in \Omega_m(v_0)$. 
		Then the last edge~$e_N$ satisfies the estimate
		\begin{align*}
			\level(e_N) < m+N.
		\end{align*}
	\end{proposition}
	\begin{proof}
		Suppose that the edge chain~$e_1,\dots,e_N\in \edges(\tria)$ with~$\tria\in \BisecT$ and $N\in\mathbb{N}$ and related vertices~$v_{0},\dots,v_N \in \vertices(\tria)$ satisfies the assumptions as stated in the proposition. 
		Lemma~\ref{lem:leavingLayers} implies that 
		\begin{align*}
			v_{N-1} \in \bigcup_{\ell=0}^{N-1} {\Lambda}_m^\ell(v_0).
		\end{align*}
		Thus, according to Lemma~\ref{lem:Layers}~\ref{itm:Layer1} the edge~$e_N = [v_{N-1},v_N]$ crosses the boundary layer~$\Gamma_m^{N-1}(v_0) \coloneqq \bigcup \gamma_m^{N-1}(v_0)$ with
		\begin{align*}
			\gamma_m^{N-1}(v_0)\coloneqq \lbrace T \in \omega_m^{(N-1)d}(v_0) \colon T \cap \partial \Omega_m(v_0) \neq \emptyset\rbrace.
		\end{align*}
		Then Lemma~\ref{lem:typeDsimplices} states the existence of a simplex~$T\in \mtree$ with $e_N\in \edges(T)$ and $\gen(T) = \level(e_N)d$.
		Since~$e_N$ and therefore also~$T$ touches the boundary~$\partial \Omega_m(v_0)$ and the layer $\bigcup_{\ell=0}^{N-1} {\Lambda}_m^\ell(v_0)$, the simplex~$T$ is either in~$\gamma_m^{N-1}(v_0)$ or the ancestor of a simplex in~$\gamma_m^{N-1}(v_0)$. 
		Since all simplices in~$\gamma_m^{N-1}(v_0)$ have level~$N+m-1$, the level of~$T$ must be strictly smaller than~$m+N$ and we have that
		\begin{align*}
			&\level(e_N) \leq \level(T) < m+N.\qedhere
		\end{align*}
	\end{proof}
	Since we can control the level of edges leaving a neighborhood, it remains to investigate edge chains that stay inside a neighborhood. 
	This can be done in situations where the related auxiliary triangulation is finer than the given triangulation~$\tria$, a property derived in Proposition~\ref{prop:FinerTria}. 
	
	\begin{proposition}[Staying in a neighborhood]
		\label{prop:stayInNgh}
		Let~$\tria\in \BisecT$ be a triangulation with colored~$\tria_0$ and let~$e_1,\dots,e_N\in \edges(\tria)$ with~$N \in \mathbb{N}$ be a chain of edges with vertices~$v_0,\dots,v_N \in \vertices(\tria)$ as in~\eqref{eq:defRelatedNodalChain}. 
		Suppose that the chain does not leave the neighborhood~$\Omega_m(v_0)$ with~$\level(v_0) \leq m$ in the sense that~$v_0,\dots,v_N\in \Omega_m(v_0)$. 
		Moreover, let us assume that~$\tria \leq \tria_m^\infty(v_0)$. 
		Then we have that
		\begin{align*}
			\level(e_j) \leq m + j \quad \text{ for all } j=1,\dots,N.  
		\end{align*}
	\end{proposition}
	\begin{proof}
		Suppose that the assumptions of the proposition are valid and let~$j\in \lbrace 1,\dots,N\rbrace$ be arbitrary. 
		By Lemma~\ref{lem:leavingLayers} we know that 
		\begin{align*}
			v_{j-1} \in \bigcup_{\ell = 0}^{j-1} \Lambda_m^\ell(v_0)\qquad\text{and}\qquad 
			v_j \in \bigcup_{\ell = 0}^j \Lambda_m^\ell(v_0).
		\end{align*}
		Thus, due to the property~$\tria \leq \tria_m^\infty(v_0)$ the edge~$e_j = [v_{j-1},v_j]$ or one of its descendants is contained in a simplex~$T\in \lambda_m^\ell(v_0)$ for some~$\ell \leq j$. 
		The property of the layers in Lemma~\ref{lem:Layers}~\ref{itm:Layer2} shows that~$\level(e_j) \leq \level(T) \leq m+j$, which proves the statement.  
	\end{proof}

\section{Global level estimates}
\label{sec:GenDiffEdgeChain}
In this section we combine the local level estimates derived in Section~\ref{sec:locGenDiff} with the results on neighborhoods in  Section~\ref{sec:leavingTheNgh} to conclude our main result in Theorem~\ref{thm:main-grading2}.   
In Subsection~\ref{subsec:colored-T0} we prove the main theorem for colored initial triangulations. 
Then, in Subsection~\ref{sec:stevenson} we extend the result to initial triangulations satisfying the matching neighbor condition by Stevenson. 
In both cases we only have to show level estimates for chains of simplices, cf.~Proposition~\ref{pro:genest} and Proposition~\ref{pro:genest-stevenson} below. 
Indeed, then Lemma~\ref{lem:lvlDiffToMain} proves that the grading is in both cases $\gamma \leq 2$, which leads to Theorem~\ref{thm:main-grading2}.    

In the following we repeatedly consider an \emph{edge chain}~$e_0, \ldots, e_N \in \edges(\tria)$ for $N \in \mathbb{N}$, with corresponding \emph{vertex chain}~$v_0, \ldots, v_N \in \vertices(\tria)$, i.e., 
\begin{align}
	\label{eq:defRelatedNodalChain2}
e_j = [v_{j-1},v_j] \;\; \text{ and } \;\; v_{j-1} = e_{j-1} \cap e_{j}  \quad \text{ for } j = 1, \ldots, N. 
\end{align}
Compared to \eqref{eq:defRelatedNodalChain} only the first edge $e_0$ is added.

\subsection{Colored initial triangulation}
\label{subsec:colored-T0}
We start with level estimates for triangulations~$\tria\in \BisecT$ with colored initial triangulation~$\tria_0$. 
The following result uses the constants~$J_n \in \mathbb{N}_0$ defined in~\eqref{eq:defJn-2} and 
\begin{align}
	\label{eq:genconst}
\genconst \coloneqq 1 + \sum_{n=0}^{d-2} J_n.
\end{align}
They depend only on the initial triangulation $\tria_0$, cf.~\eqref{def:C0} and Remark~\ref{rmk:bounds-CT0}. 
Note that for each dimension~$n \in \{0, \ldots, d-2\}$ of critical macro-vertices, the respective constant~$J_n$ enters only once in \eqref{eq:genconst}. 
Recall the notion of distance~$\delta_\tria$ between simplices in Definition~\ref{def:distances}.

\begin{proposition}[Level and generation estimates]
	\label{pro:genest}
  Suppose that~$\tria_0$ is colored and let~$\tria \in \BisecT$ be an arbitrary triangulation. 
  \begin{enumerate}
 \item 
 \label{itm:PropLvlGenEst1}
 Any edge chain~$e_0,\dots,e_N\in \edges(\tria)$ as in \eqref{eq:defRelatedNodalChain2} satisfies 
 \begin{align*}
 \level(e_N) - \level(e_0) \leq N + \genconst.
 \end{align*}
 \item \label{itm:PropLvlGenEst2a}
 All simplices~$T,T'\in \tria$ satisfy that
 \begin{align*}
 	\level(T') - \level(T) &\leq \delta_{\tria}(T',T) + \genconst,\\
   \gen(T') - \gen(T) &\leq  (\delta_{\tria}(T',T) + \genconst + 1)d - 1. 
  \end{align*}
  \end{enumerate}
\end{proposition} 
The remainder of the subsection is devoted to the proof of this proposition. 
A key step in the proof is the following consequence of Proposition~\ref{prop:leavongOldNode}.
\begin{corollary}[Sequence of neighborhoods]
\label{cor:SeqOfNgh}
  Let~$\tria\in \BisecT$ be a triangulation with colored initial triangulation~$\tria_0$. 
  Let~$e_0,\ldots, e_N \in \edges(\tria)$ with~$N \in \mathbb{N}$ be a chain of edges in~$\tria$ with corresponding vertex chain~$v_0,\ldots, v_N \in \vertices(\tria)$ as in~\eqref{eq:defRelatedNodalChain2}. 
  Let $m\in \mathbb{N}$ be such that~$\level(e_0) < m$. 
  Then there exists an index~$k \in \mathbb{N}_0$ with~$k \leq N$ such that the remaining chain stays in a neighborhood in the sense that
  \begin{align*}
    v_j \in \Omega_{m+k}(v_k)\qquad\text{ for all } j = k,\ldots, N.
  \end{align*}
  Additionally, we have the level estimate
  \begin{align*}
      \level(e_{k}) < m + k.
  \end{align*}
\end{corollary}
\begin{proof}
   Let~$e_0,\ldots, e_N \in \edges(\tria)$ be an edge chain with vertices~$v_0,\ldots, v_N \in \vertices(\tria)$ as in~\eqref{eq:defRelatedNodalChain2}.   
   Let~$m \in \mathbb{N}_0$ be such that~$m >\level(e_0)$. 
   Since~$\level(v_0)\leq \level(e_0) < m$, the neighborhood~$\Omega_m(v_0)$ as in~\eqref{eq:defOmegahat} is well defined.
  If the chain does not leave the neighborhood~$\Omega_m(v_0)$, then by assumption the corollary is valid with~$k=0$. 
  
  Otherwise, if the chain leaves~$\Omega_m(v_0)$ for the first time after~$k_1\in \mathbb{N}$ steps, i.e., $v_{k_1} \notin \Omega_m(v_0)$, then Proposition~\ref{prop:leavongOldNode} applies and yields
  \begin{align*}
    \level(e_{k_1}) < m + k_1.
  \end{align*}
  We proceed inductively. For~$j\geq 1$ we define~$k_{j+1}\in \mathbb{N}$ as the index where the chain leaves the neighborhood~$\Omega_{m + k_{j}}(v_{k_{j}})$ for the first time, i.e.,~$v_{k_{j+1}} \notin \Omega_{m + k_{j}}(v_{k_{j}})$. 
  By the induction hypothesis we have that
  \begin{align*}
  	\level(v_{k_j}) \leq  \level(e_{k_j})< m+k_j.  
  \end{align*} 
 Applying Proposition~\ref{prop:leavongOldNode} inductively to the chain~$e_{k_{j}+1}, \ldots, e_{k_{j+1}}$ with length $k_{j+1}-k_j$ we conclude that
  \begin{align*}
  \level(e_{k_{j+1}}) < (m + k_{j}) + (k_{j+1} - k_j) = m + k_{j+1}. 
  \end{align*}
  Since the chain consists of finitely many edges, there is a last neighborhood that the chain does not leave. 
  This concludes the proof of the corollary. 
\end{proof}

The key step in the proof of Proposition~\ref{pro:genest} is to apply Corollary~\ref{cor:SeqOfNgh}. 
Then we combine the local generation estimates derived in Section~\ref{sec:locGenDiff} with Proposition~\ref{prop:FinerTria} on the auxiliary triangulations to conclude the that the chain does not leave a neighborhood $\Omega_{m+k}(v_k)$ and that $\tria \leq \tria_{m+k}^\infty(v_k)$. 
This allows us to apply Proposition~\ref{prop:stayInNgh} on chains that stay in a neighborhood, which leads to the final result. 

Unfortunately, this strategy fails if~$v_k$ is a critical $n$-macro vertex, i.e. if~$n\leq d-2$, since on high valence vertices large level jumps may occur. 
For this reason, we extend Corollary~\ref{cor:SeqOfNgh} to be able to treat those level jumps. 
Note that this can be done in a sophisticated manner so that each macro dimension has to be handled only once.   
The following extension involves the constants~$J_n \in \mathbb{N}_0$ as defined in~\eqref{eq:defJn-2}. 
\begin{proposition}[Increasing hyperface dimension]
	\label{prop:LeavingIncreasesDim}
For~$\tria\in \BisecT$ with colored initial triangulation~$\tria_0$, suppose that~$e_0,\dots, e_N \in \edges(\tria)$, with~$N \in \mathbb{N}$, is an edge chain with vertices~$v_0,\dots,v_N$ as in~\eqref{eq:defRelatedNodalChain2}. 
Let~$v_0$ be an $n$-macro vertex and let~$m \in \mathbb{N}_0$ be such that~$m>\level(e_0)$.  
Furthermore, assume that the chain does not leave the neighborhood~$\Omega_m(v_0)$ in the sense that
\begin{align}
	\label{eq:NotLeaving}
v_0,\dots, v_N \in \Omega_m(v_0).
\end{align} 
Then there exists an index~$k \in \mathbb{N}_0$ with~$k\leq N$ such that the remaining chain $e_{k+1},e_{k+2},\ldots \in \edges(\tria)$ does not leave the neighborhood~$\Omega_{m + k + J_n}(v_k)$ and
\begin{align}
	\label{eq:PropHyperFaceLeaving}
\begin{aligned}
\level(e_k) < m + k + J_n.
\end{aligned}
\end{align}
If~$v_k$ is an $n'$-macro vertex with dimension~$n' \leq n$, then additionally one has that
\begin{align}
	\label{eq:PropHyperStay}
\tria\leq \tria^\infty_{m+k+J_n}(v_k).
\end{align}
\end{proposition} 
	\begin{proof}
	Suppose the assumptions of the proposition hold. 
	The statement in~\eqref{eq:PropHyperFaceLeaving} follows by applying Corollary~\ref{cor:SeqOfNgh} with~$m+ J_n$ instead of~$m$.  
	It remains to show~\eqref{eq:PropHyperStay}. We therefore assume that~$v_k$ is an $n'$-macro vertex with $n'\leq n$. We would like to apply Proposition~\ref{prop:FinerTria} with~$v = v_k$ and $m$ replaced by $m+k+J_n$. 
	Since~\eqref{eq:PropHyperFaceLeaving} ensures that 
	\begin{align}\label{est:leavedim-1}
		\level(v_k) \leq \level(e_k) < m + k + J_n.
	\end{align}
it remains to verify that for the maximal level jump as in Definition~\ref{def:maxLvlJump-2} we have 
	\begin{align}
		\label{est:leavedim-2}
		\jump(v_k) + \min_{e\in \edges(\tria)\colon v_k\in e} \level(e) \leq  J_n + m + k +1.
	\end{align}
Lemma~\ref{lem:passingKfaces} states that the only $n'$-subsimplex of $\tria_0$ with $n'\leq n$ that intersects~$\Omega_m(v_0)$ is the $n$-macro subsimplex~$F$ with~$v_0\in F$. 
It follows that~$n' = n$ and that~$v_k \in F$. 
Hence, there exists an initial simplex~$\mathfrak{T}_0 \in \tria_0$ with~$e_0 \subset \mathfrak{T}_0$ and $v_k\in \mathfrak{T}_0$. 
Within this macro simplex level jumps are bounded by two, cf.~Proposition~\ref{pro:gsharp-2d-Kuhn} and Lemma~\ref{lem:genSharpToLvl}. 
Hence, Proposition~\ref{prop:FinerTria} implies that
\begin{align}
	\label{eq:ProofSMallerInMacor}
	\tria|_{\mathfrak{T}_0} \coloneqq \lbrace T \in \tria\colon T \subset \mathfrak{T}_0\rbrace \leq  \lbrace T \in \tria_m^\infty(v_0) \colon T \subset \mathfrak{T}_0\rbrace \eqqcolon \tria_m^\infty(v_0)|_{\mathfrak{T}_0}. 
\end{align}
Because by~\eqref{eq:NotLeaving} the chain remains in the neighborhood $\Omega_m(v_0)$, Lemma~\ref{lem:leavingLayers} ensures that 
$v_k \in \bigcup_{\ell=0}^k \Lambda_m^\ell(v_0)$.
Since~$v_k\in\vertices(\tria|_{\mathfrak{T}_0})$ it follows that
\begin{align*}
	v_k \in \bigcup_{\ell=0}^k \Lambda_m^\ell(v_0)|_{\mathfrak{T}_0} \qquad\text{with }\Lambda_m^\ell(v_0)|_{\mathfrak{T}_0} \coloneqq \lbrace T \in \Lambda_m^\ell(v_0)\colon T \subset \mathfrak{T}_0\rbrace.
\end{align*}
Consequently, there is an edge chain in~$\tria_m^\infty(v_0)|_{\mathfrak{T}_0}$ consisting of at most~$k$ edges that connects~$v_0$ and~$v_k$.
Thanks to this Proposition~\ref{prop:stayInNgh} applies and shows that there is an edge~$e\in \edges(\tria|_{\mathfrak{T}_0})$ with~$v_k \in e$ and 
\begin{align*}
	\level(e) \leq m + k.
\end{align*}
Therefore, we have that
\begin{align*}
	\jump(v_k) + \min_{e'\in \mathcal{E}(\tria)\colon v \in e'} \level(e') &\leq J_n + \level(e) \leq m + k + J_n + 1,
\end{align*}
which proves~\eqref{est:leavedim-2}. 
With both~\eqref{est:leavedim-1} and~\eqref{est:leavedim-2} Proposition~\ref{prop:FinerTria} yields the claim. 
	\end{proof}

We are now in the position to prove Proposition~\ref{pro:genest}. 
\begin{proof}[Proof of Proposition~\ref{pro:genest}]
Let~$\tria\in \BisecT$ be a triangulation with colored~$\tria_0$.

\textit{Step 1 (Level estimates for edge chains)}.
Let~$e_0,\dots,e_N \in \edges(\tria)$ be an edge chain with related vertices~$v_0,\dots,v_N$ as in \eqref{eq:defRelatedNodalChain2}. 
We set~$m\coloneqq \level(e_0) + 1$ and we apply Corollary~\ref{cor:SeqOfNgh} to obtain existence of an index~$k_1 \in \mathbb{N}_0$ with $k_1 \leq N$ such that
\begin{align}
	\label{eq:proofFinal1}
\level(e_{k_1}) < m + k_1 \quad\text{and}\quad v_{r} \in \Omega_{m+k_1}(v_{k_1})\quad\text{ for all } r = k_1,\dots, N. 
\end{align}
Let~$v_{k_1}$ be an $n$-macro vertex with $n \in \{0, \ldots, d\}$. 
The properties in~\eqref{eq:proofFinal1} allow us to apply  Proposition~\ref{prop:LeavingIncreasesDim}, with $m$  replaced by $m+k_1$ and $v_0$ replaced by $v_{k_1}$ in the statement. 
It shows that there exists an index~$k_2 \geq k_1$ with 
\begin{align}
	\label{eq:proofFinal2}
\level(e_{k_2}) < m + k_2 + J_n \quad\text{and}\quad v_{r} \in \Omega_{m+k_2 + J_n}(v_{k_2})\text{ for all }r=k_2,\ldots,N. 
\end{align}
Let $v_{k_2}$ be an $n'$-macro vertex with~$n' \in \{0, \ldots, d\}$.

 \textit{Case 1.}
If~$n'\leq n$, Proposition~\ref{prop:LeavingIncreasesDim} yields that $\tria \leq \tria^\infty_{m + k_2 + J_n}(v_{k_2})$.
Using this and \eqref{eq:proofFinal2} by Proposition~\ref{prop:stayInNgh}  it follows that 
\begin{align*}
\level(e_{r}) \leq m + J_n + r \qquad \text{for all }r = k_2,\dots,N.
\end{align*}

\textit{Case 2.} 
If~$v_{k_2}$ is an $n'$-macro vertex with~$n'> n$, we apply  Proposition~\ref{prop:LeavingIncreasesDim} again with $m$ replaced by $m+k_2+J_n$ and $e_0$ replaced by $e_{k_2}$. Hence, there exists an index~$k_3 \geq k_2$ such that for all~$r=k_3,\dots,N$ one has that
\begin{align*}
\level(e_{k_3}) < m + k_3 + J_n + J_{n'} \qquad\text{and}\qquad v_{r} \in \Omega_{m+k_3 + J_n + J_{n'}}(v_{k_3}). 
\end{align*}
We repeat the argument inductively starting again in~\eqref{eq:proofFinal2}. 
Since the macro dimension can increase at most~$d+1$ times, this procedure terminates.
From this we obtain the upper bound in~\ref{itm:PropLvlGenEst1} with~$\genconst$ as defined in \eqref{eq:genconst}, that is,
\begin{align}
	\label{eq:LvlSharpEst}
	\begin{aligned}
\level(e_N) &\leq m + N + \sum_{n=0}^d J_n = \level(e_0) + N + 1 + \sum_{n=0}^{d-2} J_n \\
&= \level(e_0) + N + \Gamma.
\end{aligned}
\end{align}  

\textit{Step 2 (From edges to simplices).}
Let~$T_0,T_N\in \tria$ denote two simplices with distance~$\delta_\tria(T_0,T_N) = N \in \mathbb{N}$. They are connected by an edge chain~$e_1,\dots,e_{N-1}\in \edges(\tria)$ with vertices
\begin{align*}
 e_1 \cap T_0 = v_0 \in \vertices(T_0)\qquad\text{and}\qquad e_{N-1} \cap T_N = v_{N-1}\in \vertices(T_N).
\end{align*}
Let $e_N\in \edges(T_N)$ denote an edge that contains $v_{N-1}$ and the youngest vertex of~$T_N$. 
The edge $e_N$ is connected to $e_{N-1}$ and satisfies
\begin{align}
	\label{eq:Proofateadsaa1}
	\begin{aligned}
\gen(T_N) &= \max_{v\in \vertices(T)} \gen(v) = \gen(e_N) \leq \level(e_N) d,\\
 \level(T_N)& = \max_{v\in \vertices(T)} \level(v)  = \level(e_N).
	\end{aligned}
\end{align}
For any edge~$e_0\in \edges(T_0)$ with $v_0 \in e_0$ we have that
\begin{align}\label{eq:Proofateadsaa2}
\level(e_0) \leq \max_{e\in \edges(T_0)} \level(e) = \level(T_0).
\end{align}
Then $e_0,\dots,e_N\in \edges(\tria)$ is an edge chain and hence \eqref{eq:LvlSharpEst} applies. Combining \eqref{eq:LvlSharpEst} with \eqref{eq:Proofateadsaa1} and \eqref{eq:Proofateadsaa2}
yields the first inequality in~\ref{itm:PropLvlGenEst2a} in the sense that
\begin{align*}
\level(T_N) - \level(T_0) \leq \level(e_N) - \level(e_0) \leq N + \Gamma.
\end{align*}
Moreover, we obtain by \eqref{eq:LvlSharpEst}, \eqref{eq:Proofateadsaa1}, and \eqref{eq:generation-level-type} that
\begin{align*}
\gen(T_N) - \gen(T_0) & \leq \level(e_N)d - (\level(T_0) - 1) d - \type(T_0) \\
 & \leq (\level(e_N) - \level(e_0))d + d - 1\\
 & \leq (N + \Gamma + 1)d -1.  
\end{align*}
This concludes the proof.
\end{proof}

\subsection{Matching neighbor condition}
\label{sec:stevenson}
The level estimates in Proposition~\ref{pro:genest} are proved under the assumption that~$\tria_0$ is colored. 
In this section we show that the level estimates easily extend to the case of initial triangulation satisfying the more general \emph{matching neighbor condition} by Stevenson~\cite[Sec.~4]{Stevenson08}. 
The matching neighbor condition guarantees that successive uniform refinements of $\tria_0$ are conforming, see Theorem~4.3 and Remark~4.4 in \cite{Stevenson08}. This is the only property we use, and for this reason we refrain from presenting the definition.
As above~$\tria_0^+$ denotes the triangulation arising from~$d$~uniform refinements of~$\tria_0$. 

\begin{lemma}[Colored $\tria_0^+$]
	\label{lem:tria0plus_colored}
If a triangulation~$\tria_0$ satisfies the matching neighbor condition, then~$\tria_0^+$ is colored in the sense that it satisfies Assumption~\ref{ass:initial-triangulation}. 
Furthermore, for any $T \in \tria_0^+$ and its descendants Algorithms~\ref{algo:genbased} and~\ref{algo:Traxler} coincide.
\end{lemma}
\begin{proof}
	To all vertices~$\vertices(\tria_0)$ we assign the
	generation $-d$. 
	The new vertices of the $j$-th uniform refinement are assigned with the generation $-d+j$ for~$j = 1, \ldots, d$. 
	Hence, each $d$-simplex in~$\tria_0^+$ has vertices with generation~$-d, \ldots, 0$. 
	This is because~$d$ uniform refinements of simplices with type~$d$ (tag~$\gamma = 0$ in Algorithm~\ref{algo:Traxler} by Traxler) bisect each edge in~$\tria_0$ exactly once and no newly created edge is bisected. 
	For simplices in~$\tria_0^+$ and their descendants the generation based routine in Algorithm~\ref{algo:genbased} and the routine by Traxler in Algorithm~\ref{algo:Traxler} coincide.
\end{proof}
Defining the generation of any $T_0 \in \tria_0$ as $\gen(T_0) = -d$ and increasing the generation by one in each bisection leads to generations satisfying the identity \eqref{eq:gen-simplices-vertices}, and levels as in \eqref{eq:def-type-level} for any $T \in \tria \in \Bisec(\tria_0)$ with coloring as in Lemma~\ref{lem:tria0plus_colored}.
Based on this we obtain a generalization of Proposition~\ref{pro:genest} involving the constant 
\begin{align}
	\label{eq:def-genconst2}
\genconst^+ \coloneqq \genconst + 1\qquad\text{with }\genconst\text{ defined in \eqref{eq:genconst}}.
\end{align}

\begin{proposition}[Estimates for matching neighbor]
	\label{pro:genest-stevenson}
	Suppose that~$\tria_0$ satisfies the matching neighbor condition and let~$\tria \in \BisecT$ be a triangulation.  
 For all~$T,T'\in \tria$ we have
  \begin{align*}
  	\level(T') - \level(T) & \leq \delta_{\tria}(T',T) + \genconst^+,\\
    \gen(T') - \gen(T) &\leq  (\delta_{\tria}(T',T) + \genconst^+ +  1)d - 1. 
  \end{align*}
\end{proposition}
\begin{proof}
The proof is based on the corresponding estimates for~$\Bisec(\tria_0^+)$ in Proposition~\ref{pro:genest} available thanks to the fact that~$\tria_0^+$ is colored, cf.~Lemma~\ref{lem:tria0plus_colored}.

Let~$\tria \in \BisecT$ be a triangulation with simplices~$T,T' \in \tria$. 
It suffices to consider the case~$T \neq T'$.
If~$\gen(T') < d$, the fact that~$\delta_\tria(T',T) \geq 1$ leads to the claim. 
For this reason, we focus on the case~$\gen(T')\geq d$. Let~$T_0,\dots,T_N \in \tria$ denote a shortest chain of intersecting simplices with~$T_0 = T$ and~$T_N = T'$, for~$N = \delta_{\tria}(T,T')$, cf.~Definition~\ref{def:distances}. 
Let~$M \in \set{ 0,\dots , N}$ denote the smallest index such that~$\gen(T_j) \geq d$ for all~$j=M,\dots,N$. 
Then~$T_M, \dots, T_N \in \tria_0^+$ is a simplicial chain in~$\Bisec(\tria_0^+)$.  
	Hence Proposition~\ref{pro:genest} applies and shows that
	\begin{align*}
		\level(T_N) - \level(T_M) \leq \delta_{\tria}(T_M,T_N) + \Gamma  = N-M + \Gamma. 
	\end{align*}
	If~$M=0$ this yields the statement of the lemma.
	If on the other hand~$M\geq 1$, then by the choice of~$M$ we have that~$\gen(T_{M-1}) < d$. 
	Thus, there exists a descendant~$T_{M-1}^+\in \tria_0^+$ of~$T_{M-1}$ with~$\gen(T_{M-1}^+) = d$ and~$T_{M-1}^+ \cap T_M \neq \emptyset$. 
	Applying Proposition~\ref{pro:genest} to the triangulation~$\tria \vee \tria_0^+ \in \Bisec(\tria_0^+)$ and the chain of intersecting simplices~$T_{M-1}^+,T_M,\dots, T_N$ yields that
	\begin{align*}
		\level(T_N) - \level(T_{M-1}^+) \leq N-(M-1) + \genconst.
	\end{align*}
	Since by~\eqref{eq:def-type-level} we have that~$\level(T_{M-1}^+) - \level(T_0) \leq 1$ and using the definition of~$\genconst^+$ in~\eqref{eq:def-genconst2} altogether we obtain the estimate
	\begin{align*}
		\level(T_N) - \level(T_0) \leq N + \genconst + 1 = N + \Gamma^+.
	\end{align*}
	This proves the level estimate. The generation estimate follows similarly. 
\end{proof}

\begin{remark}[More general initial triangulations]
Note that it is immediate to extend the results to initial triangulations for which a colored triangulation is generated by a fixed number~$r\in \mathbb{N}$ of full uniform refinements. 
	By the same arguments one obtains level estimates and grading~$\gamma\leq 2$, with constant $\genconst$ increased by~$r$. 
\end{remark}

\section{Summary: Proof of Theorem~\ref{thm:main-grading2}}
\label{sec:summary}
\begin{figure}
	\begin{tikzpicture}
		\tikzset{
			punkt/.style={rectangle,rounded corners,draw=black, very thick,text width=0.3\textwidth,minimum height=2em,text centered},
		}
		
		\node[punkt, inner sep=5pt] (thm) {Theorem~\ref{thm:main-grading2}: $\gamma \leq 2$ (sharp)};
		
		\node[punkt, inner sep=5pt, above = of thm] (lvlest) {level estimates for chains};
		\node(lvlestmin) at (lvlest) {};
		\node(lvlestmax) at (lvlest) {};		
		\node [inner sep=0,minimum size=0,above = of lvlest] (dummy0) {}; 		
		\node [inner sep=0,minimum size=0,above = of dummy0] (dummy1) {}; 
		\node [inner sep=0,minimum size=0,below = -5pt of dummy1] (dummy1min) {}; 
		\node [inner sep=0,minimum size=0,below = -2pt of dummy1] (dummy1max) {};		
		\node [inner sep=0,minimum size=0,above = of dummy1] (dummy2) {};  
		
		\node[punkt, inner sep=5pt, left = of dummy2] (localest) {local level estimates\\ (Sec.~\ref{sec:locGenDiff})};
		
		\node[punkt, inner sep=5pt, right = of dummy2] (ngbh) {auxiliary triangulation \& neighborhood (Sec.~\ref{sec:leavingTheNgh})};
		
		\node [inner sep=0,minimum size=0,left = of ngbh.south east] (ngbhsoutheast) {}; 		
		
		\path (localest.south)
		edge[very thick,decorate,decoration={brace,mirror,raise=.0cm,amplitude=8pt}]
		(localest.south -|ngbh.south);

		\draw[<-, shorten <=13pt,very thick,left,align=right]  (lvlestmin) to node {$\tria_0$ colored\\ (Prop.~\ref{pro:genest}) \\ \vphantom{(Prop.~\ref{pro:genest-stevenson})} } (dummy1min); 
		\draw[<-, shorten <=13pt,very thick,right,align=left]  (lvlestmax) to node {$\tria_0$ with matching\\neighbor condition \\			(Prop.~\ref{pro:genest-stevenson})} (dummy1min); 	
		
		\draw[<-, shorten <=1pt,very thick,right]  (thm) to node {Lem.~\ref{lem:lvlDiffToMain}} (lvlest); 
		
		\draw[<-, shorten <=1pt,very thick,right]  (thm) to [bend right=35]  node {Rem.~\ref{rem:sharpGradingEst}} (ngbhsoutheast); 
	\end{tikzpicture}
	\caption{Outline of the proof of main result in Theorem~\ref{thm:main-grading2}.}\label{fig:OutlineProof}
\end{figure}
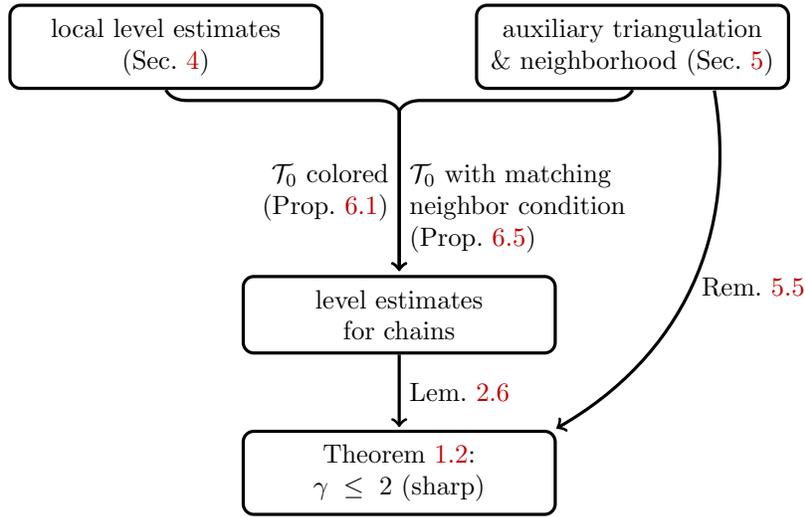
We conclude with a summary of the arguments leading to the main result in Theorem~\ref{thm:main-grading2}, see Figure~\ref{fig:OutlineProof}. 
It states that there are constants~$0<c_1\leq c_2 < \infty$ such that for each~$\tria\in \BisecT$ there exists a function~$h_\tria\colon \tria \to (0,\infty)$ with 
\begin{align}
	\label{eq:Summary-a}
c_1\, \textup{diam}(T) \leq h_\tria(T) &\leq c_2\, \textup{diam}(T) &&\text{for all }T\in \tria,\\
\label{eq:Summary-b}
 h_\tria(T) &\leq 2 h_\tria(T')&&\text{for all }T,T'\in \tria\text{ with }T\cap T'\neq \emptyset,
\end{align}
cf.~Definition~\ref{def:grading}. 
By Definition~\ref{def:regMeshSize-T} the mesh size function~$h_{\tria}$  has grading~$\gamma \leq 2$, which means that~\eqref{eq:Summary-b} is satisfied. 
It remains to verify the equivalence~$h_\tria \eqsim \textup{diam}(T)$ in~\eqref{eq:Summary-a}. 
Lemma~\ref{lem:lvlDiffToMain} states that this follows from the level estimate
\begin{align}
	\label{eq:Summary2}
		\level(T) - \level(T')\leq  \delta_{\tria}(T,T') + c_0  \qquad\text{for all }T,T'\in \tria\in \BisecT
\end{align}
with constant~$c_0 < \infty$. 
The equivalence constants~$c_1$, $c_2$ in~\eqref{eq:Summary-a} depend solely on~$c_0$, the shape regularity of~$\tria_0$, and the quasi-uniformity of~$\tria_0$. 
Combining local generation estimates derived in Section~\ref{sec:locGenDiff} with properties of a specific auxiliary triangulation investigated in Section~\ref{sec:leavingTheNgh} we arrive at Proposition~\ref{pro:genest} and~\ref{pro:genest-stevenson}, for colored initial triangulation $\tria_0$, and for~$\tria_0$ satisfying the matching neighbor condition by Stevenson, respectively. 
These propositions prove the level estimate~\eqref{eq:Summary2} with constant~$c_0$ depending only on the maximal edge valence in~$\tria_0$ and on~$d$. 
Combining these arguments shows that the grading is~$\gamma \leq 2$. 
Furthermore, Remark~\ref{rem:sharpGradingEst} states that this upper bound is sharp. 
This concludes the proof of Theorem~\ref{thm:main-grading2}.

\printbibliography 

\end{document}

%% file: PictureTex/BadMacro2d.tex
\begin{tikzpicture}[xscale = 1.2, yscale = .8]
\draw (0.3125,0.0) -- (0.0,0.0);
\draw (0.5412658773652742,0.31249999999999994) -- (0.0,0.0);
\draw (1.0825317547305484,0.6249999999999999) -- (2.165063509461097,1.2499999999999998);
\draw (0.5412658773652742,0.31249999999999994) -- (1.0825317547305484,0.6249999999999999);
\draw (0.5412658773652742,0.31249999999999994) -- (0.625,0.0);
\draw (0.0,0.0) -- (-2.1650635094610964,1.250000000000001);
\draw (0.0,0.0) -- (-2.5,1.3877787807814457e-15);
\draw (0.0,0.0) -- (-2.1650635094610973,-1.2499999999999987);
\draw (0.3125,0.0) -- (0.625,0.0);
\draw (0.3125,0.0) -- (0.5412658773652742,0.31249999999999994);
\draw (0.625,0.0) -- (1.25,0.0);
\draw (0.625,0.0) -- (1.0825317547305484,0.6249999999999999);
\draw (2.5,0.0) -- (2.165063509461097,1.2499999999999998);
\draw (2.5,0.0) -- (2.1650635094610955,-1.2500000000000022);
\draw (2.165063509461097,1.2499999999999998) -- (1.2500000000000004,2.1650635094610964);
\draw (1.2500000000000004,2.1650635094610964) -- (6.938893903907228e-16,2.5);
\draw (6.938893903907228e-16,2.5) -- (-1.2499999999999993,2.1650635094610973);
\draw (-1.2499999999999993,2.1650635094610973) -- (-2.1650635094610964,1.250000000000001);
\draw (-2.1650635094610964,1.250000000000001) -- (-2.5,1.3877787807814457e-15);
\draw (-2.5,1.3877787807814457e-15) -- (-2.1650635094610973,-1.2499999999999987);
\draw (-2.1650635094610973,-1.2499999999999987) -- (-1.2500000000000018,-2.1650635094610955);
\draw (-1.2500000000000018,-2.1650635094610955) -- (-2.0816681711721685e-15,-2.5);
\draw (-2.0816681711721685e-15,-2.5) -- (1.249999999999998,-2.1650635094610977);
\draw (1.249999999999998,-2.1650635094610977) -- (2.1650635094610955,-1.2500000000000022);
\draw (1.25,0.0) -- (2.5,0.0);
\draw (1.25,0.0) -- (2.165063509461097,1.2499999999999998);
\draw (1.25,0.0) -- (2.1650635094610955,-1.2500000000000022);
\draw (1.0825317547305484,0.6249999999999999) -- (1.25,0.0);
\draw (1.0825317547305484,0.6249999999999999) -- (0.6250000000000002,1.0825317547305482);
\draw (0.5412658773652742,0.31249999999999994) -- (0.3125000000000001,0.5412658773652741);
\draw (0.6250000000000002,1.0825317547305482) -- (2.165063509461097,1.2499999999999998);
\draw (0.6250000000000002,1.0825317547305482) -- (1.2500000000000004,2.1650635094610964);
\draw (0.6250000000000002,1.0825317547305482) -- (6.938893903907228e-16,2.5);
\draw (0.625,0.0) -- (1.0825317547305477,-0.6250000000000011);
\draw (0.3125,0.0) -- (0.5412658773652739,-0.31250000000000056);
\draw (1.0825317547305477,-0.6250000000000011) -- (2.1650635094610955,-1.2500000000000022);
\draw (1.0825317547305477,-0.6250000000000011) -- (1.25,0.0);
\draw (1.0825317547305477,-0.6250000000000011) -- (0.624999999999999,-1.0825317547305489);
\draw (0.5412658773652739,-0.31250000000000056) -- (0.0,0.0);
\draw (0.624999999999999,-1.0825317547305489) -- (1.249999999999998,-2.1650635094610977);
\draw (0.624999999999999,-1.0825317547305489) -- (2.1650635094610955,-1.2500000000000022);
\draw (0.624999999999999,-1.0825317547305489) -- (-2.0816681711721685e-15,-2.5);
\draw (0.3125000000000001,0.5412658773652741) -- (0.0,0.0);
\draw (0.3125000000000001,0.5412658773652741) -- (1.0825317547305484,0.6249999999999999);
\draw (0.3125000000000001,0.5412658773652741) -- (0.6250000000000002,1.0825317547305482);
\draw (0.3125000000000001,0.5412658773652741) -- (3.469446951953614e-16,1.25);
\draw (3.469446951953614e-16,1.25) -- (0.0,0.0);
\draw (3.469446951953614e-16,1.25) -- (6.938893903907228e-16,2.5);
\draw (3.469446951953614e-16,1.25) -- (0.6250000000000002,1.0825317547305482);
\draw (3.469446951953614e-16,1.25) -- (-0.6249999999999997,1.0825317547305486);
\draw (-0.6249999999999997,1.0825317547305486) -- (0.0,0.0);
\draw (-0.6249999999999997,1.0825317547305486) -- (6.938893903907228e-16,2.5);
\draw (-0.6249999999999997,1.0825317547305486) -- (-1.2499999999999993,2.1650635094610973);
\draw (-0.6249999999999997,1.0825317547305486) -- (-2.1650635094610964,1.250000000000001);
\draw (0.5412658773652739,-0.31250000000000056) -- (0.625,0.0);
\draw (0.5412658773652739,-0.31250000000000056) -- (1.0825317547305477,-0.6250000000000011);
\draw (0.5412658773652739,-0.31250000000000056) -- (0.3124999999999995,-0.5412658773652744);
\draw (0.3124999999999995,-0.5412658773652744) -- (0.0,0.0);
\draw (0.3124999999999995,-0.5412658773652744) -- (1.0825317547305477,-0.6250000000000011);
\draw (0.3124999999999995,-0.5412658773652744) -- (0.624999999999999,-1.0825317547305489);
\draw (0.3124999999999995,-0.5412658773652744) -- (-1.0408340855860843e-15,-1.25);
\draw (-1.0408340855860843e-15,-1.25) -- (0.0,0.0);
\draw (-1.0408340855860843e-15,-1.25) -- (-2.0816681711721685e-15,-2.5);
\draw (-1.0408340855860843e-15,-1.25) -- (0.624999999999999,-1.0825317547305489);
\draw (-1.0408340855860843e-15,-1.25) -- (-0.6250000000000009,-1.0825317547305477);
\draw (-0.6250000000000009,-1.0825317547305477) -- (0.0,0.0);
\draw (-0.6250000000000009,-1.0825317547305477) -- (-1.2500000000000018,-2.1650635094610955);
\draw (-0.6250000000000009,-1.0825317547305477) -- (-2.0816681711721685e-15,-2.5);
\draw (-0.6250000000000009,-1.0825317547305477) -- (-2.1650635094610973,-1.2499999999999987);
\draw (0.03,0.2)	node {\scriptsize -2};
\draw (2.65,0)	node {\scriptsize -1};
\draw (2.3,1.35)	node {\scriptsize 0};
\draw (1.35,2.35)	node {\scriptsize -1};
\draw (0.0,2.75)	node {\scriptsize 0};
\draw (-1.35,2.35)	node {\scriptsize -1};
\draw (-2.3,1.35)	node {\scriptsize 0};
\draw (-2.65,0)	node {\scriptsize -1};
\draw (-2.3,-1.35)	node {\scriptsize 0};
\draw (-1.35,-2.4)	node {\scriptsize -1};
\draw (0.,-2.75)	node {\scriptsize 0};
\draw (1.35,-2.4)	node {\scriptsize -1};
\draw (2.3,-1.35)	node {\scriptsize 0};
\draw (1.4,0)	node {\scriptsize 1};
\draw (1.12,0.83)	node {\scriptsize 2};
\draw (0.64,1.35)	node {\scriptsize 1};
\draw (0.81,0)	node {\scriptsize 3};
\draw (1.12,-0.83)	node {\scriptsize 2};
\draw (0.64,-1.35)	node {\scriptsize 1};
\draw (0.65,0.4125)	node {\scriptsize 4};
\draw (0.32,0.76)	node {\scriptsize 3};
\draw (0.1,1.38)	node {\scriptsize 2};
\draw (-0.64,1.35)	node {\scriptsize 1};
\draw (0.44,0)	node {\scriptsize 5};
\draw (0.65,-0.4125)	node {\scriptsize 4};
\draw (0.32,-0.76)	node {\scriptsize 3};
\draw (0.1,-1.38)	node {\scriptsize 2};
\draw (-0.64,-1.35)	node {\scriptsize 1};

\draw[fill = gray, fill opacity = 0.5] (0.0,0.0) -- (-2.5,1.3877787807814457e-15) -- (-2.1650635094610964,1.250000000000001) -- (0.0,0.0);

\draw[fill = gray, fill opacity = 0.5] (0.0,0.0) -- (0.3125,0.0)  --  (0.5412658773652742,0.31249999999999994) -- (0.0,0.0);
\end{tikzpicture}

%% file: PictureTex/Layer2d.tex
\begin{tikzpicture}[xscale = 5.5, yscale = 3]
   \draw (0.5,0.0) -- (0,0);
\draw (0.5,0.5) -- (0,0);
\draw (1.0,0.984375) -- (1,1);
\draw (0.984375,0.984375) -- (1,1);
\draw (0.75,0.75) -- (0.5,0.5);
\draw (0.5,0.0) -- (0.5,0.5);
\draw (1.0,0.984375) -- (1.0,0.96875);
\draw (1.0,0.984375) -- (0.984375,0.984375);
\draw (0.984375,0.0) -- (1,0);
\draw (0.875,0.875) -- (0.75,0.75);
\draw (0.9375,0.9375) -- (0.875,0.875);
\draw (0.75,0.25) -- (0.5,0.5);
\draw (0.75,0.25) -- (0.5,0.0);
\draw (0.96875,0.96875) -- (0.9375,0.9375);
\draw (0.75,0.0) -- (0.5,0.0);
\draw (0.75,0.0) -- (0.75,0.25);
\draw (0.875,0.0) -- (0.75,0.0);
\draw (0.875,0.0) -- (0.875,0.125);
\draw (0.9375,0.0) -- (0.875,0.0);
\draw (0.75,0.5) -- (0.5,0.5);
\draw (0.9375,0.0) -- (0.9375,0.0625);
\draw (0.75,0.5) -- (0.75,0.75);
\draw (0.75,0.5) -- (0.75,0.25);
\draw (0.875,0.125) -- (0.75,0.25);
\draw (0.875,0.125) -- (0.75,0.0);
\draw (0.984375,0.984375) -- (0.96875,0.96875);
\draw (0.9375,0.0625) -- (0.875,0.125);
\draw (0.875,0.625) -- (0.75,0.75);
\draw (0.9375,0.0625) -- (0.875,0.0);
\draw (0.875,0.625) -- (0.75,0.5);
\draw (0.984375,0.984375) -- (1.0,0.96875);
\draw (0.875,0.375) -- (0.75,0.25);
\draw (0.96875,0.03125) -- (0.9375,0.0625);
\draw (0.875,0.375) -- (0.75,0.5);
\draw (0.96875,0.0) -- (0.9375,0.0);
\draw (0.96875,0.0) -- (0.96875,0.03125);
\draw (0.984375,0.0) -- (0.96875,0.0);
\draw (0.984375,0.0) -- (0.984375,0.015625);
\draw (1.0,0.015625) -- (1,0);
\draw (1.0,0.015625) -- (1.0,0.03125);
\draw (1.0,0.015625) -- (0.984375,0.015625);
\draw (1.0,0.515625) -- (1.0,0.5);
\draw (0.875,0.5) -- (0.75,0.5);
\draw (0.875,0.5) -- (0.875,0.625);
\draw (0.875,0.5) -- (0.875,0.375);
\draw (0.875,0.75) -- (0.75,0.75);
\draw (1.0,0.515625) -- (1.0,0.53125);
\draw (0.875,0.75) -- (0.875,0.875);
\draw (0.875,0.75) -- (0.875,0.625);
\draw (0.875,0.25) -- (0.75,0.25);
\draw (1.0,0.515625) -- (0.984375,0.515625);
\draw (0.875,0.25) -- (0.875,0.125);
\draw (0.875,0.25) -- (0.875,0.375);
\draw (0.9375,0.5625) -- (0.875,0.625);
\draw (0.96875,0.03125) -- (0.9375,0.0);
\draw (0.9375,0.5625) -- (0.875,0.5);
\draw (0.984375,0.015625) -- (1,0);
\draw (0.9375,0.4375) -- (0.875,0.375);
\draw (0.984375,0.015625) -- (0.96875,0.03125);
\draw (0.9375,0.4375) -- (0.875,0.5);
\draw (0.96875,0.53125) -- (0.9375,0.5625);
\draw (0.9375,0.8125) -- (0.875,0.875);
\draw (0.984375,0.015625) -- (0.96875,0.0);
\draw (0.9375,0.8125) -- (0.875,0.75);
\draw (0.96875,0.53125) -- (0.9375,0.5);
\draw (0.9375,0.1875) -- (0.875,0.125);
\draw (0.984375,0.015625) -- (1.0,0.03125);
\draw (0.9375,0.1875) -- (0.875,0.25);
\draw (0.96875,0.46875) -- (0.9375,0.4375);
\draw (0.9375,0.6875) -- (0.875,0.625);
\draw (0.984375,0.515625) -- (1.0,0.5);
\draw (0.9375,0.6875) -- (0.875,0.75);
\draw (0.96875,0.46875) -- (0.9375,0.5);
\draw (0.9375,0.3125) -- (0.875,0.375);
\draw (0.984375,0.515625) -- (0.96875,0.53125);
\draw (0.9375,0.3125) -- (0.875,0.25);
\draw (0.9375,0.5) -- (0.875,0.5);
\draw (0.9375,0.5) -- (0.9375,0.5625);
\draw (0.9375,0.5) -- (0.9375,0.4375);
\draw (1.0,0.484375) -- (1.0,0.5);
\draw (1.0,0.484375) -- (1.0,0.46875);
\draw (0.96875,0.5) -- (0.9375,0.5);
\draw (0.96875,0.5) -- (0.96875,0.53125);
\draw (0.96875,0.5) -- (0.96875,0.46875);
\draw (1.0,0.484375) -- (0.984375,0.484375);
\draw (0.984375,0.5) -- (1.0,0.5);
\draw (0.984375,0.5) -- (0.96875,0.5);
\draw (0.984375,0.5) -- (0.984375,0.515625);
\draw (0.984375,0.5) -- (0.984375,0.484375);
\draw (1.0,0.765625) -- (1.0,0.75);
\draw (1.0,0.765625) -- (1.0,0.78125);
\draw (1.0,0.765625) -- (0.984375,0.765625);
\draw (0.9375,0.75) -- (0.875,0.75);
\draw (0.9375,0.75) -- (0.9375,0.8125);
\draw (0.9375,0.75) -- (0.9375,0.6875);
\draw (1.0,0.234375) -- (1.0,0.25);
\draw (0.9375,0.25) -- (0.875,0.25);
\draw (0.9375,0.25) -- (0.9375,0.1875);
\draw (0.9375,0.25) -- (0.9375,0.3125);
\draw (0.9375,0.875) -- (0.875,0.875);
\draw (1.0,0.234375) -- (1.0,0.21875);
\draw (0.9375,0.875) -- (0.9375,0.9375);
\draw (0.9375,0.875) -- (0.9375,0.8125);
\draw (0.9375,0.125) -- (0.875,0.125);
\draw (1.0,0.234375) -- (0.984375,0.234375);
\draw (0.9375,0.125) -- (0.9375,0.0625);
\draw (0.9375,0.125) -- (0.9375,0.1875);
\draw (0.9375,0.625) -- (0.875,0.625);
\draw (1.0,0.734375) -- (1.0,0.75);
\draw (0.9375,0.625) -- (0.9375,0.5625);
\draw (0.9375,0.625) -- (0.9375,0.6875);
\draw (0.9375,0.375) -- (0.875,0.375);
\draw (1.0,0.734375) -- (1.0,0.71875);
\draw (0.9375,0.375) -- (0.9375,0.4375);
\draw (0.9375,0.375) -- (0.9375,0.3125);
\draw (0.96875,0.78125) -- (0.9375,0.8125);
\draw (0.984375,0.515625) -- (1.0,0.53125);
\draw (0.96875,0.78125) -- (0.9375,0.75);
\draw (0.984375,0.515625) -- (0.96875,0.5);
\draw (0.96875,0.21875) -- (0.9375,0.1875);
\draw (0.984375,0.484375) -- (1.0,0.5);
\draw (0.96875,0.21875) -- (0.9375,0.25);
\draw (0.984375,0.484375) -- (0.96875,0.46875);
\draw (0.96875,0.71875) -- (0.9375,0.6875);
\draw (0.984375,0.484375) -- (1.0,0.46875);
\draw (0.96875,0.71875) -- (0.9375,0.75);
\draw (0.984375,0.484375) -- (0.96875,0.5);
\draw (0.96875,0.28125) -- (0.9375,0.3125);
\draw (0.984375,0.765625) -- (1.0,0.75);
\draw (0.96875,0.28125) -- (0.9375,0.25);
\draw (0.984375,0.765625) -- (0.96875,0.78125);
\draw (0.96875,0.90625) -- (0.9375,0.9375);
\draw (0.984375,0.765625) -- (1.0,0.78125);
\draw (0.96875,0.90625) -- (0.9375,0.875);
\draw (0.984375,0.765625) -- (0.96875,0.75);
\draw (0.96875,0.09375) -- (0.9375,0.0625);
\draw (0.984375,0.234375) -- (1.0,0.25);
\draw (0.96875,0.09375) -- (0.9375,0.125);
\draw (0.984375,0.234375) -- (0.96875,0.21875);
\draw (0.96875,0.59375) -- (0.9375,0.5625);
\draw (0.984375,0.234375) -- (1.0,0.21875);
\draw (0.96875,0.59375) -- (0.9375,0.625);
\draw (0.984375,0.234375) -- (0.96875,0.25);
\draw (0.96875,0.40625) -- (0.9375,0.4375);
\draw (0.984375,0.734375) -- (1.0,0.75);
\draw (0.96875,0.40625) -- (0.9375,0.375);
\draw (0.984375,0.734375) -- (0.96875,0.71875);
\draw (0.96875,0.84375) -- (0.9375,0.8125);
\draw (0.984375,0.734375) -- (1.0,0.71875);
\draw (0.96875,0.84375) -- (0.9375,0.875);
\draw (0.984375,0.734375) -- (0.96875,0.75);
\draw (0.96875,0.15625) -- (0.9375,0.1875);
\draw (0.984375,0.265625) -- (1.0,0.25);
\draw (0.96875,0.15625) -- (0.9375,0.125);
\draw (0.984375,0.265625) -- (0.96875,0.28125);
\draw (0.96875,0.65625) -- (0.9375,0.6875);
\draw (0.984375,0.265625) -- (1.0,0.28125);
\draw (0.96875,0.65625) -- (0.9375,0.625);
\draw (0.984375,0.265625) -- (0.96875,0.25);
\draw (0.96875,0.34375) -- (0.9375,0.3125);
\draw (0.984375,0.890625) -- (1.0,0.875);
\draw (0.96875,0.34375) -- (0.9375,0.375);
\draw (0.96875,0.75) -- (0.9375,0.75);
\draw (0.96875,0.75) -- (0.96875,0.78125);
\draw (0.96875,0.75) -- (0.96875,0.71875);
\draw (1.0,0.734375) -- (0.984375,0.734375);
\draw (0.96875,0.25) -- (0.9375,0.25);
\draw (0.96875,0.25) -- (0.96875,0.21875);
\draw (0.96875,0.25) -- (0.96875,0.28125);
\draw (1.0,0.265625) -- (1.0,0.25);
\draw (1.0,0.265625) -- (1.0,0.28125);
\draw (1.0,0.265625) -- (0.984375,0.265625);
\draw (0.984375,0.75) -- (1.0,0.75);
\draw (0.984375,0.75) -- (0.96875,0.75);
\draw (0.984375,0.75) -- (0.984375,0.765625);
\draw (0.984375,0.75) -- (0.984375,0.734375);
\draw (0.984375,0.25) -- (1.0,0.25);
\draw (0.984375,0.25) -- (0.96875,0.25);
\draw (0.984375,0.25) -- (0.984375,0.234375);
\draw (0.984375,0.25) -- (0.984375,0.265625);
\draw (1.0,0.890625) -- (1.0,0.875);
\draw (1.0,0.890625) -- (1.0,0.90625);
\draw (1.0,0.890625) -- (0.984375,0.890625);
\draw (1.0,0.109375) -- (1.0,0.125);
\draw (1.0,0.109375) -- (1.0,0.09375);
\draw (1.0,0.109375) -- (0.984375,0.109375);
\draw (1.0,0.609375) -- (1.0,0.625);
\draw (1.0,0.609375) -- (1.0,0.59375);
\draw (1.0,0.609375) -- (0.984375,0.609375);
\draw (1.0,0.390625) -- (1.0,0.375);
\draw (1.0,0.390625) -- (1.0,0.40625);
\draw (1.0,0.390625) -- (0.984375,0.390625);
\draw (1.0,0.859375) -- (1.0,0.875);
\draw (1.0,0.859375) -- (1.0,0.84375);
\draw (0.96875,0.875) -- (0.9375,0.875);
\draw (0.96875,0.875) -- (0.96875,0.90625);
\draw (0.96875,0.875) -- (0.96875,0.84375);
\draw (1.0,0.859375) -- (0.984375,0.859375);
\draw (0.96875,0.125) -- (0.9375,0.125);
\draw (0.96875,0.125) -- (0.96875,0.09375);
\draw (0.96875,0.125) -- (0.96875,0.15625);
\draw (1.0,0.140625) -- (1.0,0.125);
\draw (0.96875,0.625) -- (0.9375,0.625);
\draw (0.96875,0.625) -- (0.96875,0.59375);
\draw (0.96875,0.625) -- (0.96875,0.65625);
\draw (1.0,0.140625) -- (1.0,0.15625);
\draw (0.96875,0.375) -- (0.9375,0.375);
\draw (0.96875,0.375) -- (0.96875,0.40625);
\draw (0.96875,0.375) -- (0.96875,0.34375);
\draw (0.96875,0.9375) -- (0.9375,0.9375);
\draw (1.0,0.140625) -- (0.984375,0.140625);
\draw (0.96875,0.9375) -- (0.96875,0.96875);
\draw (0.96875,0.9375) -- (0.96875,0.90625);
\draw (0.96875,0.0625) -- (0.9375,0.0625);
\draw (1.0,0.640625) -- (1.0,0.625);
\draw (0.96875,0.0625) -- (0.96875,0.03125);
\draw (0.96875,0.0625) -- (0.96875,0.09375);
\draw (0.96875,0.5625) -- (0.9375,0.5625);
\draw (1.0,0.640625) -- (1.0,0.65625);
\draw (0.96875,0.5625) -- (0.96875,0.53125);
\draw (0.96875,0.5625) -- (0.96875,0.59375);
\draw (0.96875,0.4375) -- (0.9375,0.4375);
\draw (1.0,0.640625) -- (0.984375,0.640625);
\draw (0.96875,0.4375) -- (0.96875,0.46875);
\draw (0.96875,0.4375) -- (0.96875,0.40625);
\draw (0.96875,0.8125) -- (0.9375,0.8125);
\draw (1.0,0.359375) -- (1.0,0.375);
\draw (0.96875,0.8125) -- (0.96875,0.78125);
\draw (0.96875,0.8125) -- (0.96875,0.84375);
\draw (0.96875,0.1875) -- (0.9375,0.1875);
\draw (1.0,0.359375) -- (1.0,0.34375);
\draw (0.96875,0.1875) -- (0.96875,0.21875);
\draw (0.96875,0.1875) -- (0.96875,0.15625);
\draw (0.96875,0.6875) -- (0.9375,0.6875);
\draw (1.0,0.359375) -- (0.984375,0.359375);
\draw (0.96875,0.6875) -- (0.96875,0.71875);
\draw (0.96875,0.6875) -- (0.96875,0.65625);
\draw (0.96875,0.3125) -- (0.9375,0.3125);
\draw (0.984375,0.875) -- (1.0,0.875);
\draw (0.96875,0.3125) -- (0.96875,0.28125);
\draw (0.96875,0.3125) -- (0.96875,0.34375);
\draw (0.984375,0.890625) -- (0.96875,0.90625);
\draw (0.984375,0.890625) -- (1.0,0.90625);
\draw (0.984375,0.890625) -- (0.96875,0.875);
\draw (0.984375,0.109375) -- (1.0,0.125);
\draw (0.984375,0.109375) -- (0.96875,0.09375);
\draw (0.984375,0.109375) -- (1.0,0.09375);
\draw (0.984375,0.109375) -- (0.96875,0.125);
\draw (0.984375,0.609375) -- (1.0,0.625);
\draw (0.984375,0.609375) -- (0.96875,0.59375);
\draw (0.984375,0.609375) -- (1.0,0.59375);
\draw (0.984375,0.609375) -- (0.96875,0.625);
\draw (0.984375,0.390625) -- (1.0,0.375);
\draw (0.984375,0.390625) -- (0.96875,0.40625);
\draw (0.984375,0.390625) -- (1.0,0.40625);
\draw (0.984375,0.390625) -- (0.96875,0.375);
\draw (0.984375,0.859375) -- (1.0,0.875);
\draw (0.984375,0.859375) -- (0.96875,0.84375);
\draw (0.984375,0.859375) -- (1.0,0.84375);
\draw (0.984375,0.859375) -- (0.96875,0.875);
\draw (0.984375,0.140625) -- (1.0,0.125);
\draw (0.984375,0.140625) -- (0.96875,0.15625);
\draw (0.984375,0.140625) -- (1.0,0.15625);
\draw (0.984375,0.140625) -- (0.96875,0.125);
\draw (0.984375,0.640625) -- (1.0,0.625);
\draw (0.984375,0.640625) -- (0.96875,0.65625);
\draw (0.984375,0.640625) -- (1.0,0.65625);
\draw (0.984375,0.640625) -- (0.96875,0.625);
\draw (0.984375,0.359375) -- (1.0,0.375);
\draw (0.984375,0.359375) -- (0.96875,0.34375);
\draw (0.984375,0.359375) -- (1.0,0.34375);
\draw (0.984375,0.359375) -- (0.96875,0.375);
\draw (0.984375,0.953125) -- (1.0,0.9375);
\draw (0.984375,0.953125) -- (0.96875,0.96875);
\draw (0.984375,0.953125) -- (1.0,0.96875);
\draw (0.984375,0.953125) -- (0.96875,0.9375);
\draw (0.984375,0.046875) -- (1.0,0.0625);
\draw (0.984375,0.046875) -- (0.96875,0.03125);
\draw (0.984375,0.046875) -- (1.0,0.03125);
\draw (0.984375,0.046875) -- (0.96875,0.0625);
\draw (0.984375,0.546875) -- (1.0,0.5625);
\draw (0.984375,0.546875) -- (0.96875,0.53125);
\draw (0.984375,0.546875) -- (1.0,0.53125);
\draw (0.984375,0.546875) -- (0.96875,0.5625);
\draw (0.984375,0.453125) -- (1.0,0.4375);
\draw (0.984375,0.453125) -- (0.96875,0.46875);
\draw (0.984375,0.453125) -- (1.0,0.46875);
\draw (0.984375,0.453125) -- (0.96875,0.4375);
\draw (0.984375,0.796875) -- (1.0,0.8125);
\draw (0.984375,0.796875) -- (0.96875,0.78125);
\draw (0.984375,0.796875) -- (1.0,0.78125);
\draw (0.984375,0.796875) -- (0.96875,0.8125);
\draw (0.984375,0.203125) -- (1.0,0.1875);
\draw (0.984375,0.203125) -- (0.96875,0.21875);
\draw (0.984375,0.203125) -- (1.0,0.21875);
\draw (0.984375,0.203125) -- (0.96875,0.1875);
\draw (0.984375,0.703125) -- (1.0,0.6875);
\draw (0.984375,0.703125) -- (0.96875,0.71875);
\draw (0.984375,0.703125) -- (1.0,0.71875);
\draw (0.984375,0.703125) -- (0.96875,0.6875);
\draw (0.984375,0.296875) -- (1.0,0.3125);
\draw (0.984375,0.296875) -- (0.96875,0.28125);
\draw (0.984375,0.296875) -- (1.0,0.28125);
\draw (0.984375,0.296875) -- (0.96875,0.3125);
\draw (0.984375,0.921875) -- (1.0,0.9375);
\draw (0.984375,0.921875) -- (0.96875,0.90625);
\draw (0.984375,0.921875) -- (1.0,0.90625);
\draw (0.984375,0.921875) -- (0.96875,0.9375);
\draw (0.984375,0.078125) -- (1.0,0.0625);
\draw (0.984375,0.078125) -- (0.96875,0.09375);
\draw (0.984375,0.078125) -- (1.0,0.09375);
\draw (0.984375,0.078125) -- (0.96875,0.0625);
\draw (0.984375,0.578125) -- (1.0,0.5625);
\draw (0.984375,0.578125) -- (0.96875,0.59375);
\draw (0.984375,0.578125) -- (1.0,0.59375);
\draw (0.984375,0.578125) -- (0.96875,0.5625);
\draw (0.984375,0.421875) -- (1.0,0.4375);
\draw (0.984375,0.421875) -- (0.96875,0.40625);
\draw (0.984375,0.421875) -- (1.0,0.40625);
\draw (0.984375,0.421875) -- (0.96875,0.4375);
\draw (0.984375,0.828125) -- (1.0,0.8125);
\draw (0.984375,0.828125) -- (0.96875,0.84375);
\draw (0.984375,0.828125) -- (1.0,0.84375);
\draw (0.984375,0.828125) -- (0.96875,0.8125);
\draw (0.984375,0.171875) -- (1.0,0.1875);
\draw (0.984375,0.171875) -- (0.96875,0.15625);
\draw (0.984375,0.171875) -- (1.0,0.15625);
\draw (0.984375,0.171875) -- (0.96875,0.1875);
\draw (0.984375,0.671875) -- (1.0,0.6875);
\draw (0.984375,0.671875) -- (0.96875,0.65625);
\draw (0.984375,0.671875) -- (1.0,0.65625);
\draw (0.984375,0.671875) -- (0.96875,0.6875);
\draw (0.984375,0.328125) -- (1.0,0.3125);
\draw (0.984375,0.328125) -- (0.96875,0.34375);
\draw (0.984375,0.328125) -- (1.0,0.34375);
\draw (0.984375,0.328125) -- (0.96875,0.3125);
\draw (0.984375,0.875) -- (0.96875,0.875);
\draw (0.984375,0.875) -- (0.984375,0.890625);
\draw (0.984375,0.875) -- (0.984375,0.859375);
\draw (0.984375,0.125) -- (1.0,0.125);
\draw (0.984375,0.125) -- (0.96875,0.125);
\draw (0.984375,0.125) -- (0.984375,0.109375);
\draw (0.984375,0.125) -- (0.984375,0.140625);
\draw (0.984375,0.625) -- (1.0,0.625);
\draw (0.984375,0.625) -- (0.96875,0.625);
\draw (0.984375,0.625) -- (0.984375,0.609375);
\draw (0.984375,0.625) -- (0.984375,0.640625);
\draw (0.984375,0.375) -- (1.0,0.375);
\draw (0.984375,0.375) -- (0.96875,0.375);
\draw (0.984375,0.375) -- (0.984375,0.390625);
\draw (0.984375,0.375) -- (0.984375,0.359375);
\draw (1.0,0.953125) -- (1.0,0.9375);
\draw (1.0,0.953125) -- (1.0,0.96875);
\draw (1.0,0.953125) -- (0.984375,0.953125);
\draw (1.0,0.046875) -- (1.0,0.0625);
\draw (1.0,0.046875) -- (1.0,0.03125);
\draw (1.0,0.046875) -- (0.984375,0.046875);
\draw (1.0,0.546875) -- (1.0,0.5625);
\draw (1.0,0.546875) -- (1.0,0.53125);
\draw (1.0,0.546875) -- (0.984375,0.546875);
\draw (1.0,0.453125) -- (1.0,0.4375);
\draw (1.0,0.453125) -- (1.0,0.46875);
\draw (1.0,0.453125) -- (0.984375,0.453125);
\draw (1.0,0.796875) -- (1.0,0.8125);
\draw (1.0,0.796875) -- (1.0,0.78125);
\draw (1.0,0.796875) -- (0.984375,0.796875);
\draw (1.0,0.203125) -- (1.0,0.1875);
\draw (1.0,0.203125) -- (1.0,0.21875);
\draw (1.0,0.203125) -- (0.984375,0.203125);
\draw (1.0,0.703125) -- (1.0,0.6875);
\draw (1.0,0.703125) -- (1.0,0.71875);
\draw (1.0,0.703125) -- (0.984375,0.703125);
\draw (1.0,0.296875) -- (1.0,0.3125);
\draw (1.0,0.296875) -- (1.0,0.28125);
\draw (1.0,0.296875) -- (0.984375,0.296875);
\draw (1.0,0.921875) -- (1.0,0.9375);
\draw (1.0,0.921875) -- (1.0,0.90625);
\draw (1.0,0.921875) -- (0.984375,0.921875);
\draw (1.0,0.078125) -- (1.0,0.0625);
\draw (1.0,0.078125) -- (1.0,0.09375);
\draw (1.0,0.078125) -- (0.984375,0.078125);
\draw (1.0,0.578125) -- (1.0,0.5625);
\draw (1.0,0.578125) -- (1.0,0.59375);
\draw (1.0,0.578125) -- (0.984375,0.578125);
\draw (1.0,0.421875) -- (1.0,0.4375);
\draw (1.0,0.421875) -- (1.0,0.40625);
\draw (1.0,0.421875) -- (0.984375,0.421875);
\draw (1.0,0.828125) -- (1.0,0.8125);
\draw (1.0,0.828125) -- (1.0,0.84375);
\draw (1.0,0.828125) -- (0.984375,0.828125);
\draw (1.0,0.171875) -- (1.0,0.1875);
\draw (1.0,0.171875) -- (1.0,0.15625);
\draw (1.0,0.171875) -- (0.984375,0.171875);
\draw (1.0,0.671875) -- (1.0,0.6875);
\draw (1.0,0.671875) -- (1.0,0.65625);
\draw (1.0,0.671875) -- (0.984375,0.671875);
\draw (1.0,0.328125) -- (1.0,0.3125);
\draw (1.0,0.328125) -- (1.0,0.34375);
\draw (1.0,0.328125) -- (0.984375,0.328125);
\draw (0.984375,0.9375) -- (1.0,0.9375);
\draw (0.984375,0.9375) -- (0.96875,0.9375);
\draw (0.984375,0.9375) -- (0.984375,0.953125);
\draw (0.984375,0.9375) -- (0.984375,0.921875);
\draw (0.984375,0.0625) -- (1.0,0.0625);
\draw (0.984375,0.0625) -- (0.96875,0.0625);
\draw (0.984375,0.0625) -- (0.984375,0.046875);
\draw (0.984375,0.0625) -- (0.984375,0.078125);
\draw (0.984375,0.5625) -- (1.0,0.5625);
\draw (0.984375,0.5625) -- (0.96875,0.5625);
\draw (0.984375,0.5625) -- (0.984375,0.546875);
\draw (0.984375,0.5625) -- (0.984375,0.578125);
\draw (0.984375,0.4375) -- (1.0,0.4375);
\draw (0.984375,0.4375) -- (0.96875,0.4375);
\draw (0.984375,0.4375) -- (0.984375,0.453125);
\draw (0.984375,0.4375) -- (0.984375,0.421875);
\draw (0.984375,0.8125) -- (1.0,0.8125);
\draw (0.984375,0.8125) -- (0.96875,0.8125);
\draw (0.984375,0.8125) -- (0.984375,0.796875);
\draw (0.984375,0.8125) -- (0.984375,0.828125);
\draw (0.984375,0.1875) -- (1.0,0.1875);
\draw (0.984375,0.1875) -- (0.96875,0.1875);
\draw (0.984375,0.1875) -- (0.984375,0.203125);
\draw (0.984375,0.1875) -- (0.984375,0.171875);
\draw (0.984375,0.6875) -- (1.0,0.6875);
\draw (0.984375,0.6875) -- (0.96875,0.6875);
\draw (0.984375,0.6875) -- (0.984375,0.703125);
\draw (0.984375,0.6875) -- (0.984375,0.671875);
\draw (0.984375,0.3125) -- (1.0,0.3125);
\draw (0.984375,0.3125) -- (0.96875,0.3125);
\draw (0.984375,0.3125) -- (0.984375,0.296875);
\draw (0.984375,0.3125) -- (0.984375,0.328125);
\draw (0.984375,0.96875) -- (0.96875,0.96875);
\draw (0.984375,0.96875) -- (1.0,0.96875);
\draw (0.984375,0.96875) -- (0.984375,0.984375);
\draw (0.984375,0.96875) -- (0.984375,0.953125);
\draw (0.984375,0.03125) -- (0.96875,0.03125);
\draw (0.984375,0.03125) -- (1.0,0.03125);
\draw (0.984375,0.03125) -- (0.984375,0.015625);
\draw (0.984375,0.03125) -- (0.984375,0.046875);
\draw (0.984375,0.53125) -- (0.96875,0.53125);
\draw (0.984375,0.53125) -- (1.0,0.53125);
\draw (0.984375,0.53125) -- (0.984375,0.515625);
\draw (0.984375,0.53125) -- (0.984375,0.546875);
\draw (0.984375,0.46875) -- (0.96875,0.46875);
\draw (0.984375,0.46875) -- (1.0,0.46875);
\draw (0.984375,0.46875) -- (0.984375,0.484375);
\draw (0.984375,0.46875) -- (0.984375,0.453125);
\draw (0.984375,0.78125) -- (0.96875,0.78125);
\draw (0.984375,0.78125) -- (1.0,0.78125);
\draw (0.984375,0.78125) -- (0.984375,0.765625);
\draw (0.984375,0.78125) -- (0.984375,0.796875);
\draw (0.984375,0.21875) -- (0.96875,0.21875);
\draw (0.984375,0.21875) -- (1.0,0.21875);
\draw (0.984375,0.21875) -- (0.984375,0.234375);
\draw (0.984375,0.21875) -- (0.984375,0.203125);
\draw (0.984375,0.71875) -- (0.96875,0.71875);
\draw (0.984375,0.71875) -- (1.0,0.71875);
\draw (0.984375,0.71875) -- (0.984375,0.734375);
\draw (0.984375,0.71875) -- (0.984375,0.703125);
\draw (0.984375,0.28125) -- (0.96875,0.28125);
\draw (0.984375,0.28125) -- (1.0,0.28125);
\draw (0.984375,0.28125) -- (0.984375,0.265625);
\draw (0.984375,0.28125) -- (0.984375,0.296875);
\draw (0.984375,0.90625) -- (0.96875,0.90625);
\draw (0.984375,0.90625) -- (1.0,0.90625);
\draw (0.984375,0.90625) -- (0.984375,0.890625);
\draw (0.984375,0.90625) -- (0.984375,0.921875);
\draw (0.984375,0.09375) -- (0.96875,0.09375);
\draw (0.984375,0.09375) -- (1.0,0.09375);
\draw (0.984375,0.09375) -- (0.984375,0.109375);
\draw (0.984375,0.09375) -- (0.984375,0.078125);
\draw (0.984375,0.59375) -- (0.96875,0.59375);
\draw (0.984375,0.59375) -- (1.0,0.59375);
\draw (0.984375,0.59375) -- (0.984375,0.609375);
\draw (0.984375,0.59375) -- (0.984375,0.578125);
\draw (0.984375,0.40625) -- (0.96875,0.40625);
\draw (0.984375,0.40625) -- (1.0,0.40625);
\draw (0.984375,0.40625) -- (0.984375,0.390625);
\draw (0.984375,0.40625) -- (0.984375,0.421875);
\draw (0.984375,0.84375) -- (0.96875,0.84375);
\draw (0.984375,0.84375) -- (1.0,0.84375);
\draw (0.984375,0.84375) -- (0.984375,0.859375);
\draw (0.984375,0.84375) -- (0.984375,0.828125);
\draw (0.984375,0.15625) -- (0.96875,0.15625);
\draw (0.984375,0.15625) -- (1.0,0.15625);
\draw (0.984375,0.15625) -- (0.984375,0.140625);
\draw (0.984375,0.15625) -- (0.984375,0.171875);
\draw (0.984375,0.65625) -- (0.96875,0.65625);
\draw (0.984375,0.65625) -- (1.0,0.65625);
\draw (0.984375,0.65625) -- (0.984375,0.640625);
\draw (0.984375,0.65625) -- (0.984375,0.671875);
\draw (0.984375,0.34375) -- (0.96875,0.34375);
\draw (0.984375,0.34375) -- (1.0,0.34375);
\draw (0.984375,0.34375) -- (0.984375,0.359375);
\draw (0.984375,0.34375) -- (0.984375,0.328125);  
\node[circle,fill=black,inner sep=0.6pt] (ell) at (0,0) {\,  };
\node[below] (ell2) at (0,0) {$v$};
\end{tikzpicture}

%% file: PictureTex/Layer3D.tex
\tdplotsetmaincoords{85}{180.8}
\begin{tikzpicture}[xscale = 5.5, yscale = 2.8,tdplot_main_coords]
\draw (0.0,0.875,0.125) -- (0,1,0);
\draw (0.0,0.875,0.0) -- (0,1,0);
\draw (0.5,0.0,0.0) -- (1,0,0);
\draw (0.5,0.0,0.0) -- (0.5,0.5,0.0);
\draw (0.0,0.875,0.0) -- (0.0,0.75,0.0);
\draw (0.125,0.875,0.0) -- (0,1,0);
\draw (0.25,0.75,0.0) -- (0.5,0.5,0.0);
\draw (0.125,0.875,0.0) -- (0.25,0.75,0.0);
\draw (0.5,0.5,0.0) -- (1,0,0);
\draw (0.5,0.0,0.5) -- (1,0,0);
\draw (0.5,0.0,0.5) -- (0.5,0.5,0.0);
\draw (0.0,0.875,0.0) -- (0.0,0.75,0.25);
\draw (0.0,0.875,0.125) -- (0.0,0.75,0.25);
\draw (0.0,0.875,0.0) -- (0.125,0.875,0.0);
\draw (0.0,0.875,0.125) -- (0.125,0.875,0.0);
\draw (0.5,0.0,0.0) -- (0.5,0.0,0.5);
\draw (0.125,0.0,0.875) -- (0,0,1);
\draw (0.0,0.875,0.125) -- (0.0,0.875,0.0);
\draw (0.25,0.0,0.75) -- (0.5,0.0,0.5);
\draw (0.0,0.0,0.875) -- (0,0,1);
\draw (0.0,0.0,0.875) -- (0.0,0.0,0.75);
\draw (0.125,0.0,0.875) -- (0.25,0.0,0.75);
\draw (0.125,0.0,0.875) -- (0.0,0.0,0.75);
\draw (0.0,0.0,0.875) -- (0.125,0.125,0.75);
\draw (0.0,0.0,0.875) -- (0.125,0.0,0.875);
\draw (0.125,0.875,0.0) -- (0.0,0.75,0.0);
\draw (0.125,0.875,0.0) -- (0.0,0.75,0.25);
\draw (0.25,0.25,0.5) -- (0.5,0.5,0.0);
\draw (0.25,0.25,0.5) -- (0.5,0.0,0.5);
\draw (0.125,0.125,0.75) -- (0,0,1);
\draw (0.125,0.125,0.75) -- (0.25,0.25,0.5);
\draw (0.125,0.125,0.75) -- (0.25,0.0,0.75);
\draw (0.125,0.125,0.75) -- (0.0,0.0,0.75);
\draw (0.25,0.25,0.0) -- (0.5,0.5,0.0);
\draw (0.25,0.25,0.0) -- (0.5,0.0,0.5);
\draw (0.25,0.25,0.0) -- (0.5,0.0,0.0);
\draw (0.125,0.125,0.75) -- (0.0,0.25,0.5);
\draw (0.125,0.125,0.75) -- (0.0,0.25,0.75);
\draw (0.25,0.0,0.75) -- (0.25,0.25,0.5);
\draw (0.125,0.0,0.875) -- (0.125,0.125,0.75);
\draw (0.0,0.125,0.75) -- (0,0,1);
\draw (0.0,0.125,0.75) -- (0.0,0.25,0.5);
\draw (0.0,0.0,0.875) -- (0.0,0.125,0.75);
\draw (0.0,0.125,0.75) -- (0.0,0.0,0.75);
\draw (0.0,0.125,0.75) -- (0.125,0.125,0.75);
\draw (0.25,0.0,0.25) -- (0.5,0.0,0.5);
\draw (0.25,0.0,0.25) -- (0.5,0.0,0.0);
\draw (0.25,0.0,0.25) -- (0.25,0.25,0.0);
\draw (0.0,0.125,0.75) -- (0.0,0.25,0.75);
\draw (0.125,0.0,0.125) -- (0,0,0);
\draw (0.125,0.0,0.125) -- (0.25,0.0,0.25);
\draw (0.125,0.0,0.125) -- (0.25,0.0,0.0);
\draw (0.125,0.0,0.0) -- (0,0,0);
\draw (0.25,0.5,0.25) -- (0.5,0.5,0.0);
\draw (0.125,0.0,0.125) -- (0.125,0.125,0.0);
\draw (0.125,0.0,0.125) -- (0.0,0.0,0.25);
\draw (0.25,0.5,0.25) -- (0.25,0.75,0.0);
\draw (0.25,0.5,0.25) -- (0.25,0.25,0.5);
\draw (0.25,0.25,0.25) -- (0.5,0.5,0.0);
\draw (0.25,0.25,0.25) -- (0.5,0.0,0.5);
\draw (0.0,0.125,0.0) -- (0,0,0);
\draw (0.25,0.25,0.25) -- (0.25,0.25,0.5);
\draw (0.0,0.125,0.0) -- (0.0,0.25,0.0);
\draw (0.25,0.25,0.25) -- (0.25,0.25,0.0);
\draw (0.125,0.0,0.0) -- (0.25,0.0,0.0);
\draw (0.25,0.0,0.0) -- (0.5,0.0,0.0);
\draw (0.25,0.0,0.0) -- (0.25,0.25,0.0);
\draw (0.25,0.0,0.0) -- (0.25,0.0,0.25);
\draw (0.125,0.0,0.0) -- (0.125,0.125,0.0);
\draw (0.125,0.0,0.0) -- (0.125,0.0,0.125);
\draw (0.0,0.125,0.0) -- (0.0,0.0,0.25);
\draw (0.0,0.125,0.875) -- (0,0,1);
\draw (0.0,0.125,0.875) -- (0.0,0.25,0.75);
\draw (0.0,0.125,0.875) -- (0.125,0.125,0.75);
\draw (0.0,0.125,0.0) -- (0.125,0.125,0.0);
\draw (0.0,0.125,0.875) -- (0.0,0.125,0.75);
\draw (0.0,0.0,0.125) -- (0,0,0);
\draw (0.25,0.5,0.0) -- (0.5,0.5,0.0);
\draw (0.0,0.0,0.125) -- (0.0,0.0,0.25);
\draw (0.25,0.5,0.0) -- (0.25,0.75,0.0);
\draw (0.25,0.5,0.0) -- (0.25,0.5,0.25);
\draw (0.25,0.5,0.0) -- (0.25,0.25,0.5);
\draw (0.25,0.5,0.0) -- (0.25,0.25,0.25);
\draw (0.25,0.5,0.0) -- (0.25,0.25,0.0);
\draw (0.0,0.0,0.125) -- (0.125,0.125,0.0);
\draw (0.0,0.0,0.125) -- (0.125,0.0,0.125);
\draw (0.0,0.625,0.0) -- (0.0,0.5,0.0);
\draw (0.0,0.0,0.125) -- (0.0,0.125,0.0);
\draw (0.0,0.625,0.0) -- (0.0,0.75,0.0);
\draw (0.0,0.5,0.125) -- (0.0,0.5,0.0);
\draw (0.0,0.5,0.125) -- (0.0,0.5,0.25);
\draw (0.25,0.0,0.5) -- (0.5,0.0,0.5);
\draw (0.0,0.5,0.125) -- (0.125,0.625,0.0);
\draw (0.25,0.0,0.5) -- (0.25,0.25,0.5);
\draw (0.25,0.0,0.5) -- (0.25,0.0,0.75);
\draw (0.25,0.0,0.5) -- (0.25,0.25,0.0);
\draw (0.25,0.0,0.5) -- (0.25,0.0,0.25);
\draw (0.25,0.0,0.5) -- (0.25,0.25,0.25);
\draw (0.0,0.5,0.125) -- (0.0,0.625,0.0);
\draw (0.0,0.5,0.125) -- (0.125,0.375,0.25);
\draw (0.0,0.625,0.0) -- (0.0,0.5,0.25);
\draw (0.0,0.5,0.125) -- (0.0,0.375,0.25);
\draw (0.0,0.625,0.0) -- (0.125,0.625,0.0);
\draw (0.0,0.5,0.125) -- (0.125,0.5,0.125);
\draw (0.0,0.375,0.125) -- (0.0,0.5,0.0);
\draw (0.125,0.125,0.0) -- (0,0,0);
\draw (0.125,0.125,0.0) -- (0.25,0.25,0.0);
\draw (0.125,0.125,0.0) -- (0.25,0.0,0.25);
\draw (0.125,0.125,0.0) -- (0.25,0.0,0.0);
\draw (0.125,0.125,0.0) -- (0.0,0.0,0.25);
\draw (0.125,0.125,0.0) -- (0.0,0.25,0.0);
\draw (0.125,0.625,0.0) -- (0.0,0.5,0.0);
\draw (0.125,0.625,0.0) -- (0.25,0.75,0.0);
\draw (0.125,0.625,0.0) -- (0.0,0.75,0.0);
\draw (0.125,0.625,0.0) -- (0.0,0.5,0.25);
\draw (0.125,0.625,0.0) -- (0.25,0.5,0.25);
\draw (0.125,0.625,0.0) -- (0.25,0.5,0.0);
\draw (0.125,0.375,0.25) -- (0.0,0.5,0.0);
\draw (0.125,0.375,0.25) -- (0.25,0.25,0.5);
\draw (0.125,0.375,0.25) -- (0.0,0.25,0.5);
\draw (0.125,0.375,0.25) -- (0.0,0.25,0.25);
\draw (0.125,0.375,0.25) -- (0.0,0.5,0.25);
\draw (0.125,0.375,0.25) -- (0.25,0.25,0.25);
\draw (0.125,0.375,0.25) -- (0.25,0.5,0.25);
\draw (0.125,0.375,0.25) -- (0.25,0.5,0.0);
\draw (0.125,0.375,0.0) -- (0.0,0.5,0.0);
\draw (0.125,0.375,0.0) -- (0.25,0.25,0.0);
\draw (0.125,0.375,0.0) -- (0.0,0.25,0.0);
\draw (0.125,0.375,0.0) -- (0.0,0.25,0.25);
\draw (0.125,0.375,0.0) -- (0.25,0.25,0.25);
\draw (0.125,0.375,0.0) -- (0.25,0.5,0.0);
\draw (0.125,0.625,0.25) -- (0.0,0.5,0.5);
\draw (0.125,0.625,0.25) -- (0.25,0.75,0.0);
\draw (0.125,0.625,0.25) -- (0.0,0.75,0.0);
\draw (0.125,0.625,0.25) -- (0.0,0.75,0.25);
\draw (0.125,0.625,0.25) -- (0.0,0.5,0.25);
\draw (0.125,0.625,0.25) -- (0.25,0.5,0.25);
\draw (0.125,0.125,0.5) -- (0.0,0.0,0.5);
\draw (0.125,0.125,0.5) -- (0.25,0.25,0.5);
\draw (0.125,0.125,0.5) -- (0.25,0.0,0.75);
\draw (0.125,0.125,0.5) -- (0.0,0.0,0.75);
\draw (0.125,0.125,0.5) -- (0.0,0.25,0.5);
\draw (0.125,0.125,0.5) -- (0.25,0.0,0.5);
\draw (0.125,0.125,0.5) -- (0.0,0.25,0.25);
\draw (0.125,0.125,0.5) -- (0.25,0.25,0.25);
\draw (0.125,0.375,0.5) -- (0.0,0.5,0.5);
\draw (0.125,0.375,0.5) -- (0.25,0.25,0.5);
\draw (0.125,0.375,0.5) -- (0.0,0.25,0.5);
\draw (0.125,0.375,0.5) -- (0.0,0.25,0.75);
\draw (0.125,0.375,0.5) -- (0.0,0.5,0.25);
\draw (0.125,0.375,0.5) -- (0.25,0.5,0.25);
\draw (0.125,0.125,0.25) -- (0.0,0.0,0.5);
\draw (0.125,0.125,0.25) -- (0.25,0.25,0.0);
\draw (0.125,0.125,0.25) -- (0.25,0.0,0.25);
\draw (0.125,0.125,0.25) -- (0.0,0.0,0.25);
\draw (0.125,0.125,0.25) -- (0.0,0.25,0.0);
\draw (0.125,0.125,0.25) -- (0.25,0.0,0.5);
\draw (0.125,0.125,0.25) -- (0.0,0.25,0.25);
\draw (0.125,0.125,0.25) -- (0.25,0.25,0.25);
\draw (0.0,0.375,0.25) -- (0.0,0.5,0.0);
\draw (0.0,0.375,0.25) -- (0.0,0.25,0.5);
\draw (0.0,0.375,0.25) -- (0.0,0.25,0.25);
\draw (0.0,0.375,0.25) -- (0.125,0.375,0.25);
\draw (0.0,0.375,0.25) -- (0.0,0.5,0.25);
\draw (0.0,0.375,0.0) -- (0.0,0.5,0.0);
\draw (0.0,0.375,0.0) -- (0.0,0.25,0.0);
\draw (0.0,0.375,0.0) -- (0.0,0.25,0.25);
\draw (0.0,0.375,0.0) -- (0.125,0.375,0.0);
\draw (0.125,0.5,0.125) -- (0.0,0.5,0.0);
\draw (0.125,0.5,0.125) -- (0.25,0.5,0.25);
\draw (0.125,0.5,0.125) -- (0.0,0.5,0.25);
\draw (0.125,0.5,0.125) -- (0.125,0.625,0.0);
\draw (0.125,0.5,0.125) -- (0.125,0.375,0.25);
\draw (0.125,0.5,0.125) -- (0.25,0.5,0.0);
\draw (0.125,0.375,0.125) -- (0.0,0.5,0.0);
\draw (0.125,0.375,0.125) -- (0.25,0.25,0.25);
\draw (0.125,0.375,0.125) -- (0.0,0.25,0.25);
\draw (0.125,0.375,0.125) -- (0.125,0.375,0.25);
\draw (0.125,0.375,0.125) -- (0.125,0.375,0.0);
\draw (0.125,0.375,0.125) -- (0.25,0.5,0.0);
\draw (0.0,0.625,0.25) -- (0.0,0.5,0.5);
\draw (0.0,0.625,0.25) -- (0.0,0.75,0.0);
\draw (0.0,0.625,0.25) -- (0.0,0.75,0.25);
\draw (0.0,0.625,0.25) -- (0.125,0.625,0.25);
\draw (0.0,0.625,0.25) -- (0.0,0.5,0.25);
\draw (0.125,0.0,0.625) -- (0.0,0.0,0.5);
\draw (0.125,0.0,0.625) -- (0.25,0.0,0.75);
\draw (0.125,0.0,0.625) -- (0.0,0.0,0.75);
\draw (0.125,0.0,0.625) -- (0.125,0.125,0.5);
\draw (0.125,0.0,0.625) -- (0.25,0.0,0.5);
\draw (0.0,0.125,0.5) -- (0.0,0.0,0.5);
\draw (0.0,0.125,0.5) -- (0.0,0.25,0.5);
\draw (0.0,0.125,0.5) -- (0.0,0.0,0.75);
\draw (0.0,0.125,0.5) -- (0.125,0.125,0.5);
\draw (0.0,0.125,0.5) -- (0.0,0.25,0.25);
\draw (0.0,0.375,0.5) -- (0.0,0.5,0.5);
\draw (0.0,0.375,0.5) -- (0.0,0.25,0.5);
\draw (0.0,0.375,0.5) -- (0.0,0.25,0.75);
\draw (0.0,0.375,0.5) -- (0.125,0.375,0.5);
\draw (0.0,0.375,0.5) -- (0.0,0.5,0.25);
\draw (0.125,0.0,0.375) -- (0.0,0.0,0.5);
\draw (0.125,0.0,0.375) -- (0.25,0.0,0.25);
\draw (0.125,0.0,0.375) -- (0.0,0.0,0.25);
\draw (0.125,0.0,0.375) -- (0.125,0.125,0.25);
\draw (0.125,0.0,0.375) -- (0.25,0.0,0.5);
\draw (0.0,0.125,0.25) -- (0.0,0.0,0.5);
\draw (0.0,0.125,0.25) -- (0.0,0.25,0.0);
\draw (0.0,0.125,0.25) -- (0.0,0.0,0.25);
\draw (0.0,0.125,0.25) -- (0.125,0.125,0.25);
\draw (0.0,0.125,0.25) -- (0.0,0.25,0.25);
\draw (0.125,0.5,0.375) -- (0.0,0.5,0.5);
\draw (0.125,0.5,0.375) -- (0.25,0.5,0.25);
\draw (0.125,0.5,0.375) -- (0.0,0.5,0.25);
\draw (0.125,0.5,0.375) -- (0.125,0.625,0.25);
\draw (0.125,0.5,0.375) -- (0.125,0.375,0.5);
\draw (0.125,0.125,0.375) -- (0.0,0.0,0.5);
\draw (0.125,0.125,0.375) -- (0.25,0.25,0.25);
\draw (0.125,0.125,0.375) -- (0.25,0.0,0.5);
\draw (0.125,0.125,0.375) -- (0.125,0.125,0.5);
\draw (0.125,0.125,0.375) -- (0.0,0.25,0.25);
\draw (0.125,0.125,0.375) -- (0.125,0.125,0.25);
\draw (0.125,0.75,0.125) -- (0.25,0.75,0.0);
\draw (0.125,0.75,0.125) -- (0.0,0.75,0.0);
\draw (0.125,0.75,0.125) -- (0.0,0.75,0.25);
\draw (0.125,0.75,0.125) -- (0.125,0.875,0.0);
\draw (0.125,0.75,0.125) -- (0.125,0.625,0.25);
\draw (0.125,0.125,0.625) -- (0.25,0.25,0.5);
\draw (0.125,0.125,0.625) -- (0.25,0.0,0.75);
\draw (0.125,0.125,0.625) -- (0.0,0.0,0.75);
\draw (0.125,0.125,0.625) -- (0.125,0.125,0.75);
\draw (0.125,0.125,0.625) -- (0.0,0.25,0.5);
\draw (0.125,0.125,0.625) -- (0.125,0.125,0.5);
\draw (0.125,0.25,0.625) -- (0.25,0.25,0.5);
\draw (0.125,0.25,0.625) -- (0.0,0.25,0.5);
\draw (0.125,0.25,0.625) -- (0.0,0.25,0.75);
\draw (0.125,0.25,0.625) -- (0.125,0.125,0.75);
\draw (0.125,0.25,0.625) -- (0.125,0.375,0.5);
\draw (0.125,0.125,0.125) -- (0.25,0.25,0.0);
\draw (0.125,0.125,0.125) -- (0.25,0.0,0.25);
\draw (0.125,0.125,0.125) -- (0.0,0.0,0.25);
\draw (0.125,0.125,0.125) -- (0.125,0.125,0.0);
\draw (0.125,0.125,0.125) -- (0.0,0.25,0.0);
\draw (0.125,0.125,0.125) -- (0.125,0.125,0.25);
\draw (0.125,0.625,0.125) -- (0.25,0.75,0.0);
\draw (0.125,0.625,0.125) -- (0.0,0.75,0.0);
\draw (0.125,0.625,0.125) -- (0.0,0.5,0.25);
\draw (0.125,0.625,0.125) -- (0.125,0.625,0.0);
\draw (0.125,0.625,0.125) -- (0.25,0.5,0.25);
\draw (0.125,0.625,0.125) -- (0.125,0.625,0.25);
\draw (0.125,0.25,0.375) -- (0.25,0.25,0.5);
\draw (0.125,0.25,0.375) -- (0.0,0.25,0.5);
\draw (0.125,0.25,0.375) -- (0.0,0.25,0.25);
\draw (0.125,0.25,0.375) -- (0.125,0.375,0.25);
\draw (0.125,0.25,0.375) -- (0.25,0.25,0.25);
\draw (0.125,0.25,0.375) -- (0.125,0.125,0.5);
\draw (0.125,0.375,0.375) -- (0.25,0.25,0.5);
\draw (0.125,0.375,0.375) -- (0.0,0.25,0.5);
\draw (0.125,0.375,0.375) -- (0.0,0.5,0.25);
\draw (0.125,0.375,0.375) -- (0.125,0.375,0.25);
\draw (0.125,0.375,0.375) -- (0.25,0.5,0.25);
\draw (0.125,0.375,0.375) -- (0.125,0.375,0.5);
\draw (0.125,0.25,0.125) -- (0.25,0.25,0.0);
\draw (0.125,0.25,0.125) -- (0.0,0.25,0.0);
\draw (0.125,0.25,0.125) -- (0.0,0.25,0.25);
\draw (0.125,0.25,0.125) -- (0.125,0.375,0.0);
\draw (0.125,0.25,0.125) -- (0.25,0.25,0.25);
\draw (0.125,0.25,0.125) -- (0.125,0.125,0.25);
\draw (0.0,0.375,0.125) -- (0.0,0.25,0.25);
\draw (0.0,0.375,0.125) -- (0.125,0.375,0.25);
\draw (0.0,0.375,0.125) -- (0.0,0.375,0.25);
\draw (0.0,0.375,0.125) -- (0.125,0.375,0.0);
\draw (0.0,0.375,0.125) -- (0.0,0.375,0.0);
\draw (0.0,0.375,0.125) -- (0.125,0.375,0.125);
\draw (0.0,0.625,0.375) -- (0.0,0.5,0.5);
\draw (0.0,0.625,0.375) -- (0.0,0.75,0.25);
\draw (0.0,0.625,0.375) -- (0.125,0.625,0.25);
\draw (0.0,0.625,0.375) -- (0.0,0.625,0.25);
\draw (0.0,0.0,0.625) -- (0.0,0.0,0.5);
\draw (0.0,0.0,0.625) -- (0.0,0.0,0.75);
\draw (0.0,0.0,0.625) -- (0.125,0.125,0.5);
\draw (0.0,0.0,0.625) -- (0.125,0.0,0.625);
\draw (0.0,0.0,0.625) -- (0.0,0.125,0.5);
\draw (0.0,0.375,0.625) -- (0.0,0.5,0.5);
\draw (0.0,0.375,0.625) -- (0.0,0.25,0.75);
\draw (0.0,0.375,0.625) -- (0.125,0.375,0.5);
\draw (0.0,0.375,0.625) -- (0.0,0.375,0.5);
\draw (0.0,0.0,0.375) -- (0.0,0.0,0.5);
\draw (0.0,0.0,0.375) -- (0.0,0.0,0.25);
\draw (0.0,0.0,0.375) -- (0.125,0.125,0.25);
\draw (0.0,0.0,0.375) -- (0.125,0.0,0.375);
\draw (0.0,0.0,0.375) -- (0.0,0.125,0.25);
\draw (0.125,0.5,0.0) -- (0.0,0.5,0.0);
\draw (0.125,0.5,0.0) -- (0.25,0.5,0.0);
\draw (0.125,0.5,0.0) -- (0.125,0.625,0.0);
\draw (0.125,0.5,0.0) -- (0.125,0.5,0.125);
\draw (0.125,0.5,0.0) -- (0.125,0.375,0.25);
\draw (0.125,0.5,0.0) -- (0.125,0.375,0.125);
\draw (0.125,0.5,0.0) -- (0.125,0.375,0.0);
\draw (0.0,0.5,0.375) -- (0.0,0.5,0.5);
\draw (0.0,0.5,0.375) -- (0.0,0.5,0.25);
\draw (0.0,0.5,0.375) -- (0.125,0.625,0.25);
\draw (0.0,0.5,0.375) -- (0.0,0.625,0.25);
\draw (0.0,0.5,0.375) -- (0.125,0.375,0.5);
\draw (0.0,0.5,0.375) -- (0.0,0.375,0.5);
\draw (0.0,0.5,0.375) -- (0.125,0.5,0.375);
\draw (0.125,0.0,0.5) -- (0.0,0.0,0.5);
\draw (0.125,0.0,0.5) -- (0.25,0.0,0.5);
\draw (0.125,0.0,0.5) -- (0.125,0.125,0.5);
\draw (0.125,0.0,0.5) -- (0.125,0.0,0.625);
\draw (0.125,0.0,0.5) -- (0.125,0.125,0.25);
\draw (0.125,0.0,0.5) -- (0.125,0.0,0.375);
\draw (0.125,0.0,0.5) -- (0.125,0.125,0.375);
\draw (0.0,0.125,0.375) -- (0.0,0.0,0.5);
\draw (0.0,0.125,0.375) -- (0.0,0.25,0.25);
\draw (0.0,0.125,0.375) -- (0.125,0.125,0.5);
\draw (0.0,0.125,0.375) -- (0.0,0.125,0.5);
\draw (0.0,0.125,0.375) -- (0.125,0.125,0.25);
\draw (0.0,0.125,0.375) -- (0.0,0.125,0.25);
\draw (0.0,0.125,0.375) -- (0.125,0.125,0.375);
\draw (0.125,0.75,0.0) -- (0.25,0.75,0.0);
\draw (0.125,0.75,0.0) -- (0.0,0.75,0.0);
\draw (0.125,0.75,0.0) -- (0.125,0.875,0.0);
\draw (0.125,0.75,0.0) -- (0.125,0.75,0.125);
\draw (0.125,0.75,0.0) -- (0.125,0.625,0.0);
\draw (0.125,0.75,0.0) -- (0.125,0.625,0.125);
\draw (0.125,0.75,0.0) -- (0.125,0.625,0.25);
\draw (0.125,0.25,0.5) -- (0.25,0.25,0.5);
\draw (0.125,0.25,0.5) -- (0.0,0.25,0.5);
\draw (0.125,0.25,0.5) -- (0.125,0.125,0.75);
\draw (0.125,0.25,0.5) -- (0.125,0.125,0.625);
\draw (0.125,0.25,0.5) -- (0.125,0.25,0.625);
\draw (0.125,0.25,0.5) -- (0.125,0.375,0.25);
\draw (0.125,0.25,0.5) -- (0.125,0.25,0.375);
\draw (0.125,0.25,0.5) -- (0.125,0.375,0.375);
\draw (0.125,0.25,0.5) -- (0.125,0.125,0.5);
\draw (0.125,0.25,0.5) -- (0.125,0.375,0.5);
\draw (0.125,0.25,0.0) -- (0.25,0.25,0.0);
\draw (0.125,0.25,0.0) -- (0.0,0.25,0.0);
\draw (0.125,0.25,0.0) -- (0.125,0.125,0.0);
\draw (0.125,0.25,0.0) -- (0.125,0.125,0.125);
\draw (0.125,0.25,0.0) -- (0.125,0.375,0.0);
\draw (0.125,0.25,0.0) -- (0.125,0.25,0.125);
\draw (0.125,0.25,0.0) -- (0.125,0.125,0.25);
\draw (0.0,0.75,0.125) -- (0.0,0.75,0.0);
\draw (0.0,0.75,0.125) -- (0.0,0.75,0.25);
\draw (0.0,0.75,0.125) -- (0.125,0.875,0.0);
\draw (0.0,0.75,0.125) -- (0.0,0.875,0.0);
\draw (0.0,0.75,0.125) -- (0.125,0.625,0.25);
\draw (0.0,0.75,0.125) -- (0.0,0.625,0.25);
\draw (0.0,0.75,0.125) -- (0.125,0.75,0.125);
\draw (0.125,0.0,0.75) -- (0.25,0.0,0.75);
\draw (0.125,0.0,0.75) -- (0.0,0.0,0.75);
\draw (0.125,0.0,0.75) -- (0.125,0.125,0.75);
\draw (0.125,0.0,0.75) -- (0.125,0.0,0.875);
\draw (0.125,0.0,0.75) -- (0.125,0.125,0.5);
\draw (0.125,0.0,0.75) -- (0.125,0.0,0.625);
\draw (0.125,0.0,0.75) -- (0.125,0.125,0.625);
\draw (0.0,0.125,0.625) -- (0.0,0.25,0.5);
\draw (0.0,0.125,0.625) -- (0.0,0.0,0.75);
\draw (0.0,0.125,0.625) -- (0.125,0.125,0.75);
\draw (0.0,0.125,0.625) -- (0.0,0.125,0.75);
\draw (0.0,0.125,0.625) -- (0.125,0.125,0.5);
\draw (0.0,0.125,0.625) -- (0.0,0.125,0.5);
\draw (0.0,0.125,0.625) -- (0.125,0.125,0.625);
\draw (0.0,0.25,0.625) -- (0.0,0.25,0.5);
\draw (0.0,0.25,0.625) -- (0.0,0.25,0.75);
\draw (0.0,0.25,0.625) -- (0.125,0.125,0.75);
\draw (0.0,0.25,0.625) -- (0.0,0.125,0.75);
\draw (0.0,0.25,0.625) -- (0.125,0.375,0.5);
\draw (0.0,0.25,0.625) -- (0.0,0.375,0.5);
\draw (0.0,0.25,0.625) -- (0.125,0.25,0.625);
\draw (0.125,0.0,0.25) -- (0.25,0.0,0.25);
\draw (0.125,0.0,0.25) -- (0.0,0.0,0.25);
\draw (0.125,0.0,0.25) -- (0.125,0.125,0.0);
\draw (0.125,0.0,0.25) -- (0.125,0.0,0.125);
\draw (0.125,0.0,0.25) -- (0.125,0.125,0.25);
\draw (0.125,0.0,0.25) -- (0.125,0.0,0.375);
\draw (0.125,0.0,0.25) -- (0.125,0.125,0.125);
\draw (0.0,0.125,0.125) -- (0.0,0.25,0.0);
\draw (0.0,0.125,0.125) -- (0.0,0.0,0.25);
\draw (0.0,0.125,0.125) -- (0.125,0.125,0.0);
\draw (0.0,0.125,0.125) -- (0.0,0.125,0.0);
\draw (0.0,0.125,0.125) -- (0.125,0.125,0.25);
\draw (0.0,0.125,0.125) -- (0.0,0.125,0.25);
\draw (0.0,0.125,0.125) -- (0.125,0.125,0.125);
\draw (0.0,0.625,0.125) -- (0.0,0.75,0.0);
\draw (0.0,0.625,0.125) -- (0.0,0.5,0.25);
\draw (0.0,0.625,0.125) -- (0.125,0.625,0.0);
\draw (0.0,0.625,0.125) -- (0.0,0.625,0.0);
\draw (0.0,0.625,0.125) -- (0.125,0.625,0.25);
\draw (0.0,0.625,0.125) -- (0.0,0.625,0.25);
\draw (0.0,0.625,0.125) -- (0.125,0.625,0.125);
\draw (0.0,0.25,0.375) -- (0.0,0.25,0.5);
\draw (0.0,0.25,0.375) -- (0.0,0.25,0.25);
\draw (0.0,0.25,0.375) -- (0.125,0.375,0.25);
\draw (0.0,0.25,0.375) -- (0.0,0.375,0.25);
\draw (0.0,0.25,0.375) -- (0.125,0.125,0.5);
\draw (0.0,0.25,0.375) -- (0.0,0.125,0.5);
\draw (0.0,0.25,0.375) -- (0.125,0.25,0.375);
\draw (0.0,0.375,0.375) -- (0.0,0.25,0.5);
\draw (0.0,0.375,0.375) -- (0.0,0.5,0.25);
\draw (0.0,0.375,0.375) -- (0.125,0.375,0.25);
\draw (0.0,0.375,0.375) -- (0.0,0.375,0.25);
\draw (0.0,0.375,0.375) -- (0.125,0.375,0.5);
\draw (0.0,0.375,0.375) -- (0.0,0.375,0.5);
\draw (0.0,0.375,0.375) -- (0.125,0.375,0.375);
\draw (0.0,0.25,0.125) -- (0.0,0.25,0.0);
\draw (0.0,0.25,0.125) -- (0.0,0.25,0.25);
\draw (0.0,0.25,0.125) -- (0.125,0.375,0.0);
\draw (0.0,0.25,0.125) -- (0.0,0.375,0.0);
\draw (0.0,0.25,0.125) -- (0.125,0.125,0.25);
\draw (0.0,0.25,0.125) -- (0.0,0.125,0.25);
\draw (0.0,0.25,0.125) -- (0.125,0.25,0.125);
\draw (0.125,0.5,0.25) -- (0.25,0.5,0.25);
\draw (0.125,0.5,0.25) -- (0.0,0.5,0.25);
\draw (0.125,0.5,0.25) -- (0.125,0.625,0.0);
\draw (0.125,0.5,0.25) -- (0.125,0.5,0.125);
\draw (0.125,0.5,0.25) -- (0.125,0.375,0.25);
\draw (0.125,0.5,0.25) -- (0.125,0.625,0.25);
\draw (0.125,0.5,0.25) -- (0.125,0.5,0.375);
\draw (0.125,0.5,0.25) -- (0.125,0.375,0.5);
\draw (0.125,0.5,0.25) -- (0.125,0.625,0.125);
\draw (0.125,0.5,0.25) -- (0.125,0.375,0.375);
\draw (0.125,0.25,0.25) -- (0.25,0.25,0.25);
\draw (0.125,0.25,0.25) -- (0.0,0.25,0.25);
\draw (0.125,0.25,0.25) -- (0.125,0.375,0.25);
\draw (0.125,0.25,0.25) -- (0.125,0.375,0.125);
\draw (0.125,0.25,0.25) -- (0.125,0.375,0.0);
\draw (0.125,0.25,0.25) -- (0.125,0.125,0.5);
\draw (0.125,0.25,0.25) -- (0.125,0.125,0.375);
\draw (0.125,0.25,0.25) -- (0.125,0.125,0.25);
\draw (0.125,0.25,0.25) -- (0.125,0.25,0.375);
\draw (0.125,0.25,0.25) -- (0.125,0.25,0.125);
\node[circle,fill=black,inner sep=0.6pt] (ell) at (1,0,0) {\,  };
\node[below] (ell2) at (1,0,0) {$v$};

\end{tikzpicture}

%% file: PictureTex/Layer2d_genA.tex
\begin{tikzpicture}[xscale = 2.7, yscale = 1.5]
	\draw[fill = gray!15] (0.0,0.0) -- (2.0,0.0)--  (2.0,2.0) -- cycle;
	\fill[pattern = north east lines, pattern color = gray] (0.0,0.0) -- (2.0,0.0)--  (2.0,2.0) -- cycle;
	\draw[fill = gray!15] (0.0,0.0) -- (1.5,0.0)--  (1.5,1.5) -- cycle;
	\draw[fill = white] (0.0,0.0) -- (1.0,0.0)--  (1.0,1.0) -- cycle;
			
\draw (1.0,0.0) -- (0,0);
\draw (1.0,1.0) -- (0,0);
\draw (2.0,1.5) -- (2,2);
\draw (1.5,1.5) -- (2,2);
\draw (1.5,1.5) -- (1.0,1.0);
\draw (1.0,0.0) -- (1.0,1.0);
\draw (2.0,1.5) -- (2.0,1.0);
\draw (2.0,1.5) -- (1.5,1.5);
\draw (1.5,0.0) -- (2,0);
\draw (1.5,1.5) -- (2.0,1.0);
\draw (1.5,0.5) -- (2,0);
\draw (1.5,0.5) -- (1.0,1.0);
\draw (1.5,0.5) -- (1.0,0.0);
\draw (1.5,0.5) -- (2.0,1.0);
\draw (1.5,0.0) -- (1.0,0.0);
\draw (1.5,0.0) -- (1.5,0.5);
\draw (2.0,0.5) -- (2,0);
\draw (2.0,0.5) -- (2.0,1.0);
\draw (2.0,0.5) -- (1.5,0.5);
\draw (1.5,1.0) -- (1.0,1.0);
\draw (1.5,1.0) -- (2.0,1.0);
\draw (1.5,1.0) -- (1.5,1.5);
\draw (1.5,1.0) -- (1.5,0.5);
\draw (0.02,0.14)	node {\scriptsize -2};
\draw (2.05,0.05)	node {\scriptsize 0};
\draw (2.05,2.1)	node {\scriptsize -1};
\draw (1.04,1.14)	node {\scriptsize 1};
\draw (1.04,0.14)	node {\scriptsize 2};
\draw (2.05,1.06)	node {\scriptsize 2};
\draw (1.54,1.64)	node {\scriptsize 3};
\draw (1.54,0.64)	node {\scriptsize 3};
\draw (2.05,1.56)	node {\scriptsize 4};
\draw (1.54,0.14)	node {\scriptsize 4};
\draw (2.05,0.56)	node {\scriptsize 4};
\draw (1.54,1.14)	node {\scriptsize 4};
\draw[dotted,thick] (0.0,0.0) -- (0.0,-0.3);
\draw[dotted,thick] (1.0,0.0) -- (1.0,-0.3);
\draw[dotted,thick] (1.5,0.0) -- (1.5,-0.3);
\draw[dotted,thick] (2,0.0) -- (2,-0.3);
\draw (.5,0) node[below] {\scriptsize $\lambda_m^1$};
\draw (1.25,0) node[below] {\scriptsize $\lambda_m^2$};
\draw (1.75,0) node[below] {\scriptsize $\gamma_m^2$};
\node[circle,fill=black,inner sep=0.6pt] (ell) at (0,0) {\,  };
\node[left] (ell2) at (0,0) {$v$};
\end{tikzpicture}

%% file: PictureTex/Layer2d_genB.tex
\begin{tikzpicture}[xscale = 2.7, yscale = 1.5]
	\draw[fill = gray!30] (0.0,0.0) -- (2.0,0.0)--  (2.0,2.0) -- cycle;
	\fill[pattern = north east lines, pattern color = gray] (0.0,0.0) -- (2.0,0.0)--  (2.0,2.0) -- cycle;

	\draw[fill = gray!40] (0.0,0.0) -- (1.75,0.0)--  (1.75,1.75) -- cycle;
	\draw[fill = gray!15] (0.0,0.0) -- (1.5,0.0)--  (1.5,1.5) -- cycle;
	\draw[fill = white] (0.0,0.0) -- (1.0,0.0)--  (1.0,1.0) -- cycle;
	
\draw (1.0,0.0) -- (0,0);
\draw (1.0,1.0) -- (0,0);
\draw (2.0,1.75) -- (2,2);
\draw (1.75,1.75) -- (2,2);
\draw (1.5,1.5) -- (1.0,1.0);
\draw (1.0,0.0) -- (1.0,1.0);
\draw (2.0,1.75) -- (2.0,1.5);
\draw (2.0,1.75) -- (1.75,1.75);
\draw (1.75,0.0) -- (2,0);
\draw (1.75,1.75) -- (1.5,1.5);
\draw (1.75,1.75) -- (2.0,1.5);
\draw (1.5,0.5) -- (1.0,1.0);
\draw (1.5,0.5) -- (1.0,0.0);
\draw (1.75,0.25) -- (2,0);
\draw (1.5,0.0) -- (1.0,0.0);
\draw (1.5,0.0) -- (1.5,0.5);
\draw (1.75,0.0) -- (1.5,0.0);
\draw (1.75,0.0) -- (1.75,0.25);
\draw (2.0,0.25) -- (2,0);
\draw (1.5,1.0) -- (1.0,1.0);
\draw (2.0,0.25) -- (2.0,0.5);
\draw (1.5,1.0) -- (1.5,1.5);
\draw (1.5,1.0) -- (1.5,0.5);
\draw (1.75,0.25) -- (1.5,0.5);
\draw (1.75,0.25) -- (1.5,0.0);
\draw (1.75,0.25) -- (2.0,0.5);
\draw (1.75,1.25) -- (2.0,1.0);
\draw (1.75,1.25) -- (1.5,1.5);
\draw (1.75,1.25) -- (2.0,1.5);
\draw (1.75,1.25) -- (1.5,1.0);
\draw (1.75,0.75) -- (2.0,1.0);
\draw (1.75,0.75) -- (1.5,0.5);
\draw (1.75,0.75) -- (2.0,0.5);
\draw (1.75,0.75) -- (1.5,1.0);
\draw (2.0,0.25) -- (1.75,0.25);
\draw (2.0,1.25) -- (2.0,1.0);
\draw (2.0,1.25) -- (2.0,1.5);
\draw (2.0,1.25) -- (1.75,1.25);
\draw (2.0,0.75) -- (2.0,1.0);
\draw (2.0,0.75) -- (2.0,0.5);
\draw (2.0,0.75) -- (1.75,0.75);
\draw (1.75,1.0) -- (2.0,1.0);
\draw (1.75,1.0) -- (1.5,1.0);
\draw (1.75,1.0) -- (1.75,1.25);
\draw (1.75,1.0) -- (1.75,0.75);
\draw (1.75,1.5) -- (1.5,1.5);
\draw (1.75,1.5) -- (2.0,1.5);
\draw (1.75,1.5) -- (1.75,1.75);
\draw (1.75,1.5) -- (1.75,1.25);
\draw (1.75,0.5) -- (1.5,0.5);
\draw (1.75,0.5) -- (2.0,0.5);
\draw (1.75,0.5) -- (1.75,0.25);
\draw (1.75,0.5) -- (1.75,0.75);
\draw (0.02,0.14)	node {\scriptsize -2};
\draw (2.05,0.05)	node {\scriptsize 0};
\draw (2.05,2.1)	node {\scriptsize -1};
\draw (1.04,1.14)	node {\scriptsize 1};
\draw (1.04,0.14)	node {\scriptsize 2};
\draw (2.05,1.06)	node {\scriptsize 2};
\draw (1.54,1.64)	node {\scriptsize 3};
\draw (1.54,0.64)	node {\scriptsize 3};
\draw (2.05,1.56)	node {\scriptsize 4};
\draw (1.54,0.14)	node {\scriptsize 4};
\draw (2.05,0.56)	node {\scriptsize 4};
\draw (1.54,1.14)	node {\scriptsize 4};
\draw (1.79,1.87)	node {\scriptsize 5};
\draw (1.79,0.37)	node {\scriptsize 5};
\draw (1.79,1.37)	node {\scriptsize 5};
\draw (1.79,0.87)	node {\scriptsize 5};
\draw (2.05,1.81)	node {\scriptsize 6};
\draw (1.79,0.1)	node {\scriptsize 6};
\draw (2.05,0.31)	node {\scriptsize 6};
\draw (2.05,1.31)	node {\scriptsize 6};
\draw (2.05,0.81)	node {\scriptsize 6};
\draw (1.79,1.08)	node {\scriptsize 6};
\draw (1.79,1.58)	node {\scriptsize 6};
\draw (1.79,0.58)	node {\scriptsize 6};
\draw[dotted,thick] (0.0,0.0) -- (0.0,-0.3);
\draw[dotted,thick] (1.0,0.0) -- (1.0,-0.3);
\draw[dotted,thick] (1.5,0.0) -- (1.5,-0.3);
\draw[dotted,thick] (1.75,0.0) -- (1.75,-0.3);
\draw[dotted,thick] (2,0.0) -- (2,-0.3);
\draw (.5,0) node[below] {\scriptsize $\lambda_m^1$};
\draw (1.25,0) node[below] {\scriptsize $\lambda_m^2$};
\draw (1.625,0) node[below] {\scriptsize $\lambda_m^3$};
\draw (1.875,0) node[below] {\scriptsize $\ \, \gamma_m^3$};
\node[circle,fill=black,inner sep=0.6pt] (ell) at (0,0) {\,  };
\node[left] (ell2) at (0,0) {$v$};
\end{tikzpicture}